\documentclass[reqno,10pt]{amsart}
\usepackage{amsmath}
\usepackage{amssymb}
\usepackage{amsfonts}
\usepackage{graphicx}
\usepackage[abs]{overpic}
\usepackage{caption}
\usepackage{wrapfig}
\usepackage{float}
\usepackage{graphicx}
\usepackage{mathrsfs}
\usepackage{accents}
\usepackage{amsthm}
\usepackage{hyperref}

\usepackage{tikz-cd}
\usepackage[colorinlistoftodos]{todonotes}

\newtheorem{theorem}{Theorem}[section]
\newtheorem{lemma}[theorem]{Lemma}
\newtheorem{proposition}[theorem]{Proposition}

\newtheorem{definition}[theorem]{Definition}

\newtheorem{rem}[theorem]{Remark}

\newcommand{\Proof}{\par\noindent{\em Proof. }}
\newcommand{\eop}{\nopagebreak\hspace*{\fill}$\Box$\smallskip}

\newcommand{\N}{\Bbb N}

\newcommand{\R}{\Bbb R}
\newcommand{\C}{\Bbb C}

\newcommand{\vstress}{{S}}

\def\Id{\mathbf{Id}}
\def\id{\mathbf{id}}

\def\eps{\varepsilon}

\def\dist{\operatorname{dist}}

\def\XXint#1#2#3{{\setbox0=\hbox{$#1{#2#3}{\int}$}
     \vcenter{\hbox{$#2#3$}}\kern-.5\wd0}}

\newcommand\thickbar[1]{\accentset{\rule{.6em}{.8pt}}{#1}}

\usepackage{color}

\newcommand{\cm}{\color{black}}

\newcommand{\NNN}{\color{black}} 
\newcommand{\BBB}{\color{black}} 
\newcommand{\EEE}{\color{black}}

 \newcommand{\PPP}{\color{black}} 
\newcommand{\second}{\color{black}}
\newcommand{\first}{\color{black}}

\numberwithin{equation}{section}

\title[Derivation of von K\'{a}rm\'{a}n plate theory in the framework of viscoelasticity]{Derivation of von K\'{a}rm\'{a}n plate theory in the framework of  three-dimensional  viscoelasticity}

\author{Manuel Friedrich}
\author{Martin Kru\v{z}\'ik}


\address[Manuel Friedrich]{Applied Mathematics,  
Universit\"{a}t M\"{u}nster, Einsteinstr. 62, D-48149 M\"{u}nster, Germany}
\email{manuel.friedrich@uni-muenster.de}

\address[Martin Kru\v{z}\'ik]{Czech Academy of Sciences, Institute of Information Theory and Automation,
Pod vod\'arenskou v\v{e}\v{z}\'i 4, CZ-182 08 Praha 8, Czechia
(corresponding address) \& Faculty of Civil Engineering, Czech Technical University, Th\'akurova 7, CZ-166 29 Praha 6, Czechia}
\email{kruzik@utia.cas.cz}

\begin{document}

\maketitle

\begin{center}
\today
\end{center}
\smallskip

\begin{abstract}
We apply  a quasistatic nonlinear  model  for nonsimple viscoelastic materials at a finite-strain setting in the Kelvin's-Voigt's rheology  to derive a viscoelastic plate model of von K\'{a}rm\'{a}n type. 
We start from \second  time-discrete \rm  \EEE solutions to a model of three-dimensional  viscoelasticity considered in \cite{MFMK} where the   viscosity stress tensor   complies with the principle of time-continuous frame-indifference.
Combining the derivation of nonlinear plate theory by Friesecke, James and M\"{u}ller \BBB \cite{FrieseckeJamesMueller:02, hierarchy}, \EEE and  the abstract theory of gradient flows in metric spaces by Sandier and Serfaty \cite{S1}  we perform a dimension-reduction from 3D to 2D and identify weak solutions of viscoelastic form of von K\'{a}rm\'{a}n plates.
\end{abstract}

\section{Introduction}

Dimension-reduction problems play a significant  role in nonlinear analysis and numerics because they allow for simpler  computational approaches, still  preserving main features of the bulk system. 
\BBB In this context, it is important \EEE that  a clear relationship between the full three-dimensional problem and its lower-dimensional counterpart  is  made rigorous. \BBB The last decades have witnessed remarkable progress in that direction through the use of variational methods, particularly by  \EEE $\Gamma$-convergence \cite{DalMaso:93} together with \BBB quantitative \EEE rigidity estimates \cite{FrieseckeJamesMueller:02}. \BBB Among the large body of results, we mention here only the rigorous justification of membrane theory \cite{Dret1, Dret2}, bending theory \cite{Casarino, FrieseckeJamesMueller:02, Pantz}, and von K\'{a}rm\'{a}n theory \cite{hierarchy, lecumberry} for plates as variational limits of nonlinear three-dimensional elasticity for vanishing thickness. In particular, we refer to \cite{hierarchy} for the derivation of a hierarchy of different plate models and for a thorough literature review.  

\BBB In the present work, we apply a similar scenario of deriving plate theories  \EEE to problems in nonlinear viscoelasticity: starting from a three-dimensional model of a \BBB nonsimple viscoelastic material at a finite strain setting  in Kelvin's-Voigt's \EEE rheology (i.e., a spring and a damper coupled in parallel), recently treated by the authors in \cite{MFMK}, we derive a model of von K\'{a}rm\'{a}n (vK) viscoelastic plates.

 \BBB In \cite{MFMK}, the existence of weak solutions  for  such a three-dimensional model of nonsimple viscoelastic materials has been established  \EEE  with the help of gradient flows in metric spaces developed in \cite{AGS,S1}.  The notion of \BBB a nonsimple (or second-grade) \EEE material refers to the fact that the elastic energy depends also on the second gradient of the deformation. This concept, first suggested  by Toupin \cite{Toupin:62,Toupin:64}, has proved to be useful in modern mathematical elasticity, see e.g.~\cite{BBMK,BCO,Batra, MielkeRoubicek:16,MielkeRoubicek,Podio}. \BBB We \EEE also refer to \cite{capriz, dunn}  where thermodynamical consistency of such models has been shown.  We point out that this approach seems to be currently unavoidable \EEE in order to obtain the existence of \BBB solutions in the nonlinear \second {\it viscoelastic} \EEE setting, \EEE see  \second  \cite{MFMK, MielkeRoubicek} \rm \EEE and   \cite{MOS} for a general discussion about the interplay between the elastic energy and viscous dissipation.  \BBB  Nevertheless, a main justification is the observation  in  \cite{MFMK} that, in the small strain limit, the problem leads to the standard system  for  linearized viscoelasticity without second gradient.

In the present work, we consider a thin plate of thickness $h$ and pass to the dimension-reduction limit $h \to 0$ in the vK energy regime. We show that this gives rise to effective equations in terms of suitably rescaled in-plane and out-of-plane displacements which  feature membrane and bending terms \emph{both} in the elastic and the viscous stress.   This represents a dissipative counterpart of the purely elastic vK theory which was first formulated more than hundred years ago \cite{VonKarman}.  Besides identifying the correct 2D limiting equations of viscous vK plates, which to the best of our knowledge has not been done in previous literature, the main goals of this contribution are twofold: (1) we show  existence of solutions to the effective 2D system, and (2) we prove \second  rigorously \rm \EEE that these solutions are in a certain sense the limits of solutions to the 3D equations as $h \to 0$, \second  see Theorem \ref{maintheorem2} and Theorem \ref{maintheorem3} for details. \EEE \rm

\second Let us mention that there are previous works on viscoelastic plates \cite{Bock1,Park}, some even including inertial effects  \cite{Bock2, Bock3}. Their starting point, however,  is already a plate model.   Our model, derived rigorously from three-dimensional viscoelasticity, is new and, if viscosity is dropped, it reduces to  the well-known model of elastic plates \cite{hierarchy}. Let us emphasize that for purely elastic models neglecting viscosity various existence results  were obtained for F\"{o}ppl-von K\'{a}rm\'{a}n plates without resorting to a second-grade material, see e.g.~\cite{BK,  hierarchy, lecumberry, LMP,  MPT} . \rm   \EEE  We refer to \cite{MFMKJV} for a numerical study of vK  viscoelastic plates, \first   and mention that \rm in that paper we have also proved the existence of solutions to the viscoelastic vK plate equations   by means of converging  numerical discretizations.   This work, however,  {\it already relies} on   the  plate model derived in the present paper. \EEE   For the derivation of a plate model without viscosity but with inertia we refer to 
\cite{AMM}. \rm

\BBB We now describe our setting in more detail. Without inertia, \EEE  a    nonlinear viscoelastic material in Kelvin's-Voigt's rheology satisfies the system of equations 
\begin{align}\label{eq:viscoel}
-{\rm div}\Big(\partial_FW(\nabla w)   + \partial_{\dot F}R(\nabla w,\partial_t \nabla w)  \Big) =  f \BBB e_3 \EEE \ \ \  \text{ in $  [0,T] \times  \BBB \Omega_h \EEE $.} \end{align}
Here, $[0,T]$ is a process time interval with $T>0$,  \BBB  $\Omega_h := S \times (-\frac{h}{2},\frac{h}{2}) \subset\R^3 $  \EEE is a smooth, \BBB thin, \EEE bounded domain representing the reference configuration,  and   $w:[0,T]\times \Omega_h\to\R^3$ is a deformation mapping  with    deformation gradient $\nabla w$.  Moreover,   $W:\R^{3\times 3}\to  [0,\infty]$ is a stored energy density representing a potential of the first Piola-Kirchhoff stress tensor \BBB $\partial_F W:=\partial W/\partial F$ \EEE and  $F\in\R^{3\times 3}$ is the placeholder of $\nabla w$. \BBB The function \EEE $R:  \R^{3 \times 3} \times \R^{3 \times 3} \to [0,\infty) \EEE $ denotes a (pseudo)potential of dissipative forces,  where $\dot F \in \R^{3 \times 3}$ is the placeholder of $\partial_t \nabla w$. \BBB Finally, \EEE $f: \Omega_h\to \R$  is a volume density of external  forces, which is considered \cm independent of time and the deformation  $y$, \EEE \BBB  and  for simplicity  acting on $\Omega_h$ only in normal direction $e_3$. \EEE  

We assume that $W$ is a frame-indifferent function, i.e., $W(F)=W(QF)$ for $Q\in{\rm SO}(3)$ and $F\in\R^{3\times 3}$. This implies that $W$ depends on the right Cauchy-Green strain tensor  $C:=F^\top F$, see e.g.~\cite{Ciarlet}.   The second term on the left-hand side of \eqref{eq:viscoel} is the stress tensor $\vstress(F,\dot F):= \partial_{\dot F} R(F,\dot F)$ which has its origin in viscous dissipative mechanisms of the material.  We point out that its potential $R$ plays an  analogous role as $W$ in the case of purely elastic, i.e., non-dissipative processes. Naturally, we require that $R(F,\dot F)\ge R(F,0)=0$.  The viscous stress tensor must comply with the time-continuous frame-indifference principle  meaning that $\vstress(F,\dot F)=F\tilde\vstress(C,\dot C)$, where $\tilde\vstress$ is a symmetric matrix-valued function \BBB and  $\dot C$ denotes the  time derivative of the right Cauchy-Green strain tensor $C$. \EEE  This condition constraints 
$R$ so that  \BBB $R(F,\dot F)=\tilde R(C,\dot C)$  for some nonnegative function $\tilde R$, see \cite{Antmann, Antmann:04,MOS}. \first  In the following, we suppose that the  material is homogeneous, i.e., neither the elastic stored energy density nor  the dissipation depend on   material points. \rm \EEE  \second  Moreover, for technical reasons, we will restrict our analysis to the case of zero Poisson's ratio in the out-of-plane direction, see \eqref{eq:Q2}-\eqref{eq:Q22} and Remark \ref{rem: Poisson2} for some details in that direction. Such an assumption, also present in other works (see e.g.~\cite{BK}), simplifies the analysis. \rm \EEE

\BBB Following the study in \cite{MFMK}, \EEE we consider a version of \eqref{eq:viscoel} for \BBB second-grade  \EEE materials where the elastic stored energy density (and the  first Piola-Kirchhoff stress tensor, too) depends also on the second gradient  of $w$.  In this case, we get 

\begin{align}\label{eq:viscoel-nonsimple}
-{\rm div}\Big( \partial_F W(\nabla w) + \eps\mathcal{L}_{P}(\nabla^2 w)  + \partial_{\dot{F}}R(\nabla w,\partial_t \nabla w)  \Big) =  f \BBB e_3 \EEE \ \ \  \text{ in $ [0,T] \times \EEE \Omega_h$,}
 \end{align}
where $\eps>0$ is small and  $\mathcal{L}_{P}$ is a  first \EEE order differential operator which corresponds to an additional term $\int_{\Omega_h} P(\nabla^2 w)$  in the stored elastic energy,  \BBB associated to a convex and frame-indifferent density $P$. \EEE (We refer to  \eqref{LP-def} for more details.)  As already mentioned before, this  idea by Toupin  \cite{Toupin:62,Toupin:64} has proved to be useful in mathematical elasticity     because it brings additional compactness to the problem.  \BBB For example, concerning existence theory for  \BBB second-grade \EEE materials, \EEE no convexity properties of $W$ are needed, in particular, we do not have to assume that $W$ is polyconvex \cite{Ball:77,Ciarlet}.  Moreover, it is shown in \cite{HealeyKroemer:09} that, if $W$ satisfies suitable and physically relevant  growth conditions (as $W(F)\to\infty$ if ${\rm det}\, F\to 0$),   then every minimizer of the elastic energy is a weak solution to the corresponding Euler-Lagrange equations.

\BBB

In \cite{hierarchy}, it has been shown that for forces scaling like $\sim h^3$, which corresponds to an energy per thickness of $\sim h^4$, the nonlinear elastic energy can be related rigorously by $\Gamma$-convergence as $h\to 0$  to the so-called \emph{von K\'{a}rm\'{a}n functional}. This functional is given in terms of rescaled in-plane displacements $u$ and out-of-plane  displacements $v$. The  corresponding Euler-Lagrange equations take the form   \NNN
\begin{align*}
{\rm div}\big(\C_W\big( e(u)  + \tfrac{1}{2} \nabla v \otimes \nabla v  \big) \big) & = 0, \\
-{\rm div}\Big( \C_W\Big( e(u)  + \frac{1}{2} \nabla v \otimes \nabla v      \Big) \nabla v \Big)  + \frac{1}{12} {\rm div} \, {\rm div}\big( \C_W \nabla^2 v \big)  & =   f   \ \ \  \text{in }   S,
\end{align*}
\EEE where $e(u):=(\nabla u+(\nabla u)^\top)/2$ denotes the linear strain tensor, and  $\C_W$ is the tensor of elastic constants, derived suitably from $W$ (see \eqref{eq:Q2}--\eqref{eq: order4} below for details). The first equation corresponds to the \emph{membrane strain},  which  was used already earlier in F\"oppl's work \cite{Foppl} and leads to a nonlinearity in the vK equations. The second equation \NNN includes also a \EEE \emph{bending contribution}.  In the present context, we will see that  the  passage from the nonlinear elastic energy to the vK functional   by $\Gamma$-convergence remains true if the nonlinear energy is enhanced by a second gradient term $\eps \int_{\Omega_h} P(\nabla^2 w)$ for certain scalings of $\eps$, see Theorem \ref{th: Gamma}. 

 In the frame of viscoelastic materials, we address the relation between the nonlinear equations \eqref{eq:viscoel-nonsimple} to the following equations for viscoelastic  vK plates:  \NNN
\begin{align}\label{eq: equation-simp-intro2}
\begin{cases} 
0 = {\rm div}\Big(\C_W\big( e(u)  + \frac{1}{2} \nabla v \otimes \nabla v  \big) +  \C_R \big( e(\partial_t u) +   \nabla \partial_t v    \otimes  \nabla v \big) \Big)  , &\vspace{0.1cm} \\ 
f = -{\rm div}\Big(\Big(\C_W\big( e(u)  + \frac{1}{2} \nabla v \otimes \nabla v  \big)  +  \C_R \big( e(\partial_t u) +   \nabla \partial_t v   \otimes \nabla v \big) \Big) \nabla v \Big)  &\vspace{0.1cm}\\
 \quad \quad  + \tfrac{1}{12}  {\rm div} \, {\rm div}\Big( \C_W \nabla^2 v + \C_R \nabla^2 \partial_t v \Big)     &\ \ \ \    \text{in } [0,\infty) \times S,
\end{cases}
\end{align} \EEE
where  $\C_R$ is the tensor of viscosity coefficients which is derived from the dissipation potential $R$. More precisely, we prove  existence of solutions to \eqref{eq: equation-simp-intro2} and make the dimension reduction rigorous, i.e., we show that solutions to  \eqref{eq:viscoel-nonsimple} converge to solutions of \eqref{eq: equation-simp-intro2} in a specific sense. \second  The solutions have to be understood in a weak sense, see \eqref{eq: weak equation} for the exact definition. \rm  \EEE     We point out that the same relation is expected to hold also for the original problem of simple materials \eqref{eq:viscoel} (i.e., only the first gradient of $w$ is considered), but a proof seems  unreachable (or at least rather difficult) at the moment. \second We emphasize  that a nonsimple material model is used here because  viscous phenomena are  considered, too. In fact, dissipation due to viscosity  leads to the loss of weak lower semicontinuity in  semidiscretized variational problems, see \cite{MOS} for a detailed discussion.  \EEE

\EEE

Our general strategy is to treat the system of quasistatic viscoelasticity in the abstract setting of metric gradient flows \cite{AGS}, \BBB where  the underlying metric is given by a \emph{dissipation distance} suitably related to the potential $R$ (see \eqref{intro:R} below). \EEE To our best knowledge, this was formulated  for the first time in  \cite{MOS} for simple materials.   \BBB Similar to \cite{MFMK},  our starting point is the existence of  time-discrete solutions to the three-dimensional equations \eqref{eq:viscoel-nonsimple}. Here, the second gradient term  allows to obtain an existence result without \cm  polyconvexity conditions \cite{Ball:77}  \EEE for the dissipation  which seems to be incompatible with frame indifference. Existence of  solutions to the equations \eqref{eq: equation-simp-intro2} of viscous vK  plates is guaranteed by identifying them as limits of solutions to the nonlinear three-dimensional equations  \eqref{eq:viscoel-nonsimple}. In this context, we follow  \EEE  the abstract framework of sequences of metric gradient flows, developed in  \cite{Ortner, S1,S2}. \BBB In using this theory, \EEE  the challenge lies in proving that the additional conditions needed to ensure convergence of gradient flows are satisfied.
\BBB

More specifically, to use the  abstract convergence result, lower semicontinuity of  (i) the energies, (ii) the metrics, and (iii) the local slopes is needed. (i) The estimate for the energies  essentially follows from \cite{hierarchy} where we show that
it still holds for nonsimple materials if the contribution of the second gradient in terms of $\eps$ (see  \eqref{eq:viscoel-nonsimple}) is chosen sufficiently small, cf.\  Theorem \ref{th: Gamma}.  (ii) The lower semicontinuity of the metrics can be established
in a very similar fashion. (iii) The lower semicontinuity of the local slopes, however, is very technical and the core of our argument. We briefly explain the main idea. The local slope  of an energy $\phi$ in a metric space with metric $\mathcal{D}$ is defined by
$$|\partial \phi|_{\mathcal{D}}(y): = \limsup_{z \to y} \frac{(\phi(y) - \phi(z))^+}{\mathcal{D}(y,z)}.$$
Consider a sequence of deformations $(w^h)_h$ of thin plates converging to a limit $(u,v)$. The first step in the proof is to show that the local slope in the limiting setting at some configuration $y=(u,v)$ can be determined by considering only variations of the form $z=(u_s, v_s) = (u,v) + s(\tilde{u},\tilde{v})$ for $s>0$ small. This follows from a specific representation, see Lemma~\ref{lemma: slopes}, which is based on some generalized convexity properties. (Recall, however, that the vK model is actually nonconvex.) Then the crucial step is to choose sequences $(w^h_s)_{h,s}$ where $w^h_s \to w^h$ represents a competitor sequence for the local slope in the 3D setting. For each $s>0$, $(w^h_s)_h$ has to be constructed  as a \emph{mutual recovery sequence} of $(u_s, v_s)$, i.e., for both the elastic energies and the dissipations. In this context, the rate of convergence needs to be \emph{linear in $s$} as $h\to 0$. The realization is in fact quite technical and the details are contained in Theorem \ref{theorem: lsc-slope}. An important step in this analysis is to understand the strain difference of the configurations $w^h$ and $w^h_{s}$, see Lemma \ref{lemma: strong convergence}.

 Let us mention that a derivation of a nonlinear bending theory (see \cite{FrieseckeJamesMueller:02}) in the setting of viscoelasticity seems to be even more involved and remains an open problem: the fact that deformations are generically not near a single rigid motion does not comply with  the model investigated in \cite{MFMK}. Even more severely, the characterization of the local slope in the limiting two-dimensional setting appears to be very difficult due to the nonlinear isometry constraint.

\EEE

The plan of the paper is as follows.  In Section~\ref{sec:Model}, we introduce the nonlinear three-dimensional viscoelastic system  and its two-dimensional \BBB vK limiting version \EEE in more detail. Then we state our main results.  In particular, Proposition~\ref{maintheorem1} shows existence of time-discrete solutions to the three-dimensional problem \eqref{eq:viscoel-nonsimple}. In  Theorem~\ref{maintheorem2} we present an existence result for solutions to the vK plate equations  \eqref{eq: equation-simp-intro2} \second in the ``time-continuous setting'', \EEE which is based on identifying these solutions with  so-called  \emph{curves of maximal slope} \cite{DGMT}.     Finally, Theorem~\ref{maintheorem3} shows the relationship between the two systems which relies on an  abstract convergence result for curves of maximal slope  and \BBB  their approximation  via the minimizing movement scheme.

 Section~\ref{sec3} is devoted  to definitions of generalized minimizing movements and curves of maximal slope.  \BBB  Here, we also recall results about sequences of curves of maximal slope and their approximation by minimizing movements.    Section~\ref{sec:energy-dissipation} is devoted to properties of the elastic energies and the dissipation distances. \BBB In particular, we show that the dissipation distances give rise to complete metric spaces, and derive some generalized convexity properties in the limiting 2D setting, as well as a representation of the local slope.  \EEE  Section~\ref{Sec:relation2D3D} discusses the relation between the three- and two-dimensional setting, compactness properties, and $\Gamma$-convergence results. \BBB Moreover, here we prove the fundamental lower semicontinuity properties for local slopes. \EEE Finally, proofs of our main results can be found in Section~\ref{sec results}.

\section{The model and main results}\label{sec:Model}

\BBB \subsection{Notation} \EEE
 In what follows, we use standard notation for Lebesgue spaces,  $L^p(\Omega)$, which are measurable maps on $\Omega\subset\R^d$, \BBB $d=2,3$, \EEE integrable with the $p$-th power (if $1\le p<+\infty$) or essentially bounded (if $p=+\infty$).    Sobolev spaces, i.e., $W^{k,p}(\Omega)$ denote the linear spaces of  maps  which, together with their weak derivatives up to the order $k\in\N$, belong to $L^p(\Omega)$.  Further,  $W^{k,p}_0(\Omega)$ contains maps from $W^{k,p}(\Omega)$ having zero boundary conditions (in the sense of traces).  To emphasize the target space $\R^k$, $k=1,2,3$, we write $L^p(\Omega;\R^k)$. If $k=1$, we write $L^p(\Omega)$ as usual.   We refer to \cite{AdamsFournier:05} for more details on Sobolev spaces and their duals.  We  also denote the components of vector  functions $y$ by $y_1$, $y_2$, and  $y_3$, and so on.

If $A\in\R^{d\times d\times d\times d}$ and $e\in\R^{d\times d}$ then $Ae\in\R^{d\times d}$ is such that for $i,j\in\{1,\ldots, d\}$ we define  $(Ae)_{ij}:=A_{ijkl}e_{kl}$ where we use Einstein's summation convention. An analogous  convention is used in similar  occasions, in the sequel.   \BBB By $\Id \subset \R^{3 \times 3}$ we denote the identity matrix. We often drop $dx$ at the end of integrals if the integration variable is clear from the context. \EEE   Finally, at many spots, we closely follow notation introduced in \cite{AGS} \BBB and \cite{hierarchy} \EEE to ease readability of our work.

\subsection{The setting}
We first introduce a 3D setting following the setup in \cite{hierarchy}. We consider  a right-handed  orthonormal system $\{e_1,e_2,e_3\}$ \PPP and \EEE  $S \subset \R^2$ open, bounded with Lipschitz boundary, in the span of $e_1$ and $e_2$. \PPP  Let $h>0$ small. We consider  \EEE  \emph{deformations} $w: S \times (-\frac{h}{2},\frac{h}{2}) \to \R^3$. It is convenient to work in a fixed domain $\Omega = S \times I$ with $I:= (-\frac{1}{2},\frac{1}{2})$ and to rescale deformations according to $y(x) = w(x',hx_3)$ such that $y: \Omega \to \R^3$, where we use the abbreviation $x' = (x_1,x_2)$. We also introduce the notation $\nabla' y = y_{,1} \otimes e_1  +y_{,2} \otimes e_2$  for the in-plane gradient, and the scaled gradient
\begin{align}\label{eq: scaled1} 
\nabla_h y := \Big(\nabla' y, \frac{1}{h} y_{,3} \Big) = \nabla w.
\end{align}
Moreover, we define the scaled second gradient by 
\begin{align}\label{eq: scaled2} 
(\nabla^2_h y)_{ijk} := h^{-\delta_{3j} - \delta_{3k}} (\nabla^2 y)_{ijk} = (\nabla^2 w)_{ijk} =\partial^2_{jk} w_i \ \  \text{for $i,j,k \in \lbrace 1,2,3\rbrace$}, 
\end{align}
where $\delta_{3j},\delta_{3k}$ denotes the Kronecker delta. 

\smallskip

\textbf{Stored elastic energy \BBB density  \EEE and body forces:} 
 We assume that $W: \R^{3 \times 3} \to [0,\infty]$ is a single well, frame-indifferent stored energy \BBB density \EEE with the usual assumptions in nonlinear elasticity. We   suppose  that there exists $c>0$ such that
\begin{align}\label{assumptions-W}
\begin{split}
(i)& \ \ W \text{ continuous and \BBB $C^3$ \EEE in a neighborhood of $SO(3)$},\\
(ii)& \ \ \text{frame indifference: } W(QF) = W(F) \text{ for all } F \in \R^{3 \times 3}, Q \in SO(3),\\
(iii)& \ \ W(F) \ge c\dist^2(F,SO(3)), \  W(F) = 0 \text{ iff } F \in SO(3),
\end{split}
\end{align}
where $SO(3) = \lbrace Q\in \R^{3 \times 3}: Q^\top Q = \Id, \, \det Q=1 \rbrace$. \BBB Moreover, for $p>3$, \EEE let $P: \R^{3\times 3 \times 3} \to [0,\infty]$ be a higher order perturbation satisfying 
\begin{align}\label{assumptions-P}
\begin{split}
(i)& \ \ \text{frame indifference: } P(QZ) = P(Z) \text{ for all } Z \in \R^{3 \times 3 \times 3}, Q \in SO(3),\\
(ii)& \ \ \text{$P$ is convex and $C^1$},\\
(iii)& \ \ \text{growth condition: For all $Z \in \R^{3 \times 3 \times 3}$ we have } \\&   \ \ \ \ \ \    c_1 |Z|^p \le P(Z) \le c_2 |Z|^p, \ \ \ \ \ \ |\partial_{Z} P(Z)|  \le c_2 |Z|^{p-1} 
\end{split}
\end{align}
for $0<c_1<c_2$. Finally, $f \in L^\infty(\Omega)$ denotes a volume normal force,   i.e., a force oriented in $e_3$ direction.    Note that more general forces could in principle be included, e.g., boundary forces (see \cite{lecumberry}). This is neglected here for the sake of simplicity rather than generality. 

\smallskip

\textbf{Dissipation potential and viscous stress:} We now introduce a dissipation potential. We follow here the discussion in \cite[Section 2.2]{MOS} and \cite[\BBB Section 2\EEE]{MFMK}.  Consider a time dependent deformation $y: [0,T] \times \Omega \to \R^3$. Viscosity is not only related to the strain rate $\partial_t \nabla_h  y(t,x)$  but also to the strain $\nabla_h y(t,x)$.  It  can be expressed in terms of a  dissipation potential $R(\nabla_h y, \partial_t \nabla_h y)$, where $R: \R^{3 \times 3} \times \R^{3 \times 3} \to [0,\infty)$. An admissible potential has to satisfy frame indifference in the sense (see \cite{Antmann, MOS})
\begin{align}\label{R: frame indiff}
R(F,\dot{F}) = R(QF,Q(\dot{F} + AF))  \ \ \  \forall  Q \in SO(3), A \in \R^{3 \times 3}_{\rm skew}
\end{align}
for all $F \in GL_+(3)$ and $\dot{F} \in \R^{3 \times 3}$, where $GL_+(3) = \lbrace F \in \R^{3 \times 3}: \det F>0 \rbrace$ and $\R^{3 \times 3}_{\rm skew} = \lbrace A  \in \R^{3 \times 3}: A=-A^\top \rbrace$. 

From the point of modeling,  it is  more  convenient to \second assume \EEE the existence of a (smooth) global distance $D: GL_+(3) \times GL_+(3) \to [0,\infty)$ satisfying $D(F,F) = 0$ for all $F \in GL_+(3)$. From this, an associated dissipation potential $R$ can be calculated by
\begin{align}\label{intro:R}
R(F,\dot{F}) := \lim_{\eps \to 0} \frac{1}{2\eps^2} D^2(F+\eps\dot{F},F) = \frac{1}{4} \partial^2_{F_1^2} D^2(F,F) [\dot{F},\dot{F}]
\end{align}
for $F \in GL_+(3)$, $\dot{F} \in \R^{3 \times 3}$. Here, $\partial^2_{F_1^2} D^2(F_1,F_2)$ denotes the Hessian of  $D^2$ in direction of $F_1$ at $(F_1,F_2)$, which is a fourth order tensor.  For some $c>0$ we suppose that $D$ satisfies 
\begin{align}\label{eq: assumptions-D}
(i) & \ \ D(F_1,F_2)> 0 \text{ if } F_1^\top F_1 \neq F_2^\top F_2,\notag \\
(ii) & \ \ D(F_1,F_2) = D(F_2,F_1),\\
(iii) & \ \ D(F_1,F_3) \le D(F_1,F_2) + D(F_2,F_3),\notag \\
(iv) & \ \ \text{$D(\cdot,\cdot)$ is $C^3$ in a neighborhood of $SO(3) \times SO(3)$},\notag 
\\
(v)& \ \ \text{Separate frame indifference: } D(Q_1F_1,Q_2F_2) = D(F_1,F_2)\notag \\
& \ \  \ \ \ \ \ \ \ \ \    \ \  \ \ \ \ \ \ \ \ \    \ \  \ \ \ \ \ \ \ \ \    \ \  \ \ \ \ \ \ \ \ \   \forall Q_1,Q_2 \in SO(3),  \ \forall F_1,F_2 \in GL_+(3),\notag\\ 
(vi) & \ \ \text{$D(F,\Id) \ge c\dist(F,SO(3))$  $\forall F \in \R^{3 \times 3}$ in a neighborhood of $SO(3)$}.\notag
\end{align}
Note that conditions (i),(iii) state that $D$ is a true distance when restricted to symmetric matrices   with nonnegative determinants. We cannot expect more due to the separate frame indifference (v). We also point out that (v) implies \eqref{R: frame indiff} as shown in \cite[Lemma 2.1]{MOS}. Note that in our model we do not require any conditions of polyconvexity  \cite{Ball:77} neither for $W$ nor for $D$.  One possible example of $D$ satisfying 
\eqref{eq: assumptions-D} might be $D(F_1,F_2)= \BBB |F_1^\top F_1-F_2^\top F_2|\EEE$.   This leads  to $R(F,\dot F)=|\BBB  F^\top \dot F + \dot F^\top F \EEE |^2/2$ \second which is a standard choice. \EEE  For further examples  we refer  to \cite[Section 2.3]{MOS}.

\smallskip

\textbf{Equations of viscoelasticity in 3D:} Following the study in \cite{lecumberry}, we define clamped boundary conditions as follows. We consider  functions $\hat{u} \in W^{2,\infty}(S;\R^2)$ and $\hat{v} \in W^{3,\infty}(S)$ which represent in-plane and out-of-plane boundary conditions, respectively. We introduce the set of admissible configurations by
\begin{align}\label{eq: nonlinear boundary conditions}
\mathscr{S}_h = \Big\{ y \in W^{2,p}(\Omega;\R^3): \ y(x',x_3) = \begin{pmatrix} x' \\ hx_3 \end{pmatrix} + \begin{pmatrix} h^2 \hat{u}(x') \\ h \hat{v}(x') \end{pmatrix} - & x_3\begin{pmatrix} h^2  (\nabla' \hat{v}(x'))^\top  \notag \\ 0 \end{pmatrix}  \\&
\text{for } x' \in \partial S, \ x_3 \in I \Big\},
\end{align} 
where $I=(-\frac{1}{2},\frac{1}{2})$. \PPP Following \cite{MFMK} we formulate the equations of viscoelasticity  for a nonsimple material involving the perturbation $P$ (cf.\ \eqref{assumptions-P}). \EEE We introduce a differential operator associated to  $P$. To this end, we recall the notation of the scaled gradients in \eqref{eq: scaled1}-\eqref{eq: scaled2}. For $i,j \EEE \in \lbrace 1,\ldots, 3\rbrace$, we denote by $(\partial_ZP(\nabla^2_h y))_{ij*}$ the vector-valued function  $((\partial_ZP(\nabla^2_h y))_{ijk})_{k=1,2,3}$. We also introduce the    scaled (distributional) divergence  ${\rm div}_h g$  for a function $g \in L^1(\Omega;\R^3)$ \BBB by \EEE ${\rm div}_h g = \partial_1 g_1+ \partial_2 g_2 + \frac{1}{h}\partial_3 g_3$. We define
\begin{align}\label{LP-def}
\big(\mathcal{L}^h_P(\nabla^2_h y)\big)_{ij} =  - {\rm div}_h (\partial_ZP(\nabla^2_h y))_{ij*}, \ \ \ \  i,j \EEE \in \lbrace 1,\ldots, 3\rbrace
\end{align}   
for $y \in \mathscr{S}_h$. Let $\beta_1,\beta_2>0$. The equations of nonlinear viscoelasticity  can  be written as 
\begin{align}\label{nonlinear equation}
\begin{cases} -  {\rm div}_h \Big( \partial_FW(\nabla_h y) + \BBB h^{\beta_1} \EEE\mathcal{L}^h_{P}(\nabla^2_h y)  + \partial_{\dot{F}}R(\nabla_h y,\partial_t \nabla_h y)  \Big) =  \BBB  h^{\beta_2} \EEE  fe_3  & \text{in } [0,\infty) \times \Omega \\
y(0,\cdot) = y^h_0 & \text{in } \Omega \\
y(t,\cdot) \in \mathscr{S}_h &\text{for } t\in [0,\infty)
\end{cases}
\end{align}
for some $y^h_0 \in \mathscr{S}_h$, where $\partial_FW(\nabla_h y)$  $+h^{\beta_1}\mathcal{L}^h_{P}(\nabla^2_h y)$ denotes the    \emph{first Piola-Kirchhoff stress tensor} and $\partial_{\dot{F}}R(\nabla_h y,\partial_t \nabla_h y)$ the \emph{viscous stress} with $R$ as introduced in \eqref{intro:R}. \BBB \second As no surface forces are applied, we implicitly assume  zero  Neumann boundary conditions for the stress and the hyperstress on $S\times \lbrace -1/2, 1/2\rbrace$, i.e., on the top and the bottom of the cylinder $\Omega$, see e.g.~\cite{SKTR} for a specific form of such conditions.  As \eqref{eq: nonlinear boundary conditions} prescribes only the values of the function but not of the derivative, on the lateral boundary there arise additional Neumann  conditions from the second deformation gradient. (We again refer to \cite{SKTR} for details.)  As we will see below, however, they do not affect the effective 2D plate model.  \rm \EEE    Suitable scalings  $h^{\beta_1}$ and $h^{\beta_2}$ related to $\mathcal{L}^h_{P}$ and the normal force $f$, respectively, will be discussed below. \PPP Note that $\eps = h^{\beta_1}$ in \eqref{eq:viscoel-nonsimple}.  \EEE

\BBB 
\textbf{Time-discrete solutions to \eqref{nonlinear equation}:}  The first  auxiliary \EEE goal of this paper is to show existence of   time-discrete solutions to \eqref{nonlinear equation} for small $h>0$.  For this, we introduce a functional $I^{\beta_1,\beta_2}_h:W^{2,p}(\Omega;\R^3)\to\R$ describing the elastic energy of the body \BBB by \EEE
\begin{align}\label{nonlinear energy}
\BBB I^{\beta_1,\beta_2}_h(y) \EEE = \int_\Omega W(\nabla_h y(x))\, dx + h^{\beta_1}\int_\Omega P(\nabla^2_h y(x)) \, dx - h^{\beta_2}\int_\Omega f(x) y_3(x) \, dx
\end{align}
for a  deformation $y: W^{2,p}(\Omega;\R^3) \to \R^3$. \BBB We note that the energy takes into account the different scalings of the terms in \eqref{nonlinear equation}.  

We \EEE use an approximation scheme solving suitable time-incremental minimization problems: consider a fixed time step $\tau >0$ and suppose that an initial datum \BBB $y_0^h \in \mathscr{S}_h$ is given. Set $Y^0_{h,\tau} = y_0^h$. \EEE   Whenever $Y^0_{h,\tau}, \ldots, Y^{n-1}_{h,\tau}$ are known, $Y^n_{h,\tau}$ is defined as (if existent)
\begin{align}\label{incremental}
Y^n_{h,\tau} = {\rm argmin}_{y \in\mathscr{S}_h  } \Phi_h(\tau, Y^{n-1}_{h,\tau}; y), \ \ \ \Phi_h(\tau,y_0; y_1):= \BBB I^{\beta_1,\beta_2}_h(y_1)  \EEE +  \frac{1}{2\tau} \mathcal{D}^2(y_0,y_1), 
\end{align}
where \BBB $\mathcal{D}$ denotes the \emph{global dissipation distance} between two deformations, defined by \EEE
\begin{align*}
 \mathcal{D}(y_0,y_1):= \Big(\int_{\Omega} D^2(\nabla_h y_0,\nabla_h y_1)\Big)^{1/2}.
\end{align*} 
Suppose that, for a choice of $\tau$, a sequence $(Y^n_{h,\tau})_{n \in \N}$ solving  \eqref{incremental} exists. We define the  piecewise constant interpolation by
\begin{align}\label{ds}
 \tilde{Y}_{h,\tau} \BBB(0,\cdot) \EEE = Y^0_{h,\tau}, \ \ \ \tilde{Y}_{h,\tau} \BBB (t,\cdot) \EEE = Y^n_{h,\tau}  \ \text{for} \ t \in ( (n-1)\tau,n\tau], \ n\ge 1. \EEE
\end{align}
In the following,  $\tilde{Y}_{h,\tau}$  will be called a \BBB \emph{time-discrete solution}. We often drop the
$x$-dependence and write $ \tilde{Y}_{h,\tau}(t)$ for a time-discrete solution at time $t$.  \EEE Note that the existence of such solutions is usually guaranteed by the direct method of the calculus of variations under suitable compactness, coercivity, and lower semicontinuity assumptions, \BBB see Proposition \ref{maintheorem1} and its proof.  We point out that, in the setting of general  (but not thin)  bodies   with suitable boundary conditions imposed on the entire $\partial \Omega$, it has been shown in \cite[Theorem 2.1]{MFMK} that time-discrete solutions indeed converge to weak solutions to the system \eqref{nonlinear equation} as $\tau \to 0$. In the present context, time-discrete solutions will be the starting point to pass to a two-dimensional, time-continuous framework.

\smallskip

\textbf{Scaling and displacement fields:} Due to the scaling of the boundary conditions in \eqref{eq: nonlinear boundary conditions}, the energy of \BBB time-discrete solutions   $\tilde{Y}_{h,\tau}$  \EEE  is expected to be small in terms of $h$. More specifically, in our setting it will turn out that the energy is of order $h^4$ which corresponds to the so-called \emph{von K\'arm\'an regime}. (For an exhaustive treatment of different scaling regimes we refer the reader to \cite{hierarchy}.) \BBB As $y_3$ scales like $h$, see \eqref{eq: nonlinear boundary conditions}, a suitable choice for the scaling of the forces in \eqref{nonlinear energy} is therefore $h^3$, i.e., we set $\beta_2 = 3$. Moreover, we choose  $\beta_1 = 4-p\alpha$ for some  $0< \alpha< 1$. On the one hand, this will imply that $\Vert \nabla^2_h  \tilde{Y}_{h,\tau} \Vert_{L^p(\Omega)}$ is small, more precisely, of order $h^\alpha$, cf.\ \eqref{assumptions-P}(iii). On the other hand, $\alpha < 1$ will ensure that the higher order perturbation will vanish in the effective 2D limiting model. 

Based on this discussion, we introduce the rescaled nonlinear energy $\phi_h: W^{2,p}(\Omega;\R^3) \to [0,\infty]$ by  
\begin{align}\label{nonlinear energy-rescale}
\phi_h(y) = h^{-4}I^{4-p\alpha,3}_h(y)  =  \frac{1}{h^4}\int_\Omega W(\nabla_h y(x))\, dx + \frac{1}{h^{\alpha p}} \int_\Omega P(\nabla^2_h y(x)) \, dx - \frac{1}{h}\int_\Omega f(x) y_3(x) \, dx
\end{align}
for $y \in \mathscr{S}_h$. Similarly, for $y_0,y_1 \in \mathscr{S}_h$,  the rescaled global dissipation distance is given by
\begin{align}\label{eq: D,D0-1} 
 \mathcal{D}_h(y_0,y_1) = h^{-2}\mathcal{D}(y_0,y_1)=  h^{-2}\Big(\int_\Omega D^2(\nabla_h y_0, \nabla_h y_1) \Big)^{1/2}.
\end{align}
\EEE  Following the discussion in \cite{hierarchy}, for $y \in \mathscr{S}_h$ we introduce the corresponding \PPP averaged, \EEE scaled in-plane and out-of-plane displacements which measure the deviation from the mapping $(x',0)$:
\begin{align}\label{eq: in/out plane} 
u(x')  := \frac{1}{h^2} \int_I \Big( \begin{pmatrix}
 y^h_1 \\  y^h_2  \end{pmatrix} (x',x_3) - \begin{pmatrix}
x_1\\ x_2 \end{pmatrix} \Big) \, dx_3, \ \ \ \ \   v(x') := \frac{1}{h} \int_I  y^h_3  (x',x_3)\, dx_3,
\end{align}
where again $I=(-\frac{1}{2},\frac{1}{2})$.  \BBB In a similar fashion, given a time-discrete solution $\tilde{Y}_{h,\tau}$, we introduce \second  averaged functions $\tilde{U}_{h,\tau}$ and $\tilde{V}_{h,\tau}$ dependent on  $(t,x')$. \EEE Via the minimizing movement scheme,  we will later see that along a sequence $(\tilde{Y}_{h,\tau}(t))_{h,\tau}$ we get  $\tilde{U}_{h,\tau}(t)  \rightharpoonup u(t)$ weakly in $W^{1,2}(S;\R^2)$ and $\tilde{V}_{h,\tau}(t) \to v(t)$ strongly in $W^{1,2}(S)$ with $v(t) \in W^{2,2}(S)$ for each $t \ge 0$, when $h,\tau \to 0$.  The \BBB main \EEE goal of our work is to understand which equations are solved by $(u(t),v(t))$.

 \smallskip

\textbf{Quadratic forms:}
To formulate the effective 2D problem for the scaled in-plane and out-of-plane displacements, we need to consider \BBB various quadratic forms. First, we define \EEE $Q_W^3:\R^{3 \times 3} \to \R$  by $Q_W^3(F) = \partial^2_{F^2} W(\Id)[F,F]$. One can show that it depends only on the symmetric part $\frac{1}{2}(F^\top + F)$ and that it is positive definite on \BBB $\R^{3 \times 3}_{\rm sym} = \lbrace A  \in \R^{3 \times 3}: A=A^\top \rbrace$, \EEE cf.\ Lemma \ref{D-lin}. We also introduce  $Q_W^2: \R^{2 \times 2} \to \R$ by
\begin{align}\label{eq:Q2}
Q_W^2(G) = \min_{a \in\R^3} Q_W^3(G^* + a \otimes e_3 + e_3 \otimes a )
\end{align}
for $G \in \R^{2 \times 2}$, where the entries of  $G^* \in \R^{3 \times 3}$ are given by $G^*_{ij} = G_{ij}$ for $i,j\in \lbrace 1,2\rbrace$ and zero otherwise. Note that \eqref{eq:Q2} corresponds to a minimization over stretches in the $e_3$ direction. We will assume that the minimum in \eqref{eq:Q2} is attained for $a=0$. Similarly, we define
\begin{align}\label{eq:Q22}
Q_D^3(F) =  \frac{1}{2}\partial^2_{F^2_1} D^2(\Id,\Id)[F,F],   \ \ \ \ \ \ \ Q_D^2(G) = \min_{a \in\R^3} Q_D^3(G^* + a \otimes e_3 + e_3 \otimes a ).
\end{align}
 (\second Notice that  $Q_D^3(F)=2R(\Id,F)$ with $R$ from  \eqref{intro:R}.\EEE)   We again assume that the minimum is attained for $a=0$.

The assumption that $a=0$ is a minimum in \eqref{eq:Q2}-\eqref{eq:Q22} corresponds to a model with  zero Poisson's ratio
in the $e_3$ direction. This assumption is not needed in the \PPP purely \EEE static analysis \cite{hierarchy, lecumberry}. In our setting, it is only needed in the proof of lower semicontinuity of slopes, see Theorem \ref{theorem: lsc-slope}. Dropping this assumption would lead to a considerably more involved limiting description which we do not want to pursue here.  We refer to  Remark \ref{rem: Poisson2} for some details in that direction.

We also introduce  corresponding   symmetric   fourth order tensors $\C^d_W$ and $\C^d_D$, $d=2,3$,  satisfying
\begin{align}\label{eq: order4}
Q_W^3(F) = \C^3_W[F,F] \ \ \forall F \in \R^{3\times 3}, \ \ \ \ \ \ \ Q_W^2(G) = \C^2_W[G,G] \ \ \forall G \in \R^{2 \times 2}
\end{align}
and likewise for $\C_D^d$, $d=2,3$.

\smallskip

 \textbf{Equations of viscoelasticity in 2D:}
 We now  present  the effective 2D equations. By recalling definition \eqref{eq: in/out plane} and the boundary conditions in the 3D setting \eqref{eq: nonlinear boundary conditions},   we first introduce the relevant space by 
\begin{align}\label{eq: BClinear}
{\mathscr{S}}_0 = \lbrace (u,v) \in W^{1,2}(S;\R^2) \times W^{2,2}(S): \ u = \hat{u}, \ v = \hat{v}, \BBB \ \nabla' v = \nabla' \hat{v} \EEE \ \text{ on } \partial S \rbrace.
\end{align}
\second  The clamped boundary conditions correspond to the ones considered in \cite{lecumberry}. We emphasize that Neumann conditions in \eqref{nonlinear equation} on the lateral boundary arising from the second deformation gradient do not affect the  plate equations \eqref{eq: equation-simp} below since  our scaling makes the second gradient  vanish as the thickness of the domain tends to zero. \rm \EEE

\PPP Given $(u_0,v_0) \in \mathscr{S}_0$, \EEE we consider the equations  \NNN
\begin{align}\label{eq: equation-simp}
\begin{cases}  {\rm div}_2\Big(\C^2_W\big( e(u)  + \frac{1}{2} \nabla' v \otimes \nabla' v  \big) +  \C^2_D \big( e(\partial_t u) +   \nabla' \partial_t v   \odot \nabla' v \big) \Big)   = 0, &\vspace{0.1cm} \\ 
-{\rm div}_2\Big(\Big(\C^2_W\big( e(u)  + \frac{1}{2} \nabla' v \otimes \nabla' v  \big)  +  \C^2_D \big( e(\partial_t u) +   \nabla' \partial_t v   \odot \nabla' v \big) \Big) \nabla' v \Big)  &\vspace{0.1cm}\\
 \quad\quad\quad\quad\quad\quad\quad \quad \ \, \quad \ \  \quad  + \tfrac{1}{12}  {\rm div}_2 \, {\rm div}_2\Big( \C^2_W (\nabla')^2 v + \C^2_D (\nabla')^2 \partial_t v \Big)   =   f  &   \text{in } [0,\infty) \times S \\
u(0,\cdot) = u_0, \  v(0,\cdot) = v_0 &   \text{in } S \\
(u(t,\cdot), v(t,\cdot)) \in {\mathscr{S}}_0 & \hspace{-0.1cm} \text{ for } t\in [0,\infty)
\end{cases}
\end{align}
\EEE where $\C^2_W$ and  $\C^2_D$ are defined in \eqref{eq: order4}, \BBB and $\odot$ denotes the symmetrized tensor product. \EEE Note that the frame indifference of the energy and the dissipation (see \eqref{assumptions-W}(ii) and \eqref{eq: assumptions-D}(v), respectively) imply that the contributions only depend on the symmetric part of the strain $e(u) := \frac{1}{2}(  \nabla' u  +(\nabla' u)^\top)$ and the strain rate  $e(\partial_t u) := \frac{1}{2}( \partial_t \nabla' u + \partial_t (\nabla' u)^\top)$. Here, ${\rm div}_2$ denotes the  distributional   divergence in dimension two. 

We also say that $(u,v) \in W^{1,2}([0,\infty);{\mathscr{S}}_0)$ is a \emph{weak solution} of \eqref{eq: equation-simp}  if $u(0,\cdot) = u_0$, $v(0,\cdot) = v_0$ and for a.e.\ $t \ge 0$ we have \NNN
 \begin{subequations}\label{eq: weak equation}
\begin{align}
& \int_S \Big(\C^2_W\big( e(u)  + \tfrac{1}{2} \nabla' v \otimes \nabla' v  \big)  +  \C^2_D \big( e(\partial_t u ) +  \nabla' \partial_t v   \odot  \nabla' v \big) \Big) : \nabla' \varphi_u    = 0,\label{eq: weak equation1} \\
& \int_S \Big(\C^2_W\big( e(u)  + \tfrac{1}{2} \nabla' v \otimes \nabla' v  \big)  +  \C^2_D \big( e(\partial_t u ) +  \nabla' \partial_t v  \odot\nabla' v \big) \Big) : \big(\nabla' v  \odot  \nabla' \varphi_v  \big) \notag \\
& \quad\quad\quad\quad\quad\quad\quad \ \  \ \  \, \,\quad  + \frac{1}{12}  \int_S \Big(\C^2_W (\nabla')^2 v + \C^2_D (\nabla')^2 \partial_t v \Big) : (\nabla')^2 \varphi_v = \int_S f\varphi_v, \label{eq: weak equation2} 
\end{align}
\end{subequations} 
\EEE for all $\varphi_u \in W^{1,2}_0(S;\R^2)$ and  $\varphi_v \in W^{2,2}_0(S)$. Note that \eqref{eq: weak equation1} corresponds \second to two scalar equations \EEE and \eqref{eq: weak equation2} corresponds to one \second  scalar \rm \EEE equation, respectively.  Our goal will be to show that \BBB time-discrete solutions to \eqref{nonlinear equation}, as introduced in \eqref{ds}, \EEE converge to \PPP weak \EEE solutions to \eqref{eq: equation-simp} in a suitable sense.

  We also mention that the  equations   are related to the \emph{von K\'arm\'an} functional
\begin{align}\label{eq: phi0}
{\phi}_0(u,v) := \int_S \frac{1}{2}Q_W^2\Big( e(u) + \frac{1}{2} \nabla'v \otimes \nabla'v \Big) + \frac{1}{24}Q_W^2((\nabla')^2 v) - \int_S f v
\end{align}
for $(u,v) \in {\mathscr{S}}_0$: \second actually, \EEE we will also see  that ${\phi}_0$ can be related to $\phi_h$ (see \eqref{nonlinear energy-rescale}) in the sense of $\Gamma$-convergence, cf.\ Section \ref{sec: gamma} below. A similar relation holds \BBB for $\mathcal{D}_h$, introduced in \eqref{eq: D,D0-1}, and the \emph{global dissipation distance} in the  2D setting, \EEE defined by

 \begin{align}\label{eq: D,D0-2} 
  {\mathcal{D}}_0( (u_0,v_0),(u_1,v_1)) & := \Big(\int_S Q^2_D\Big( e( u_1) - e(u_0) + \frac{1}{2} \nabla'v_1 \otimes \nabla'v_1 - \frac{1}{2} \nabla'v_0 \otimes \nabla'v_0 \Big) \notag \\
  & \ \ \ \ \  +  \frac{1}{12}   Q_D^2\big((\nabla')^2 v_1 - (\nabla')^2 v_0 \big)  \Big)^{1/2}
\end{align}
for  $(u_0, v_0), (u_1,v_1) \in {\mathscr{S}}_0$.

\subsection{Main results}
Define the sublevel sets $\mathscr{S}_h^M := \lbrace y\in \mathscr{S}_h: \phi_h(y) \le M\rbrace$. Our general strategy will be to show that the spaces $(\mathscr{S}_h^M, \mathcal{D}_h)$ and $({\mathscr{S}}_0, {\mathcal{D}}_0)$ are complete metric spaces (see Lemma \ref{th: metric space} and Lemma \ref{th: metric space-lin} below) and to follow the \second methodology developed  in \cite{AGS}. \EEE 

To show existence of solutions to the equations,  we will apply the theory of \cite{AGS} about  curves of maximal slope. By using the property that in Hilbert spaces curves of maximal slope can be related to gradient flows, we then find weak solutions to    \eqref{eq: equation-simp}.  \BBB To understand the relation of \EEE solutions in the 3D and 2D setting, we will employ an abstract convergence result for curves of maximal slope  and \BBB  their approximation  via the minimizing movement scheme, see \cite{Ortner, S2}. The relevant results about curves of maximal slope are recalled in Section   \ref{sec3}. \EEE

Our first main result addresses the existence of \BBB time-discrete \EEE solutions to the 3D problem.

\BBB 

\begin{proposition}[Time-discrete solutions in the 3D setting]\label{maintheorem1}
Let $M>0$ and $\mathscr{S}_h^M = \lbrace y \in \mathscr{S}_h: \phi_h(y) \le M\rbrace$. Let $\beta_1 = 4-\alpha p$ and $\beta_2 = 3$.  Let $y_0^h \in \mathscr{S}_h^M$. Then, for $h>0$ sufficiently small only depending on $M$, the  sequence of minimization problems in \eqref{incremental} has a solution, and gives rise to a time-discrete solution $\tilde{Y}_{h,\tau}$ with $\tilde{Y}_{h,\tau}(0) = y_0^h$.
\end{proposition}
\EEE
Note that  the existence  of weak  \BBB (time-continuous) \EEE solutions to \eqref{nonlinear equation}  has been addressed in \cite{MFMK} for \BBB bulk materials. \EEE However, due to the \BBB thinness of the domain   representing the body \BBB and the imposed boundary conditions,   the results obtained there  are not applicable. \PPP This is  due to the fact that specific \EEE constants depend on the domain and \PPP blow up \EEE for vanishing  thickness. Therefore,  here we only prove existence of time-discrete solutions. \EEE 

 For the main definitions and notation for curves of maximal slope and strong upper gradients we refer to Section \ref{sec: defs}. In particular, we write $|\partial {\phi}_0|_{{\mathcal{D}}_0}$ for the (local) slopes, see Definition \ref{main def2}.  For the 2D problem we obtain the following results. 

\begin{theorem}[Solutions in the 2D setting]\label{maintheorem2}
The limiting 2D problem has the following properties: 

(i)  (Curves of maximal slope) \EEE For all $(u_0,v_0) \in {\mathscr{S}}_0$  there exists a curve of maximal slope $(u,v):[0,\infty) \to {\mathscr{S}}_0$ for ${\phi}_0$   with respect to the strong upper gradient $|\partial {\phi}_0|_{{\mathcal{D}}_0}$ satisfying $(u,v)(0)=(u_0,v_0)$. 
  
  (ii)  (Relation to PDE)  For all $(u_0,v_0) \in {\mathscr{S}}_0$,  each  curve of maximal slope $(u,v) :[0,\infty) \to {\mathscr{S}}_0$ with $(u,v)(0)=(u_0,v_0)$  is a weak solution  to the partial differential equations \eqref{eq: equation-simp} in the sense of \eqref{eq: weak equation}.

\end{theorem}

\first 
We  mention that in \cite[Theorem 4.1]{MFMKJV} we provide a slightly different proof of Theorem \ref{maintheorem2}(i), without using the fact that the 2D model is the limit of the time-discrete 3D model. However, the proof still heavily relies on the properties of $\phi_0$, ${\mathcal{D}}_0$, and $|\partial {\phi}_0|_{{\mathcal{D}}_0}$ derived in the present paper. \rm \EEE

Finally, we study the relation of  \BBB time-discrete solutions \eqref{ds} and \PPP weak \EEE solutions to the equations \eqref{eq: equation-simp}. To this end, we need to specify the topology of the convergence. Given $y^h \in \mathscr{S}_h$ and $(u,v) \in {\mathscr{S}}_0$, we say $y^h \stackrel{\pi\sigma}{\to} (u,v)$ as $h \to 0$ if the corresponding \second averaged and scaled \EEE displacements fields defined in  \eqref{eq: in/out plane}, denoted by $u^h$ and $v^h$,  satisfy $u^h \rightharpoonup u$ weakly in $W^{1,2}(S;\R^2)$ and $v^h \to v$ strongly in $W^{1,2}(S)$. (The symbol $\pi\sigma$ is  used because of  the abstract convergence result, see Section \ref{sec: auxi-proofs}.)

\second 
\begin{theorem}[Relation between 3D and 2D systems]\label{maintheorem3}
Let $\beta_1 = 4-p\alpha$ and $\beta_2 = 3$. Let $(u_0,v_0) \in {\mathscr{S}}_0$ be an initial datum. \\
\noindent
(i) Then, for  the family of sequences of initial data we have
\begin{align}\label{eq:datasequence}
\mathcal{B}(u_0,v_0) = \big\{ (y^h_0)_h: \ y^h_0 \in \mathscr{S}_h, \ y^h_0 \stackrel{\pi\sigma}{\to} (u_0,v_0), \ \phi_h(y^h_0) \to {\phi}_0((u_0,v_0)) \big\}\ne \emptyset.
\end{align}

\noindent
(ii) We consider a sequence $(y^h_0)_h \in \mathcal{B}(u_0,v_0)$,  a null sequence $(\tau_h)_h$, and a sequence of time-discrete solutions $\tilde{Y}_{h,\tau_h}$ as in \eqref{ds}  with  $\tilde{Y}_{h,\tau_h}(0)=y^h_0$. \\ 
\noindent Then, there exists a curve of maximal slope $(u,v): [0,\infty) \to {\mathscr{S}}_0$ for ${\phi}_0$  with respect to  $|\partial {\phi}_0|_{{\mathcal{D}}_0}$ satisfying $u(0) = u_0$ and $v(0) = v_0$ such that up to a subsequence (not relabeled) there holds 
$$\tilde{Y}_{h,\tau_h}(t) \stackrel{\pi\sigma}{\to} (u(t),v(t)), \ \ \ \ \ \ \phi_h(\tilde{Y}_{h,\tau_h}(t)) \to {\phi}_0(u(t),v(t))  \ \ \ \ \text{for all} \ t \in [0,\infty) \ \  \text{ as $h \to 0$}.$$
\end{theorem}

\EEE

\second  We point out that \eqref{eq:datasequence} corresponds to the existence of recovery sequences for the static problem. We note that the existence of the time-discrete solutions $\tilde{Y}_{h,\tau_h}$ in (ii) is guaranteed by Proposition \ref{maintheorem1}. Item (ii) shows the convergence of time-discrete solutions to   \eqref{nonlinear equation} to (time-continuous) weak solutions to \eqref{eq: equation-simp}. Moreover, we also have convergence of the energies.  \rm \EEE From now on we set $f \equiv 0 $ for convenience. The general case indeed follows with minor modifications, which are standard.\EEE

\section{Preliminaries: Curves of maximal slope}\label{sec3}

    In this section we  recall the relevant definitions \BBB about curves of maximal slope and present a convergence result of time-discrete solutions to curves of maximal slope.  \EEE

\subsection{Definitions}\label{sec: defs}

We consider a   complete metric space $(\mathscr{S},\mathcal{D})$. We say a curve $y: (a,b) \to \mathscr{S}$ is \emph{absolutely continuous} with respect to $\mathcal{D}$ if there exists $m \in L^1(a,b)$ such that
\begin{align*}
\mathcal{D}(y(s),y(t)) \le \int_s^t m(r) \, dr \ \ \   \text{for all} \ a \le s \le t \le b.
\end{align*}
The smallest function $m$ with this property, denoted by $|y'|_{\mathcal{D}}$, is called \emph{metric derivative} of  $y$  and satisfies  for a.e.\ $t \in (a,b)$   (see \cite[Theorem 1.1.2]{AGS} for the existence proof)
$$|y'|_{\mathcal{D}}(t) := \lim_{s \to t} \frac{\mathcal{D}(y(s),y(t))}{|s-t|}.$$
We  define the notion of a \emph{curve of maximal slope}. We only give the basic definition here and refer to \cite[Section 1.2, 1.3]{AGS} for motivations and more details.  By  $h^+:=\max(h,0)$ we denote the positive part of a function  $h$.

\begin{definition}[Upper gradients, slopes, curves of maximal slope]\label{main def2} 
 We consider a   complete metric space $(\mathscr{S},\mathcal{D})$ with a functional $\phi: \mathscr{S} \to (-\infty,+\infty]$.

(i) A function $g: \mathscr{S} \to [0,\infty]$ is called a strong upper gradient for $\phi$ if for every absolutely continuous curve $ \BBB y: \EEE (a,b) \to \mathscr{S}$ the function $g \circ y$ is Borel and 
$$|\phi(y(t)) - \phi(y(s))| \le \int_s^t g( y(r)) |y'|_{\mathcal{D}}(r)\,dr \  \ \  \text{for all} \ a< s \le t < b.$$

(ii) For each $y \in \mathscr{S}$ the local slope of $\phi$ at $y$ is defined by 
$$|\partial \phi|_{\mathcal{D}}(y): = \limsup_{z \to y} \frac{(\phi(y) - \phi(z))^+}{\mathcal{D}(y,z)}.$$

(iii) An absolutely continuous curve $y: (a,b) \to \mathscr{S}$ is called a curve of maximal slope for $\phi$ with respect to the strong upper gradient $g$ if for a.e.\ $t \in (a,b)$
$$\frac{\rm d}{ {\rm d} t} \phi(y(t)) \le - \frac{1}{2}|y'|^2_{\mathcal{D}}(t) - \frac{1}{2}g^2(y(t)).$$
\end{definition}

\subsection{Curves of maximal slope as limits of time-discrete solutions}\label{sec: auxi-proofs}

We consider a sequence of complete metric spaces $(\mathscr{S}_k, \mathcal{D}_k)_k$, as well as a limiting complete metric space $(\mathscr{S},\mathcal{D})$. Moreover, let $(\phi_k)_k$ be a sequence of functionals with $\phi_k: \mathscr{S}_k \to [0,\infty]$ and $\phi: \mathscr{S} \to [0,\infty]$.

We introduce time-discrete solutions for the energy $\phi_k$ and the metric $\mathcal{D}_k$ by solving suitable time-incremental minimization problems: consider a fixed time step $\tau >0$ and suppose that an initial datum $Y^0_{k,\tau}$ is given. Whenever $Y_{k,\tau}^0, \ldots, Y^{n-1}_{k,\tau}$ are known, $Y^n_{k,\tau}$ is defined as (if existent)
\begin{align}\label{eq: ds-new1}
Y_{k,\tau}^n = {\rm argmin}_{v \in \mathscr{S}_k} \Phi_k(\tau,Y^{n-1}_{k,\tau}; v), \ \ \ \Phi_k(\tau,u; v):=  \frac{1}{2\tau} \mathcal{D}_k(v,u)^2 + \phi_k(v). 
\end{align}
We suppose that for a choice of $\tau$ a sequence $(Y_{k,\tau}^n)_{n \in \N}$ solving \BBB \eqref{eq: ds-new1} \EEE exists. Then we define the  piecewise constant interpolation by
\begin{align}\label{eq: ds-new2}
 \tilde{Y}_{k,\tau}(0) = Y^0_{k,\tau}, \ \ \ \tilde{Y}_{k,\tau}(t) = Y^n_{k,\tau}  \ \text{for} \ t \in ( (n-1)\tau,n\tau], \ n\ge 1.  
\end{align}
We call  $\tilde{Y}_{k,\tau}$  a \emph{time-discrete solution}. Note that the existence of such  solutions is usually guaranteed by the direct method of the calculus of variations under suitable compactness, coercivity, and lower semicontinuity assumptions. 

Our goal is to study the limit of time-discrete solutions as $k \to \infty$. To this end, we need to introduce a suitable topology for the convergence. First, although $\mathcal{D}$ naturally induces a topology on the limiting space $\mathscr{S}$, it is often convenient to consider a weaker Hausdorff topology $\sigma$ on $\mathscr{S}$ to have more flexibility in the derivation of compactness properties (see \cite[Remark 2.0.5]{AGS}). We assume that for each $k \in \N$ there exists a map $\pi_k: \mathscr{S}_k \to \mathscr{S}$. Given  a sequence $(z_k)_k$, $z_k \in \mathscr{S}_k$, and $z \in \mathscr{S}$,  we say
\begin{align}\label{eq: sigma'}
 z_k \stackrel{\pi\sigma}{\to} z \  \ \ \ \  \text{if} \ \ \ \pi_k(z_k) \stackrel{\sigma}{\to} z. 
 \end{align}
We suppose that the topology $\sigma$ satisfies  
\begin{align}\label{compatibility}
\begin{split}
z_k \stackrel{\pi\sigma}{\to} z, &\ \  \bar{z}_k \stackrel{\pi\sigma}{\to} \bar{z}  \ \ \  \Rightarrow \ \ \ \liminf_{k \to \infty} \mathcal{D}_k(z_k,\bar{z}_k) \ge  \mathcal{D}(z,\bar{z}).
\end{split}
\end{align}
Moreover, assume that    there exists a  $\sigma$-sequentially   compact set   $K_N \subset \mathscr{S}$   such that for all $k \in \N$
\begin{align}\label{basic assumptions2}
\lbrace  \pi_k(z): \ z \in \mathscr{S}_k, \ \phi_k(z) \le N \rbrace \subset K_N.
\end{align}
Specifically, for a sequence $(z_k)_k$ with $\phi_k(z_k) \le N$, we find a subsequence (not relabeled) and $z \in \mathscr{S}$ such that $\pi_k(z_k) \stackrel{\sigma}{\to} z$. We suppose lower semicontinuity of the energies and the slopes in the following sense: for all $z \in \mathscr{S}$ and $(z_k)_k$, $z_k \in \mathscr{S}_k$, we have
\begin{align}\label{eq: implication}
\begin{split}
z_k \stackrel{\pi\sigma}{\to}  z \ \ \ \ \  \Rightarrow \ \ \ \ \  \liminf_{k \to \infty} |\partial \phi_{k}|_{\mathcal{D}_{k}} (z_{k}) \ge |\partial \phi|_{\mathcal{D}} (z), \ \ \ \ \  \liminf_{k \to \infty} \phi_{k}(z_{k}) \ge \phi(z).
\end{split}
\end{align}

We now formulate the main convergence result of time-discrete solutions to curves of maximal slope, proved in \cite[Section 2]{Ortner}.


\begin{theorem}\label{th:abstract convergence 2}
Suppose that   \eqref{compatibility}-\eqref{eq: implication} hold. Moreover, assume that    $|\partial \phi|_{\mathcal{D}}$ is a  strong upper gradient for $ \phi $.   Consider a  null sequence $(\tau_k)_k$. Let   $(Y^0_{k,\tau_k})_k$ with $Y^0_{k,\tau_k} \in \mathscr{S}_k$  and $\bar{z}_0 \in \mathscr{S}$ be initial data satisfying 
\begin{align}\label{eq: abstract assumptions1}
(i)& \ \ \sup\nolimits_k \mathcal{D} \big(\pi_k\big(Y^0_{k,\tau_k}\big),\bar{z}_0\big) < + \infty, \notag \\ 
(ii)& \ \  Y^0_{k,\tau_k} \stackrel{\pi\sigma}{\to} \bar{z}_0 , \ \ \ \ \  \phi_k(Y^0_{k,\tau_k}) \to \phi(\bar{z}_0).
\end{align}
Then for each sequence of discrete solutions $(\tilde{Y}_{k,\tau_k})_k$  starting from $(Y^0_{k,\tau_k})_k$ there exists a limiting function $z: [0,+\infty) \to \mathscr{S}$ such that up to a  subsequence (not relabeled) 
$$\tilde{Y}_{k,\tau_k}(t) \stackrel{\pi\sigma}{\to} z(t), \ \ \ \ \ \phi_k(\tilde{Y}_{\tau_k}(t)) \to \phi(z(t)) \ \ \  \ \ \ \ \ \forall t \ge 0$$
as $k \to \infty$, and $z$ is a curve of maximal slope for $\phi$ with respect to $|\partial \phi|_{\mathcal{D}}$. In particular,  $z$ satisfies the energy identity 
\begin{align}\label{maximalslope}
\frac{1}{2} \int_0^T |z'|_{\mathcal{D}}^2(t) \, dt + \frac{1}{2} \int_0^T |\partial \phi|_{\mathcal{D}}^2(z(t)) \, dt + \phi(z(T)) = \phi(\bar{z}_0) \ \  \ \ \ \forall T>0. 
\end{align} 

\end{theorem}

The statement is a combination of convergence results for curves of maximal slope \cite{MFMK, S2} with their approximation by time-discrete solutions via the minimizing movement scheme. We refer to \cite[Theorem 3.6]{MFMK} for an abstract convergence result for curves of maximal slope in a setting \BBB where \EEE conditions  \eqref{compatibility}-\eqref{eq: implication} hold. We also mention the version in the seminal work \cite{S2} where condition \eqref{compatibility} is replaced by a lower bound condition on the metric derivatives along the sequence.  

For the proof \BBB of Theorem \ref{th:abstract convergence 2} \EEE we refer to \cite[Section 2]{Ortner}. We also mention \cite[Theorem 3.7]{MFMK} for a formulation of the result which is a bit closer to the statement given here. Strictly speaking, in \cite{Ortner}, only the case of a single metric space is considered. The generalization   to a sequence of spaces,  however, is straightforward, cf. \cite[equation (2.7)]{S2}. Note that the nonnegativity of the energies $\phi_k$ can be generalized to a suitable \emph{coerciveness} condition, see \cite[(2.1.2b)]{AGS} or \cite[(2.5)]{Ortner}. This is not included here for the sake of simplicity.

 Let us also mention the recently obtained variant \cite{BCGS} where it is  not necessary to require that  $|\partial \phi|_{\mathcal{D}}$ is a \emph{strong}   upper gradient,  cf.\ \cite[Definition 1.2.1 and Definition 1.2.2]{AGS} for the definition of  strong and weak upper gradients.  This comes at the expense of the fact that  the lower semicontinuity along the sequence $(\phi_k)_k$ (see \eqref{eq: implication}) has to be replaced by a continuity condition along $(\phi_k)_k$ for sequences  $(\pi_k(z_k))_k$   converging with respect to the metric $\mathcal{D}$.

\section{Properties of energies and dissipation distances}\label{sec:energy-dissipation}

In this section we prove \BBB basic \EEE properties of the energies and dissipation distances, \BBB and we establish properties for the local slope in the 2D setting. \EEE  Let $h>0$ and $0 < \alpha < 1$. We  recall the definition of the nonlinear energy and the dissipation distance in \eqref{nonlinear energy-rescale} and \eqref{eq: D,D0-1}, respectively.  We also recall \eqref{eq: nonlinear boundary conditions} and the notation for the sublevel sets  $\mathscr{S}_h^M = \lbrace y \in \mathscr{S}_h: \phi_h(y) \le   M \rbrace$. In the whole section, $C\ge 1$ and $ 0 < c \le 1$ indicate generic constants, which may vary from line to line and depend on $M$, $S$, the exponent $p>3$ (see \eqref{assumptions-P}), \BBB $\alpha$, \EEE on the constants in \eqref{assumptions-W},  \eqref{assumptions-P}, \eqref{eq: assumptions-D}, and on the boundary data $\hat{u}$ and $\hat{v}$. However, all constants are always independent of the small parameter $h$ \BBB and the deformations $y$. \EEE

\subsection{Basic properties}

We start with some properties about the Hessian of $W$ and $D^2$. By $\partial^2 D^2$ we denote the Hessian and by $\partial^2_{F_1^2} D^2, \partial^2_{F_2^2} D^2$ the Hessian in direction of the first or second entry of $D^2$, respectively. Moreover, we define ${\rm sym }(F) = \frac{1}{2}(F + F^\top)$ for $F \in \R^{d \times d}$,   $d=2,3$.  Recall the definitions of the quadratic forms in \eqref{eq:Q2}-\eqref{eq:Q22}.  By $\Id \subset \R^{3\times 3}$ we again denote the identity matrix.   

\begin{lemma}[Properties of Hessian]\label{D-lin}
 (i) $\partial^2_{F_1^2}D^2(Y,Y) = \partial^2_{F_2^2}D^2(Y,Y)$ for all  $Y \in \R^{3 \times 3}$ in a neighborhood of \BBB $SO(3)$ \EEE such that $\partial^2 D^2(Y,Y)$ exists.  

(ii) There \BBB exists \EEE a  constant $c>0$ such that $Q^d_W(F) = Q^d_W({\rm sym}(F))  \ge c|{\rm sym}(F)|^2$ and  $Q^d_D(F) = Q^d_D({\rm sym}(F))  \ge c|{\rm sym}(F)|^2$ for \BBB all $F \in \R^{d \times d}$,  $d=2,3$. \EEE

\BBB (iii) There  exists   $C>0$ such that for all $F_0, F_1 \in \R^{3 \times 3} $ in a neighborhood of  $SO(3)$  there holds $|W(F_1) - W(F_0) - \frac{1}{2}(Q_W^3(F_1 - \Id) - Q_W^3(F_0 - \Id))| \le  \sum\nolimits_{k=1}^3  C|F_0 - \Id|^{3-k}|F_1 - F_0|^k$. \EEE
\end{lemma}

\begin{proof}
\BBB For the proof of (i) and (ii) we refer to \cite[Lemma 4.1]{MFMK}. To see (iii), we perform a Taylor expansion. First, we find 
$$W(F_1) = W(F_0) + DW(F_0):(F_1 -F_0) + \tfrac{1}{2}D^2W(F_0)[F_1-F_0,F_1-F_0] + {\rm O}(|F_1 - F_0|^3).$$
We observe that  $DW(F_0) = DW(\Id) + D^2W(\Id)(F_0 - \Id) + {\rm O}(|F_0-\Id|^2)$ and $DW(\Id) = 0$ by \eqref{assumptions-W}(iii). Moreover, $|D^2W(F_0) - D^2W(\Id)| \le C|F_0 -\Id|$  by the regularity of $W$.     Thus, we get
 \begin{align}\label{eq: W-lemma1}
W(F_1) & =   W(F_0) + D^2W(\Id)[F_0 - \Id, F_1 -F_0]  + \tfrac{1}{2}D^2W(\Id)[F_1-F_0, F_1 -F_0] \notag\\
& \ \ \   +  {\rm O}(|F_0 - \Id||F_1 - F_0|^2)  +  {\rm O}(|F_0 - \Id|^2|F_1 - F_0|)  + {\rm O}(|F_1 - F_0|^3).
\end{align} 
By recalling that $Q_W^3(F) = D^2W(\Id)[F, F]$ for $F \in \R^{3 \times 3}$ and the fact that $D^2W(\Id)[\cdot,\cdot]$ is symmetric in the two entries, an elementary computation yields
\begin{align}\label{eq: W-lemma2}
Q_W^3(F_1 - \Id) - Q_W^3(F_0 - \Id) = 2D^2W(\Id)[F_0 - \Id, F_1 -F_0]  + D^2W(\Id)[F_1-F_0, F_1 -F_0].
\end{align}
The result follows by combination of \eqref{eq: W-lemma1} and \eqref{eq: W-lemma2}.  \EEE
\end{proof}

\BBB The following geometric rigidity result will be a key ingredient for our analysis. \EEE

\begin{lemma}[Rigidity in thin domains]\label{lemma:rigidity}
For $h$ sufficiently small, for all $y  \in \mathscr{S}_h^M$ there exists a mapping $R(y) \in W^{1,2}(S; SO(3))$ satisfying  
\begin{align}\label{eq:rigidity}
(i) & \ \ \Vert \nabla_h y - R(y) \Vert^2_{L^2(\Omega)} \le Ch^4, \notag\\
(ii) & \ \  \Vert \nabla_h y - \Id \Vert^2_{L^2(\Omega)} \le Ch^2, \notag \\
(iii) & \ \ \Vert \nabla' R(y)   \Vert^2_{L^2(S)}\le C h^2  \notag \\
(iv) & \ \   \Vert R(y) -\Id  \Vert_{L^q(S)}\le C_qh,\notag\\  
(v) & \ \ \Vert \nabla_h y -\Id  \Vert_{L^\infty(\Omega)}\le Ch^{\alpha},\notag \\
(vi) & \ \ \BBB \Vert R(y) -\Id  \Vert_{L^\infty(S)}\le Ch^\alpha, \EEE  
\end{align}
where $C_q$ depends also on $q \in \PPP [1,\infty)\EEE$. (In (i), $R(y)$ is extended to   $\Omega$ by $R(y)(x',x_3) = R(y)(x')$.) 
\end{lemma}

\begin{proof} Property (i) is based on geometric rigidity \cite{FrieseckeJamesMueller:02} and is proved in \cite[Theorem 6, Remark 5]{hierarchy}, where we use that $\Vert \dist(\nabla_h y,SO(3)) \Vert^2_{L^2(\Omega)} \le CMh^4$ by \eqref{assumptions-W}(iii), \BBB \eqref{nonlinear energy-rescale}, \EEE and the fact that $y \in \mathscr{S}_h^M$. Also (ii)-(iv) are proved there with a rotation $\bar{Q}$ in place of $\Id$. The fact that we may choose $\bar{Q} = \Id$ is due to the boundary conditions, which  break the rotational invariance, see \cite[Lemma 13]{lecumberry}. 

We now show (v). By the definition of $\phi_h$ and \eqref{assumptions-P}(iii) we get $\Vert \nabla^2_h y \Vert_{L^p(\Omega)} \le Ch^{\alpha}$ for all $y \in \mathscr{S}^M_h$,   where the constant depends on $M$.   In particular, \BBB by \eqref{eq: scaled1}-\eqref{eq: scaled2}  \EEE this implies
$$\Vert \nabla y_{,3} \Vert_{L^p(\Omega)} \le \Vert \nabla_h y_{,3} \Vert_{L^p(\Omega)} = h\Vert \nabla_h \, h^{-1} y_{,3} \Vert_{L^p(\Omega)} \le h \Vert \nabla^2_h y \Vert_{L^p(\Omega)} \le Ch^{1+\alpha}.$$
As $p>3$, Poincar\'e's inequality yields some $F \in \R^{3 \times 3}$   such that  
\begin{align*}
\Vert \nabla' y -   (Fe_1, Fe_2)  \Vert_{L^\infty(\Omega)} &\le C \Vert \nabla^2 y \Vert_{L^p(\Omega)} \le C \Vert \nabla^2_h y \Vert_{L^p(\Omega)} \le C h^{\alpha}, \\ \notag
  \Vert y_{,3} -h Fe_3 \Vert_{L^\infty(\Omega)} &  \le   C   \Vert \nabla y_{,3} \Vert_{L^p(\Omega)}  \le  Ch^{1+\alpha}
\end{align*}
for a constant additionally depending on $\Omega$ and $p$. This implies $\Vert \nabla_h y -F \Vert_{L^\infty(\Omega)} \le Ch^{\alpha}$. Along with (ii), the triangle inequality, and $\alpha < 1$ we obtain $|F - \Id| \le Ch^\alpha$. This \BBB concludes the proof of \EEE (v).  

\BBB Finally, we show (vi). A careful inspection of the proof of \cite[Theorem 6]{hierarchy} shows that $R(y)(x')$ may be defined as the nearest-point projection onto $SO(3)$ of
$$\PPP \int_I \EEE \int_{ x' + h(-1,1)^2} \frac{1}{h^2}\psi\Big( \frac{x'-z'}{h} \Big) \, \nabla_h y(z',z_3) \, dz' \, dz_3, $$
where $I = (-1/2,1/2)$ and $\psi \in C_c^\infty((-1,1)^2)$ denotes a standard mollifier, i.e., $\psi \ge 0$ and $\int_{(-1,1)^2}\psi=1$. Then (vi) follows from (v).
\end{proof}

By Lemma \ref{lemma:rigidity} we get that $\nabla_h y$ is approximated by the $SO(3)$-valued function $R(y)$. As the energy is invariant under rotations, the energy of $y$ is essentially controlled by the distance of $R(y)^\top \nabla_h y$ from $\Id$. To this end, we introduce the quantity   
\begin{align}\label{eq: G(y)}
G^h(y) := \frac{R(y)^\top \nabla_h y - \Id}{h^2}. 
\end{align}

In the following we set for shorthand $H_Y := \frac{1}{2}\partial^2_{F_1^2} D^2(Y,Y) = \frac{1}{2}\partial^2_{F_2^2} D^2(Y,Y)$  for $Y \in \R^{3 \times 3}$  in a neighborhood of $SO(3)$. Given a deformation $y \in \mathscr{S}^M_h$, we also introduce the mapping $H_{\nabla_h y}: \Omega \to \R^{3 \times 3 \times 3 \times 3}$   by $H_{\nabla_h y}(x) = H_{\nabla_h y(x)}$ for $x \in \Omega$. \BBB Note that this is well defined for $h$ sufficiently small by \eqref{eq:rigidity}(v). \EEE  Recall the definition of $\mathcal{D}_h$ and ${\mathcal{D}}_0$ in \eqref{eq: D,D0-1} and \eqref{eq: D,D0-2}, respectively, and the  definition of $Q_W^3$, $Q_D^3$ in \eqref{eq:Q2}-\eqref{eq:Q22}.

\begin{lemma}[Dissipation and energy]\label{lemma: metric space-properties}
 Let $h$ sufficiently small. Then, for all $y,y_0,y_1 \in \mathscr{S}_h^M$ and all open subsets $U \subset \Omega$  we have  
\begin{align*}
(i) & \ \ \Big|\int_U D^2(\nabla_h y_0,\nabla_h y_1) - \int_U Q_D^3 (\nabla_h y_1 -  \nabla_h y_0) \Big| \le Ch^\alpha \Vert \nabla_h y_1-   \nabla_h y_0 \Vert^2_{L^2(U)}, \notag\\
(ii) & \ \ \BBB \Big|\mathcal{D}_h(y_0,y_1)^2 -  \int_\Omega  Q^3_D \big(G^h(y_0) - G^h(y_1)\big)  \Big| \le C h^{\alpha} \Vert G^h(y_0) - G^h(y_1) \Vert^2_{L^2(\Omega)} \le Ch^{\alpha}, \EEE \notag \\
(iii) & \ \ |\Delta(y)| \le Ch^{\alpha}, \ \ \ \text{where} \ \ \Delta(y) := \frac{1}{h^4}\int_\Omega W(\nabla_h y)  - \int_\Omega  \frac{1}{2} Q_W^3(G^h( y)),\\
 (iv)  & \ \ \BBB |\Delta(y_0) - \Delta(y_1)|   \le Ch^{\alpha}\Vert G^h(y_0) - G^h(y_1) \Vert_{L^2(\Omega)} \le Ch^{\alpha}. \EEE
\end{align*}
\end{lemma}

\begin{proof}
As a preparation, we observe that by the uniform bound on $\nabla_h y_0$, $\nabla_h y_1$ (see \eqref{eq:rigidity}(v)) and  a Taylor expansion \BBB at $(\nabla_h y_0,\nabla_h y_0)$ \EEE we obtain for all open subsets $U \subset \Omega$
\begin{align}\label{eq: Taylor1}
\Big|\int_U D^2(\nabla_h y_0, \nabla_h y_1) - \int_U H_{\nabla_h y_0}[\nabla_h (y_1 -  y_0),\nabla_h (y_1 -   y_0) ] \Big| \le C \Vert \nabla_h (y_1-   y_0) \Vert^3_{L^3(U)}.
\end{align}
\BBB We recall \EEE \eqref{eq: G(y)} and define $G(y_i) \BBB := \EEE h^2G^h(y_i) \BBB = R(y_i)^\top \nabla_h y_i - \Id\EEE$, $i=0,1$, for convenience. Using the separate frame indifference  \eqref{eq: assumptions-D}(v) we have 
$$\int_\Omega D^2(\nabla_h y_0, \nabla_h y_1) = \int_\Omega D^2\big(R(y_0)^\top\nabla_h y_0, R(y_1)^\top \nabla_h y_1\big).$$
 Thus, by \BBB $h^4 \mathcal{D}_h(y_0,y_1)^{2} =  \int_\Omega D^2(\nabla_h y_0, \nabla_h y_1) $ and  \EEE  again by Taylor expansion we also get
\begin{align}\label{eq: Taylor2}
\Big|h^4 \mathcal{D}_h(y_0,y_1)^{\BBB 2 \EEE} - \int_\Omega H_{R(y_0)^\top\nabla_h y_0}& [G(y_1) - G(y_0), G(y_1) - G(y_0)] \Big| \le C \Vert   G(y_1) - G(y_0) \Vert^3_{L^3(\Omega)}. 
\end{align} 

We now show (i). By the regularity of $D$ and \eqref{eq:rigidity}(v)  we get $\Vert H_{\nabla_h y_0} -   \C^3_D   \Vert_\infty \le Ch^\alpha$, where the fourth order tensor  $\C^3_D$ associated to $Q^3_D $ is defined in \eqref{eq: order4}. Therefore, we obtain
$$\Big|\int_U H_{\nabla_h y_0}[\nabla_h (y_1 -  y_0),\nabla_h (y_1 -   y_0) ] - \int_U Q^3_D(\nabla_h y_1 -  \nabla_h y_0) \Big| \le Ch^\alpha \Vert \nabla_h y_1 - \nabla_h y_0 \Vert^2_{L^2(U)}$$
\BBB for all open $U \subset \Omega$. \EEE By using \eqref{eq: Taylor1} and again \eqref{eq:rigidity}(v) we  get (i).

To see (ii), we observe \BBB $\Vert H_{R(y_0)^\top\nabla_h y_0} - \C^3_D\Vert_\infty\le C\Vert\nabla_h y_0 - R(y_0)\Vert_\infty \le Ch^{\alpha}$ by the regularity of $D$ and \eqref{eq:rigidity}(v),(vi).   Thus,  we get  
\begin{align}\label{eq: Taylor3}
\Big|\int_\Omega  H_{R(y_0)^\top\nabla_h y_0}[G(y_1) - G(y_0),G(y_1) - G(y_0) ] & - \int_\Omega Q^3_D \big(G(y_1) - G(y_0) \big) \Big| \notag\\& \le Ch^\alpha \Vert G(y_1) - G(y_0) \Vert_{L^2(\Omega)}^{2}.
\end{align}
In a similar fashion, \eqref{eq:rigidity}(v),(vi) also imply $\Vert G(y_i) \Vert_\infty \le Ch^\alpha$ for $i=0,1$ and thus 
\begin{align*}
\Vert   G(y_1) - G(y_0) \Vert^3_{L^3(\Omega)}  \le C h^\alpha \Vert G(y_1) - G(y_0) \Vert_{L^2(\Omega)}^{2}.
\end{align*}
This together with \eqref{eq: Taylor2}-\eqref{eq: Taylor3} (divided by $h^4$), and $G^h(y_i) = h^{-2}G(y_i)$ for $i=0,1$ yields 
$$
\Big|\mathcal{D}_h(y_0,y_1)^2 -  \int_\Omega  Q^3_D \big(G^h(y_0) - G^h(y_1)\big)  \Big| \le C h^\alpha \Vert G^h(y_0) - G^h(y_1) \Vert^2_{L^2(\Omega)}.
$$
This  shows the first inequality of (ii). To see the second inequality, we use \eqref{eq:rigidity}(i) and \eqref{eq: G(y)}.  \EEE 

We now show (iii) and (iv).  We use the frame indifference of $W$ and   \BBB Lemma \ref{D-lin}(iii) (with $F_i = R(y_i)^\top\nabla_h y_i = \Id + G(y_i)$ for $i=0,1$)   to obtain
\begin{align*}
|\Delta(y_1) - \Delta(y_0)|  & \le Ch^{-4}  \sum\nolimits_{k=1}^3  \int_\Omega |R(y_0)^\top\nabla_h y_0-\Id|^{3-k}|R(y_1)^\top\nabla_h y_1 - R(y_0)^\top\nabla_h y_0|^k \\
&  = Ch^{-4}  \sum\nolimits_{k=1}^3  \int_\Omega|G(y_0)|^{3-k}|G(y_1) - G(y_0)|^k \\
& \le Ch^{-4} \int_\Omega (|G(y_1)| + |G(y_0)|)^2|G(y_1) - G(y_0)|,
\end{align*}
where  $\Delta(y_0)$ and $\Delta(y_1)$ \EEE are defined in the statement of the lemma. The fact that $\Vert G(y_i) \Vert_\infty \le Ch^\alpha$ for $i=0,1$ (see \eqref{eq:rigidity}(v),(vi)) and H\"older's inequality yield
$$|\Delta(y_1) -  \Delta(y_0)| \le Ch^{\alpha-4} \big(\Vert   G(y_0) \Vert_{L^2(\Omega)} + \Vert G(y_1)   \Vert_{L^2(\Omega)} \big) \, \Vert G(y_1) - G(y_0) \Vert_{L^2(\Omega)}. $$
Using $G^h(y_i) = h^{-2}G(y_i)$ for $i=0,1$ and \eqref{eq:rigidity}(i) we obtain the first inequality of (iv). The second inequality follows again by \eqref{eq:rigidity}(i). Finally, to see (iii), we apply (iv)   for $y_0 = y$ and $y_1 = \id$, where we use $\Delta(y_1) = 0$. \EEE 
\end{proof}

\subsection{Metric spaces and their properties}\label{sec: metric}

In this section we  prove that $(\mathscr{S}^M_h, \mathcal{D}_h)$ and  $({\mathscr{S}}_0,{\mathcal{D}}_0)$ are complete metric spaces.   We start with the 3D setting. Recall the definition of  $\phi_h$ and $\mathcal{D}_h$ in \eqref{nonlinear energy-rescale}-\eqref{eq: D,D0-1}. \BBB As a preparation, \EEE  we address the positivity of $\mathcal{D}_h$.

\begin{lemma}[Positivity of $\mathcal{D}_h$]\label{lemma: positivity}
Let $M>0$ and let $h$ sufficiently small.   Let $y_0, y_1 \in \mathscr{S}^M_h$ with $\mathcal{D}_h(y_0,y_1) = 0$. Then $y_0 = y_1$.
\end{lemma}

\begin{proof}
It is convenient to formulate the problem for the original (not rescaled) functions $w_0$ and $w_1$ defined on $\Omega_h = S \times (-\frac{h}{2},\frac{h}{2})$. To explain the main idea, we first assume that $\Omega_h$ is the union of pairwise disjoint cubes of sidelength $h$ up to a set of negligible measure.  Denote the family of cubes by $\mathcal{Q}$. By $\mathcal{Q}_1 \subset \mathcal{Q}$ we denote the cubes whose boundaries share at least one face with $\partial S \times (-\frac{h}{2},\frac{h}{2})$. Let $\mathcal{Q}_2 \subset \mathcal{Q}\setminus \mathcal{Q}_1$ be the cubes whose boundaries share at least one face with a cube in $\mathcal{Q}_1$. In a similar fashion, we define $\mathcal{Q}_i$, $i \ge 2$, and find $\mathcal{Q} = \bigcup_{i=1}^I \mathcal{Q}_i$ for some $I \in \N$. 

We now first show that $w_0 = w_1$ on each $Q \in \mathcal{Q}_1$. To this end, fix $Q \in \mathcal{Q}_1$. From Lemma \ref{lemma: metric space-properties}(i) (in terms of $w_0,w_1$ instead of $y_0,y_1$) we get
\begin{align}\label{eq: positivity1}
\Big|\int_Q D^2(\nabla w_0,\nabla w_1) - \int_Q  Q^3_D (\nabla w_1 -  \nabla w_0)) \Big| \le Ch^\alpha \Vert \nabla  w_1-  \nabla  w_0 \Vert^2_{L^2(Q)}. 
\end{align}
Since $w_1 = w_0$ on $\partial S \times (-\frac{h}{2}, \frac{h}{2})$, we get that $w_1 = w_0$ on at least one face of $\partial Q$. Then Korn's inequality (see, e.g., \cite[Proposition 1]{hierarchy}) implies
$$\Vert \nabla w_1 - \nabla w_0 \Vert^2_{L^2(Q)} \le C \Vert {\rm sym}(\nabla w_1 - \nabla w_0) \Vert^2_{L^2(Q)}. $$
Together with Lemma \ref{D-lin}(ii) this shows
\begin{align}\label{eq: positivity2}
\Vert \nabla w_1 - \nabla w_0 \Vert^2_{L^2(Q)} \le C \int_Q Q^3_D (\nabla w_1 -  \nabla w_0). 
\end{align}
For $h$ sufficiently small, \eqref{eq: positivity1}-\eqref{eq: positivity2} \BBB along \EEE with $\int_Q D^2(\nabla w_0,\nabla w_1) = 0$ show $\nabla w_1 = \nabla w_0$ a.e.\ on $Q$. Since  $w_1 = w_0$ on at least one face of $\partial Q$, this also gives $w_1 = w_0$ a.e.\ on $Q$, as desired. 

We now proceed \BBB iteratively \EEE to show that $w_1 = w_0$ on each $Q \in \mathcal{Q}$: suppose that the property has already been shown for all $Q \in \bigcup_{i=1}^j \mathcal{Q}_i$. Then $w_1 = w_0$ on each $Q \in \mathcal{Q}_{j+1}$ follows from the above arguments noting that $w_1 = w_0$ on at least one face of  $\partial Q$ since $w_1 = w_0$ on all squares  $Q \in \bigcup_{i=1}^j \mathcal{Q}_i$. 

This shows $w_1 = w_0$ on $\Omega_h$ in the case that $\Omega_h$ is the union of pairwise disjoint cubes of sidelength $h$ up to a set of negligible measure. In the general case, we may cover $\Omega_h$ by cubes of sidelength $h$, again denoted by $\mathcal{Q}$, such that for each $Q \in \mathcal{Q}$ the set $Q \cap \Omega_h$ is bilipschitzly equivalent to a cube of sidelength $h$ with a  controlled Lipschitz constant. We note that the constant in Korn's inequality can be chosen uniformly for these sets. We may again decompose $\mathcal{Q}$ into pairwise disjoint families $(\mathcal{Q}_i)_i$ and show \BBB iteratively \EEE  that $w_1$ coincides with $w_0$ on all cubes. 
\end{proof}

\begin{lemma}[Properties of $(\mathscr{S}^M_h, \mathcal{D}_h)$ and $\phi_h$]\label{th: metric space}
\BBB Let $M>0$. \EEE For $h>0$ sufficiently small we have
\begin{itemize}
\item[(i)] $(\mathscr{S}^M_h, \mathcal{D}_h)$ is a complete metric space.
\item[(ii)] Compactness: If $(y_n)_n \subset \mathscr{S}^M_h$, then $(y_n)_n$ admits a subsequence converging weakly in $W^{2,p}(\Omega;\R^3)$ and  strongly in  $W^{1,\infty}(\Omega;\R^3)$.
\item[(iii)]  Topologies: The \BBB topology \EEE induced by $\mathcal{D}_h$ coincides with the weak $W^{2,p}(\Omega;\R^3)$ topology. 
\item[(iv)] Lower semicontinuity: $\mathcal{D}_h(y_n,y) \to 0$ \   \ $\Rightarrow$   \ \ $\liminf_{n \to \infty} \phi_h(y_n) \ge \phi_h(y)$.
\end{itemize}
\end{lemma}

\begin{proof}
\BBB We start with (ii). \EEE We recall   \eqref{assumptions-P}(iii) and \eqref{nonlinear energy-rescale}, and find $\Vert \PPP \nabla^2_h \EEE y \Vert^p_{L^p(\Omega)} \le CMh^{p\alpha}$ for all $y \in \mathscr{S}^M_h$. This together with   \eqref{eq:rigidity}(v)  and the boundary conditions \eqref{eq: nonlinear boundary conditions}    shows  $\sup_{y \in \mathscr{S}^M_h}\Vert y \Vert_{W^{2,p}(\Omega)} < \infty$. Since $p>3$, (ii) follows from a standard compactness argument. 

\BBB We now show (iii). (a) We first suppose that $y_n \rightharpoonup y$ weakly in $W^{2,p}(\Omega;\R^3)$. As $p>3$, this \EEE implies $y_n \to y$ strongly in $W^{1,\infty}(\Omega;\R^3)$ and thus $\mathcal{D}_h(y_n,y) \to 0$ by dominated convergence. \BBB (b) On the other hand, assume that $\mathcal{D}_h(y_n,y) \to 0$.  Item \EEE (ii) yields that a subsequence \BBB (not relabeled) \EEE of $(y_n)_n$ converges weakly in $W^{2,p}(\Omega;\R^3)$ to some $\tilde{y}$. \BBB By (a) we   find $\mathcal{D}_h(y_n,\tilde{y}) \to 0$. \EEE  The triangle inequality, \PPP see \eqref{eq: assumptions-D}(iii), \EEE then shows $\mathcal{D}_h(y,\tilde{y}) \le \lim_{n\to \infty} (\mathcal{D}_h(y,y_n) + \mathcal{D}_h(y_n,\tilde{y})) = 0$. \BBB This yields $y = \tilde{y}$ by Lemma \ref{lemma: positivity}, \EEE  and therefore $y_n \rightharpoonup y$ weakly in $W^{2,p}(\Omega;\R^3)$. \PPP As the limit is independent of the subsequence, the convergence actually holds for the whole sequence. \EEE

The proof of (iv) follows again by \second the dominated convergence theorem \EEE and the convexity of $P$, see \eqref{assumptions-P}(ii).

\BBB Finally, we show (i). \EEE Apart from the positivity, all properties of a metric follow directly from \eqref{eq: assumptions-D} and \eqref{eq: D,D0-1}.  The positivity has has been addressed in Lemma \ref{lemma: positivity}. \BBB It therefore remains to show that $(\mathscr{S}^M_h, \mathcal{D}_h)$ is complete. Let $(y_k)_k \PPP \subset \mathscr{S}_h^M\EEE$ be a Cauchy sequence with respect to $\mathcal{D}_h$. By  (ii) and (iii) we find $y \in W^{2,p}(\Omega;\R^3)$ and a subsequence   (not relabeled)   such that $\lim_{k\to \infty}\mathcal{D}_h(y_k,y) = 0$. By (iv) \BBB and the trace theorem \EEE we get $y \in \mathscr{S}_h^M$. The fact that $(y_k)_k$ is a Cauchy sequence now implies that the whole sequence $y_k$ converges to $y$ with respect to $\mathcal{D}_h$. This concludes the proof. \EEE
\end{proof}

Similar properties can be derived in the   2D setting.   Recall the definition of $\phi_0$ and $\mathcal{D}_0$   in   \eqref{eq: phi0} and \eqref{eq: D,D0-2}, respectively.   For convenience, we will use the notations 
\begin{align}\label{eq: important notation}
e(u) = {\rm sym}(\nabla' u), \ \ \ \ \ \ \ \  B(v_1,v_2) = \frac{1}{2}{\rm sym}(\nabla' v_1 \otimes \nabla' v_2).
\end{align}

\begin{lemma}[Properties of $({\mathscr{S}}_0,  {\mathcal{D}}_0)$ and ${\phi}_0$]\label{th: metric space-lin}
We have:

\begin{itemize}
\item[(i)] $({\mathscr{S}}_0,  {\mathcal{D}}_0)$ is a complete metric space.
\item[(ii)] Compactness: If $(u_n,v_n)_n \subset {\mathscr{S}}_0$ is a sequence with $\sup_n {\phi}_0(u_n,v_n)<+\infty$, then $(u_n,v_n)_n$ is bounded in $W^{1,2}(S;\R^2) \times W^{2,2}(S)$. 
\item[(iii)] Topologies: The metrics ${\mathcal{D}}_0$ and $d((u,v),(u',v') ) := \Vert u - u'\Vert_{W^{1,2}(S)} + \Vert v - v' \Vert_{W^{2,2}(S)}$ are equivalent.  In particular, $\Vert v - v' \Vert_{W^{2,2}(S)} \le C{\mathcal{D}}_0((u,v),(u',v') )$  for $C=C(S)>0$. 
\item[(iv)] Continuity: ${\mathcal{D}}_0( (u_n,v_n), (u,v)) \to 0$ \ \   $\Rightarrow$ \ \  $\lim_{n \to \infty} {\phi}_0 (u_n,v_n) =  {\phi}_0(u,v)$.
\end{itemize}
\end{lemma}

\begin{proof}
We first prove (ii). By Lemma \ref{D-lin}(ii) and the \BBB triangle inequality   we find
$$\Vert e(u) \Vert^2_{L^2(S)} \le C\int_S Q_W^2 \big( e(u) + B(v,v) \big)  + C\Vert B(v,v) \Vert^2_{L^2(S)}$$
for a    universal  constant $C>0$. \EEE This together with Korn's   and Poincar\'e's inequality, and   the fact that $u = \hat{u}$ on $\partial S$ gives
\begin{align}\label{eq: linear comp1}
\Vert  u \Vert^2_{W^{1,2}(S)} &\le  C \Vert   u- \hat{u} \Vert^2_{W^{1,2}(S)} + C\Vert   \hat{u} \Vert^2_{W^{1,2}(S)} \le   C\Vert e(u-\hat{u}) \Vert^2_{L^2(S)} + C\Vert   \hat{u} \Vert^2_{W^{1,2}(S)}\notag\\
&  \le C\Vert   \hat{u} \Vert^2_{W^{1,2}(S)} + C\Vert B(v,v) \Vert^2_{L^2(S)} + C\int_S Q_W^2\big( e(u) + B(v,v)\big).
\end{align}
Now consider $(u,v) \in {\mathscr{S}}_0$ with ${\phi}_0(u,v) \le M$ for $M>0$. Then by Lemma \ref{D-lin}(ii) \BBB and \eqref{eq: phi0} \EEE we find $\Vert (\nabla')^2 v \Vert_{L^2(S)} \le C$,  where the constant depends on $M$.  Since $v = \hat{v}$ and $\nabla' v = \nabla' \hat{v}$ on $\partial S$, \BBB by an argumentation similar to \eqref{eq: linear comp1} \EEE we also get 
\begin{align}\label{eq: linear comp1XXX}
\Vert v \Vert_{W^{2,2}(S)} \BBB \le \Vert v - \hat{v} \Vert_{W^{2,2}(S)} + \Vert \hat{v} \Vert_{W^{2,2}(S)} \le C\Vert (\nabla')^2(v - \hat{v}) \Vert_{L^{2}(S)} + \Vert \hat{v} \Vert_{W^{2,2}(S)} \EEE \le C,
 \end{align} where the constant additionally depends on \BBB $\Vert \hat{v} \Vert_{W^{2,2}(S)}$. \EEE   Then  ${\phi}_0(u,v) \le M$ together with  \eqref{eq: linear comp1} \BBB and the embedding $W^{2,2}  \hookrightarrow W^{1,4}$ (in dimension two) \EEE yields   $\Vert  u \Vert_{W^{1,2}(S)} \le C$, where the constant also depends on $\Vert \hat{u} \Vert_{W^{1,2}(S)}$.  This shows property (ii).

We now address (iii). \BBB As a preparation, consider \EEE $(u_0,v_0), (u_1,v_1) \in {\mathscr{S}}_0$.  Arguing similarly as before, using Lemma \ref{D-lin}(ii) and Poincar\'e's inequality, we find 
\begin{align}\label{eq: linear comp2}
\Vert v_0 - v_1 \Vert_{W^{2,2}(S)}^2 \le \BBB C \EEE  \int_S Q_D^2\big((\nabla')^2  v_0-   (\nabla')^2 v_1 \big).
\end{align}
 Here, we used that $v_0$ and $v_1$ as well as their first derivatives coincide on $\partial S$. Repeating the argumentation leading to \eqref{eq: linear comp1} we find  
\begin{align}\label{eq: linear comp3}
\Vert u_0 - u_1\Vert^2_{W^{1,2}(S)}& \le C\Vert  e( u_0) - e( u_1) \Vert^2_{L^2(S)}\\
&  \le  C\int_S \big(|B(v_0,v_0) - B(v_1,v_1)|^2 +   Q_D^2 \big( e(u_0 - u_1) + B(v_0,v_0) - B(v_1,v_1)\big)\big).\notag
\end{align}
Here, in the first step, we used Korn's and Poincar\'e's inequality. In the second step, \BBB we applied \EEE Lemma \ref{D-lin}(ii)  \BBB  and the triangle inequality. \EEE

\BBB We now suppose that $\mathcal{D}_0((u_n,v_n),(u,v)) \to 0$  (see \eqref{eq: D,D0-2}). Then $v_n \to v$ strongly in $W^{2,2}(S)$ by  \eqref{eq: linear comp2}. This also implies   $v_n \to v$ in $W^{1,4}(S)$ and thus, in view of \eqref{eq: linear comp3}, we find that $u_n \to u$ strongly in $W^{1,2}(S;\R^2)$. Therefore, the \EEE sequence converges with respect to  the metric $d$ \BBB defined in the statement of \EEE  (iii). \BBB Conversely, \EEE given a sequence $(u_n,v_n)_n$ with $u_n \to u$ strongly in $W^{1,2}(S;\R^2)$ and $v_n \to v$ in $W^{2,2}(S)$, we also get $v_n \to v$ in $W^{1,4}(S)$. We observe that \BBB $\mathcal{D}_0$ \EEE  is continuous with respect to this convergence, i.e., ${\mathcal{D}}_0( (u_n,v_n),(u,v) ) \to 0$.   This yields the equivalence of the metrics, and \eqref{eq: linear comp2} shows the second part of (iii).   

The proof of (iv) is similar, noting that also ${\phi}_0$ is continuous with respect to   the topology  induced by $\mathcal{D}_0$, see \eqref{eq: phi0}.  

We finally prove (i). The positivity \BBB and the completeness \EEE follow from  (iii).  Thus, the only thing left to show is the triangle inequality. It can be derived by an elementary computation, using that ${\mathcal{D}}^2_0$ is   the sum of two quadratic forms. We omit the details. 
\end{proof}

\begin{rem}\label{rem: small computation}
{\normalfont

For later purposes, we remark that the proof of (ii) can be generalized: a similar argument shows that for given $(\tilde{u},\tilde{v}) \in {\mathscr{S}}_0$ and $v \in W^{2,2}(S)$  we get
$$\int_S | e(\tilde{u}) + {\rm sym}(\nabla' \tilde{v} \otimes \nabla' v)|^2 + \int_S | (\nabla')^2 \tilde{v} |^2    + \Vert v \Vert_{W^{\BBB 2,2}(S)}   \le M \ \,   \Rightarrow \ \,    \Vert \tilde{u} \Vert_{W^{1,2}(S)}    + \Vert \tilde{v} \Vert_{W^{2,2}(S)} \le C,$$
where $C$ depends on $M$, $\hat{u}$ and $\hat{v}$. This follows by repeating the argument in \eqref{eq: linear comp1}-\eqref{eq: linear comp1XXX} and using $\Vert B(\tilde{v},{v})\Vert_{L^2(S)} \le C\Vert \tilde{v} \Vert_{W^{1,4}(S)}\Vert v \Vert_{W^{1,4}(S)} \le C\Vert \tilde{v} \Vert_{W^{\BBB 2,2}(S)}\Vert v \Vert_{W^{2,2}(S)}$.  
}\end{rem}

\subsection{\BBB Generalized geodesics and properties of slopes in the 2D setting\EEE}\label{sec: convexity}

In this section,  we \BBB first \EEE derive convexity properties for the energy and the dissipation distance \BBB in the 2D setting along generalized geodesics. (We refer to \cite[Section 3.2, Section 3.4]{MOS} for a discussion about generalized geodesics in a related problem.)  Afterwards, we derive fundamental properties for the local slope which will be instrumental to  \EEE use the theory in \cite{AGS}. For $M >0$, we define   the sublevel sets  ${\mathscr{S}}_0^M = \lbrace (u,v) \in {\mathscr{S}}_0: \ {\phi}_0(u,v) \le M \rbrace$.

 \begin{lemma}[Convexity and generalized geodesics in the 2D setting]\label{th: convexity2}
 Let $M >0$. Then there exist smooth functions $\Phi^1,\Phi^2_M:[0,\infty) \to [0,\infty)$ satisfying   $\lim_{t \to 0}\Phi^1(t)/t =1$ and $\lim_{t \to 0}\Phi^2_M(t)/t =0$ such that for all $(u_0,v_0) \in {\mathscr{S}}_0^M$ \BBB and all $(u_1,v_1) \in {\mathscr{S}}_0$ \EEE we have  
\begin{align*}
(i)& \ \    \mathcal{D}_0\big((u_0,v_0), (u_s,v_s)\big)  \le  s \Phi^1\big(\mathcal{D}_0\big((u_0,v_0), (u_1,v_1)\big)\big), \\
(ii) & \ \ \phi_0(u_s,v_s)  \le (1-s) \phi_0(u_0,v_0) + s\phi_0(u_1,v_1) +s\Phi^2_M\big(\mathcal{D}_0\big((u_0,v_0), (u_1,v_1)\big)\big),  
\end{align*}
where $u_s := (1-s) u_0 + su_1$ and  $v_s := (1-s) v_0 + sv_1$,  $s \in [0,1]$. 
\end{lemma}

 \begin{proof}
 For convenience, we first introduce some abbreviations and provide some preliminary estimates. We let $D = \mathcal{D}_0((u_0,v_0), (u_1,v_1))$. We  use the notation defined in \eqref{eq: important notation} and also set $B_{\rm diff} = B(v_0-v_1,v_0-v_1)$. Particularly,   by Lemma \ref{th: metric space-lin}(iii) and a Sobolev embedding   we observe that 
  \begin{align}\label{eq: calcu-v}
\Vert B_{\rm diff} \Vert_{L^2(S)}  \le C \Vert \nabla' v_0 - \nabla' v_1 \Vert^2_{L^4(S)} \le C\Vert v_0 - v_1 \Vert^2_{W^{2,2}(S)}  \le CD^2.  
\end{align}
For brevity, \BBB we also \EEE   introduce
\begin{align}\label{eq: H-1}
G^s_0 = e(u_s ) + B(v_s,v_s) 
\end{align}
for $s \in [0,1]$. (The notation $G_0$ is borrowed from \cite{hierarchy}, see also Lemma \ref{lemma: compacti2} below.)  
With the definition of $\mathcal{D}_0$ in \eqref{eq: D,D0-2} and Lemma \ref{D-lin}(ii) we get
\begin{align}\label{eq: H-2}
\Vert G_0^1 - G_0^0 \Vert_{L^2(S)}^2 \le C \int_S Q^2_D(G_0^1-G^0_0) \le \BBB C \EEE D^2.
\end{align} 
Similarly, we observe \BBB by \eqref{eq: phi0}  (with $f \equiv 0$), \EEE Lemma \ref{D-lin}(ii), and the fact that $(u_0,v_0) \in \mathscr{S}^M_0$ 
\begin{align}\label{eq: H-3}
\Vert G^0_0 \Vert_{L^2(S)}^2 \le C \int_S Q_W^2(G^0_0) \le C\phi_0(u_0,v_0)\le CM.
\end{align}

We now start with the proof of (i).  First, we observe  
$$  \frac{1}{12}   \int_S Q^2_D( (\nabla')^2 v_s - (\nabla')^2 v_0 )  =  s^2   \frac{1}{12}   \int_S Q^2_D ( (\nabla')^2 v_1 - (\nabla')^2 v_0 ).$$ 
We will show that \BBB there exists $C>0$ independent of $s$ such that \EEE
\begin{align}\label{eq: suffices to show1}
\int_S  Q^2_D(G^s_0 - G^0_0) \le  s^2 \int_S  Q^2_D(G^1_0-G_0^0) + C s^2 D^3 + Cs^2 D^4 
\end{align}
\BBB for $s \in [0,1]$. \EEE Then  recalling the definition of $\mathcal{D}_0$ in \eqref{eq: D,D0-2},   (i) follows for the function $\Phi^1(t) = \sqrt{t^2 + Ct^3 + Ct^4}$.  To show \eqref{eq: suffices to show1},   recalling \eqref{eq: important notation}, we obtain by an elementary expansion  
\begin{align*}
  B( v_s,v_s)  - B(v_0,v_0) & =  s \big( B( v_1,v_1 )  -  B( v_0,v_0 ) \big) - s(1-s)\big( B( v_0,v_0) +   B( v_1, v_1 )  -2 B( v_0, v_1) \big) \notag\\
&= s (B(v_1,v_1 ) -  B( v_0,v_0)) - s(1-s) B( v_0 - v_1, v_0 - v_1) \notag \\ 
&= s ( B(v_1,v_1)   -  B(v_0,v_0)) - s(1-s) B_{\rm diff},  
\end{align*}
where the last equality follows from the definition of $B_{\rm diff}$. By recalling \eqref{eq: H-1}   this implies
\begin{align}\label{eq: expand-v}
G_0^s-G^0_0  = s \big( G_0^1 - G^0_0 - (1-s) B_{\rm diff}\big).  
\end{align}
Recalling also $\C^2_D$ in  \eqref{eq: order4},   an expansion   and the Cauchy-Schwartz inequality   then yield
 \begin{align*}
 \int_S  Q^2_D(G^s_0-G^0_0)   & = s^2 \int_S   Q^2_D (G_0^1-G^0_0) +  s^2(1-s)^2  \int_S  Q^2_D(B_{\rm diff})  \\ & \ \ \ - 2s^2(1-s) \int_S\C^2_D[  G_0^1-G^0_0, B_{\rm diff} ]\\
 & \le  s^2 \int_S   Q^2_D(G^1_0-G^0_0) + Cs^2 \Vert B_{\rm diff} \Vert_{L^2(S)}^2  + Cs^2 \Vert G_0^1-G^0_0 \Vert_{L^2(S)} \Vert B_{\rm diff}\Vert_{L^2(S)} .
  \end{align*}
 By using \eqref{eq: calcu-v} and \eqref{eq: H-2} we now see that \eqref{eq: suffices to show1} holds. This concludes the proof of (i).

\BBB Recall the definition of $\phi_0$ in \eqref{eq: phi0}. \EEE We show (ii) for the function $\Phi^2_M(t) = C\sqrt{M}t^2 + \BBB Ct^3 + \EEE Ct^4$ for some $C>0$. Due to convexity of   $s \mapsto \int_S \frac{1}{24} \EEE Q^3_W( \BBB (\nabla')^2 \EEE v_s)$, it suffices to show
 \begin{align}\label{eq: suffices to show2}
\frac{1}{2}   \int_S  Q_W^2(G^s_0) \, dx' &\le    \frac{1-s}{2} \int_S  Q_W^2(G_0^0)  +    \frac{s}{2}  \int_S  Q_W^2(G^1_0)    +C\sqrt{M}s D^2 \BBB +CsD^3\EEE + Cs^2D^4.
\end{align}  
In view of  \eqref{eq: order4} and  \eqref{eq: expand-v}, an elementary expansion yields  
\begin{align*}
Q_W^2(G^s_0) &= Q^2_W( G^0_0 + G^s_0 - G^0_0)=  Q_W^2\big( (1-s) G^0_0 +s G^1_0 - s(1-s) B_{\rm diff}\big) \\
& = (1-s)Q^2_W(G^0_0) + sQ^2_W(G^1_0) - (1-s)s Q^2_W(G^0_0-G^1_0) \\
& \ \ \ -  2s(1-s)^2C^2_W[G^0_0,B_{\rm diff}] -  2s^2(1-s) C^2_W[G^1_0,B_{\rm diff}] + s^2(1-s)^2 Q_W^2(B_{\rm diff}).
 \end{align*}
 Then  taking the integral  we obtain
 \begin{align*}
 \int_S   Q_W^2(G^s_0)  &\le (1-s) \int_S   Q_W^2(G^0_0)   + s \int_S   Q_W^2(G^1_0)   \\ &  \ \ \  +Cs\big( \Vert G^0_0 \Vert_{L^2(S)}   + \Vert G^1_0 \Vert_{L^2(S)}  \big)\Vert B_{\rm diff} \Vert_{L^2(S)}    + Cs^2  \Vert B_{\rm diff} \Vert^2_{L^2(S)}.
\end{align*}
\BBB By using $\Vert G^1_0 \Vert_{L^2(S)} \le \Vert G^0_0 \Vert_{L^2(S)} + \Vert G^1_0 - G^0_0 \Vert_{L^2(S)}$, \EEE  \eqref{eq: calcu-v}, \EEE and \eqref{eq: H-2}-\eqref{eq: H-3} we get \eqref{eq: suffices to show2}. This concludes the proof of (ii).  
 \end{proof}

 Now we derive representations and properties of the local slope corresponding to ${\phi}_0$. \BBB Recall the notation ${\mathscr{S}}_0^M = \lbrace (u,v) \in {\mathscr{S}}_0: \ {\phi}_0(u,v) \le M \rbrace$. \EEE

\begin{lemma}[\BBB Local slope in the 2D setting\EEE]\label{lemma: slopes}
\BBB Let $M>0$. \EEE The local slope for the energy ${\phi}_0$ admits the representation 
$$ |\partial {\phi}_{\BBB 0 \EEE }|_{{\mathcal{D}}_0}(u,v) = \underset{(u',v') \neq (u,v)}{\sup_{(u',v') \in {\mathscr{S}}_0}} \   \frac{\Big({\phi}_0(u,v) - {\phi}_0(u',v') - \Phi^2_M\big({\mathcal{D}}_0((u,v),(u',v')) \big) \Big)^+}{\Phi^1\big({\mathcal{D}}_0((u,v),(u',v'))\big)} $$
\BBB for all $(u,v) \in {\mathscr{S}}^M_0$, \EEE  where $\Phi^1$ and $\Phi^2_M$ are the functions from Lemma \ref{th: convexity2}. The slope is a strong upper gradient for ${\phi}_0$.
\end{lemma}

\Proof \BBB We follow \EEE the lines of the proofs of Theorem 2.4.9 and Corollary 2.4.10 in \cite{AGS}. Let \BBB $M>0$ and $(u,v) \in \mathscr{S}_0^M$.  \EEE Let $\Phi^1,\Phi^2_M$ be the functions from Lemma \ref{th: convexity2}, and recall that $\lim_{t \to 0}\Phi^1(t)/t =1$ and $\lim_{t \to 0}\Phi^2_M(t)/t =0$.  We recall the definition of the local slope in Definition \ref{main def2} and obtain 
\begin{align*}
|\partial {\phi}_0|_{{\mathcal{D}}_0}(u,v) & = \limsup_{(u',v') \to (u,v)} \frac{({\phi}_0(u,v) - {\phi}_0(u',v'))^+}{{\mathcal{D}}_0((u,v),(u',v'))} \\&
= \limsup_{(u',v') \to (u,v)} \frac{\Big({\phi}_0(u,v) - {\phi}_0(u',v') - \Phi^2_M\big({\mathcal{D}}_0((u,v),(u',v'))\big)\Big)^+}{\Phi^1\big({\mathcal{D}}_0((u,v),(u',v'))\big)}   \\
&\le \sup_{(u',v') \neq (u,v)} \  \frac{\Big({\phi}_0(u,v) - {\phi}_0(u',v') - \Phi^2_M\big({\mathcal{D}}_0((u,v),(u',v'))\big)\Big)^+}{\Phi^1\big({\mathcal{D}}_0((u,v),(u',v'))\big)},
\end{align*}
where the supremum is taken over functions in $\mathscr{S}_0$.  In the second equality we used  that $(u',v') \to (u,v)$ means ${\mathcal{D}}_0((u,v),(u',v')) \to 0$, and the fact that   $\lim_{t \to 0}\Phi^1(t)/t =1$ and $\lim_{t \to 0}\Phi^2_M(t)/t =0$.

 To see the other inequality, \BBB we fix $(u',v') \neq (u,v)$. It \EEE is not restrictive to suppose that 
\begin{align*}
\phi_0(u,v) - \phi_0(u',v') - \Phi^2_M\big({\mathcal{D}}_0((u,v),(u',v'))\big)>0.
\end{align*}
Define $u_s = (1-s)u + s u'$ and $v_s = (1-s)v + s v'$ for $s \in [0,1]$. \BBB By \EEE Lemma \ref{th: convexity2} with $(u_0,v_0) = (u,v)$ and $(u_1,v_1) = (u',v')$  we get
 \begin{align*}
\frac{\phi_0(u,v) - \phi_0(u_s,v_s)}{\mathcal{D}_0((u_0,v_0),(u_s,v_s))}  &\ge \frac{ s{\phi}_0(u,v) - s{\phi}_0(u',v') - s\Phi^2_M\big({\mathcal{D}}_0((u,v),(u',v'))\big) }{s\Phi^1\big({\mathcal{D}}_0((u,v),(u',v'))\big)}. 
\end{align*}
Note $(u_s,v_s) \to (u,s)$ as $s \to 0$ with respect to the topology induced by ${\mathcal{D}}_0$, \PPP see Lemma \ref{th: metric space-lin}(iii). \EEE \BBB Therefore,
$$|\partial {\phi}|_{{\mathcal{D}}_0}(u,v)   \ge \frac{ {\phi}_0(u,v) - {\phi}_0(u',v') -\Phi^2_M\big({\mathcal{D}}_0((u,v),(u',v'))\big) }{\Phi^1\big({\mathcal{D}}_0((u,v),(u',v'))\big)}. $$
The claim \EEE now follows by taking the supremum with respect to $(u',v')$.

It remains to show that $|\partial {\phi}_0|_{{\mathcal{D}}_0}$ is a strong upper gradient. Let us first recall from
\cite[Lemma 1.2.5]{AGS} that in  the complete metric space \BBB $({\mathscr{S}}_0,{\mathcal{D}}_0)$, \EEE the global slope
\begin{align}\label{global slope}
\mathcal{G}_{{\phi}_0}(u,v) := \underset{(u',v') \neq (u,v)}{\sup_{(u',v') \in {\mathscr{S}}_0}} \frac{({\phi}_0(u,v) - {\phi}_0(u',v'))^+}{{\mathcal{D}}_0((u,v),(u',v'))}
\end{align}
is a strong upper gradient for ${\phi}_0$ since ${\phi}_0$ is ${\mathcal{D}}_0$-lower semicontinuous (see Lemma \ref{th: metric space-lin}(iv)). Moreover,  \cite[Lemma 1.2.5]{AGS} also states  that     $|\partial {\phi}_0|_{{\mathcal{D}}_0}$ is a weak upper gradient for ${\phi}_0$ in the sense of \cite[Definition 1.2.2]{AGS}. We do not repeat the definition of weak upper gradients, but only mention that weak upper gradients are also strong upper gradients if for each absolutely continuous curve $z:(a,b) \to {\mathscr{S}}_0$ \BBB with respect to $\mathcal{D}_0$ \EEE satisfying  $|\partial {\phi}_0|_{{\mathcal{D}}_0}(z)|z'|_{{\mathcal{D}}_0} \in L^1(a,b)$, the function ${\phi}_0 \circ z$ is absolutely continuous. 

To check that ${\phi}_0 \circ z$ is absolutely continuous, we first extend the  curve $z$ continuously to $[a,b]$ and introduce the compact metric space $\mathscr{S}' = z([a,b])$ with the metric induced by  $\mathcal{D}_0$. \BBB Choose $M>0$ such that $\mathscr{S}'  \subset \mathscr{S}_0^M$, which is possible due to \eqref{eq: phi0}, \eqref{eq: D,D0-2}, and the fact that ${\rm diam}(\mathscr{S}'):= \sup_{s,t \in [a,b]} \mathcal{D}_0(z(s),z(t)) < +\infty$. \EEE  Let  $\mathcal{G}'_{{\phi}_0}$ be the global slope as introduced in \eqref{global slope} with respect to $\mathscr{S}'$ (instead of ${\mathscr{S}}_0$).  The representation of the local slope implies
$$\mathcal{G}'_{{\phi}_0}(u,v) =\underset{(u',v') \neq (u,v)}{\sup_{(u',v') \in \mathscr{S}'}}  \frac{({\phi}_0(u,v) - {\phi}_0(u',v'))^+}{{\mathcal{D}}_0((u,v),(u',v'))} \le C_1|\partial {\phi}_0|_{{\mathcal{D}}_0}(u,v) + C_2 $$
for all $(u,v) \in \mathscr{S}' \BBB \subset \mathscr{S}_0^M \EEE $, where
$$C_1 := \sup_{t \in [0,{\rm diam}(\mathscr{S}')]}\frac{\Phi^1(t)}{t} < + \infty, \ \ \  \  \ \ C_2 := \sup_{t \in [0,{\rm diam}(\mathscr{S}')]}\frac{\Phi^2_M(t)}{t} < + \infty.  $$ 
In particular, since $|\partial {\phi}_0|_{{\mathcal{D}}_0}(z)|z'|_{{\mathcal{D}}_0} \in L^1(a,b)$, it follows $\BBB\mathcal{G}'_{{\phi}_0} \EEE (z)|z'|_{{\mathcal D}_0} \in L^1(a,b)$.  As discussed above, \BBB $\mathcal{G}'_{{\phi}_0}$ \EEE is a strong upper gradient. Thus,  we indeed get that ${\phi}_0 \circ z$ is absolutely continuous, see Definition \ref{main def2}. 
\eop

\section{Relation between 3D and 2D setting}\label{Sec:relation2D3D}

In this section we consider sequences $(y^h)_h$ with $y^h \in \mathscr{S}^M_h =   \lbrace y \in \mathscr{S}_h: \ \phi_h(y) \le  M \rbrace$. \BBB In Section \ref{sec: compactness-hierarchy} we \EEE derive compactness properties and identify suitable limiting objects in terms of scaled in-plane and out-of-plane displacement fields. \BBB Afterwards, in Section \ref{sec: gamma} we derive a $\Gamma$-convergence result  and prove lower semicontinuity for the local slopes along the passage from the 3D to the 2D setting. \EEE As in Section \ref {sec:energy-dissipation}, $C>0$  indicates a generic constant which is independent of $h$, but may depend on $M$, $S$,   $p>3$, \BBB $\alpha \in (0,1)$, \EEE  on the constants in \eqref{assumptions-W},  \eqref{assumptions-P}, \eqref{eq: assumptions-D}, and the boundary data $\hat{u}$ and $\hat{v}$. \EEE

\subsection{Compactness and identification of limiting strain}\label{sec: compactness-hierarchy}

 Given $y^h \in \mathscr{S}_h^M$, we define the \PPP averaged, \EEE scaled in-plane and out-of-plane displacement fields by 
\begin{align}\label{eq: out-plane-strain}
u^h(x') := \frac{1}{h^2} \int_I \Big( \begin{pmatrix}
y_1^h\\ y_2^h \end{pmatrix} (x',x_3) - \begin{pmatrix}
x_1\\ x_2 \end{pmatrix} \Big) \, dx_3, \ \ \ \ v^h(x') := \frac{1}{h} \int_I y_3^h(x',x_3)\, dx_3,
\end{align}
where $I = (-\frac{1}{2},\frac{1}{2})$. By recalling the boundary conditions \eqref{eq: nonlinear boundary conditions} and \eqref {eq: BClinear} we \BBB observe \EEE  $(u^h, v^h) \in {\mathscr{S}}_0$.

In view of  the rigidity estimate in Lemma \ref{lemma:rigidity},  \BBB a deformation  $y^h \in \mathscr{S}_h^M$ \EEE is close to the affine map $(x', x_3) \mapsto (x', h x_3)$   and thus the displacements defined in \eqref{eq: out-plane-strain} are suitably controlled. More specifically, we have the following    compactness result.

\begin{lemma}[Compactness for displacements]\label{lemma: compacti}
Consider a sequence $(y^h)_h$ with $y^h \in \mathscr{S}^M_h$ for all $h$. Then up to passing to a  subsequence (not relabeled) we find 
$(u,v) \in {\mathscr{S}}_0$  such that 
\begin{align}\label{eq: convergence u,v}
& u^h \rightharpoonup u \ \text{ weakly in} \ W^{1,2}(S;\R^2), \notag \\
& v^h \to v \  \text{ strongly in} \ W^{1,2}(S).
\end{align}
\end{lemma}

For the proof  we refer to \cite[Lemma 1]{hierarchy} and \cite[Lemma 13]{lecumberry}.

\begin{rem}[Ansatz for recovery sequences]\label{rem: compatible}
{\normalfont

Note that the convergence results in Lemma \ref{lemma: compacti} are compatible with the ansatz 
\begin{align}\label{eq: ansatzi}
y^h(x',x_3) = \begin{pmatrix}
x' \\ h x_3
\end{pmatrix}
+
 \begin{pmatrix}
h^2{u}(x') \\ h {v}(x') 
\end{pmatrix}
-h^2x_3
 \begin{pmatrix}
(\nabla' {v}(x') )^\top \\ 0
\end{pmatrix}
+ h^3x_3 d^h(x'),
\end{align}
for $(u,v) \in {\mathscr{S}}_0$ and $d^h \in W^{2,\infty}_0(S;\R^3)$, which additionally satisfy the regularity
\begin{align}\label{eq: regularity-new} 
\Vert u \Vert_{W^{2,\infty}(S)} + \Vert v \Vert_{W^{3,\infty}(S)} + \BBB \sqrt{h} \EEE \Vert d^h \Vert_{W^{2,\infty}(S)} \le C'.  
\end{align}
We observe that the boundary conditions \eqref{eq: nonlinear boundary conditions} are satisfied,  i.e., $y^h \in \mathscr{S}_h$.   For later purposes, we note that the scaled gradient is given by
\begin{align}\label{eq: recovery sequence-derivative}
\nabla_h y^h & = \Id +  
 \begin{pmatrix} \begin{array}{c|c}
h^2 \nabla' u & -h(\nabla' v)^\top \\ \hline h \nabla' v & 0
\end{array}\end{pmatrix}
-h^2x_3
 \begin{pmatrix} \begin{array}{c|c}
(\nabla')^2 v & 0 \\ \hline 0 & 0 
\end{array}\end{pmatrix} 
  +h^2  d^h \otimes e_3 + h^3x_3 (\nabla' d^h \, | \, 0).
\end{align}
We also note that the (scaled) second gradient   has the form 
$$
(\nabla')^2 y^h  = \begin{pmatrix}
h^2(\nabla')^2 u - h^2x_3 (\nabla')^3 v \\ h (\nabla')^2 v  
\end{pmatrix} + h^3 x_3   (\nabla')^2 d^h,   \ \
\frac{1}{h}\nabla' y^h_{,3} = \begin{pmatrix}
- h (\nabla')^2 v \\ 0 \end{pmatrix} +  h^2\nabla' d^h
$$
and $\frac{1}{h^2}y_{,33}^h = 0$. \BBB By \eqref{eq: regularity-new}   this \EEE implies $\Vert (\nabla')^2 y^h\Vert_{\BBB \infty}\le Ch$,  $\Vert \nabla' (\frac{1}{h} y^h_{,3})\Vert_\infty \le Ch$ and  $y_{,33}^h = 0$ for a constant $C$ depending on $C'$. Consequently, by using $\alpha < 1$ \BBB and \EEE \eqref{assumptions-P}(iii) we compute 
\begin{align}\label{eq: newP}
\limsup_{h \to \infty} \frac{1}{h^{\alpha p}} \int_\Omega P(\nabla_h^2 y^h) &\le \limsup_{h \to \infty}\frac{C}{h^{\alpha p}} \int_\Omega \Big(\Big|(\nabla')^2 y^h\Big|^p + \Big|  \nabla' \Big(\frac{1}{h}  y^h_{,3}\Big)\Big|^p + \Big|   \frac{1}{h^2}  y^h_{,33} \Big|^p\Big)\notag\\
& \le \lim_{h\to 0} Ch^{p(1-\alpha)} = 0. 
\end{align}
\BBB Later \EEE in Theorem \ref{th: Gamma} and Theorem \ref{theorem: lsc-slope} we will use this ansatz to construct recovery sequences for the energy and the dissipations. We remark that, on the one hand, our ansatz is slightly simpler than the one in \cite[(119)]{hierarchy} since we consider a model with  zero Poisson's ratio  in $e_3$ direction,  see Remark \ref{rem: Poisson1} and Remark \ref{rem: Poisson2} for more details. On the other hand, some additional effort arises from the fact that our \BBB 3D model \EEE contains second gradient terms.
}
\end{rem}

In \eqref{eq: G(y)} we have already introduced the quantity $G^h(y^h) = h^{-2}(R(y^h)^\top \nabla_h y^h - \Id)$  which essentially controls the energy of $y^h$. By \eqref{eq:rigidity}(i) we get that $G^h(y^h)$ converges weakly in $L^2$ (up to a subsequence) to some $G$. The next result shows that the symmetric part of the in-plane components of $G$ can be identified in terms of the in-plane displacement $u$ and the  out-of-plane displacement $v$.

\begin{lemma}[Identification of scaled limiting strain]\label{lemma: compacti2}
Consider the \BBB setting \EEE of Lemma \ref{lemma: compacti}.  \BBB Let $R(y^h): S \to SO(3)$ be  the corresponding mappings from Lemma \ref{lemma:rigidity}. \EEE Then
\begin{align}\label{eq: A convergence}
\frac{1}{h} (R(y^h) - \Id)  \to A(v) := e_3 \otimes \nabla' v - \nabla' v \otimes e_3 \ \text{ in } \ L^q( S;  \R^{3 \times 3}), \  1\le q < \infty. 
\end{align}
Moreover, we find $G \in L^2(\Omega;\R^{3 \times 3})$ such that
\begin{align*}
 G^h(y^h) = \frac{R(y^h)^\top \nabla_h y^h - \Id}{h^2} \rightharpoonup G \ \text{ \PPP weakly \EEE in } \ L^2(\Omega;\R^{3 \times 3}). 
 \end{align*}
The $2 \times 2$ submatrix $G''$ given by $G''_{ij} = G_{ij}$ for $i,j \in \lbrace 1, 2 \rbrace$ satisfies
$$G''(x',x_3) = {G}_0(x') + x_3{G}_1(x') $$
with 
\begin{align*}
\BBB {\rm sym}({G}_0) \EEE = e( u) + \frac{1}{2} \nabla' v \otimes \nabla' v   \ \ \text{and} \ \ {G}_1 = - (\nabla')^2 v.
\end{align*}
\end{lemma}

Here, we again use the notation $e(u) = {\rm sym}(\nabla' u)$, \BBB see \eqref{eq: important notation}. \EEE For the proof we refer to \cite[Lemma 1, Lemma 2]{hierarchy} \BBB (see also \cite[Lemma 16]{lecumberry}) \EEE from which we also adopted the notation for ease of readability.   \BBB Note that, in fact, \EEE an inspection of the proof shows that the mappings $R(y^h)$ introduced there can be chosen as the   ones from  Lemma \ref{lemma:rigidity}.    \EEE  In particular, the result shows that the relevant components of $G$ are affine in the thickness variable $x_3$.   In the following, we will also use the notation  
\begin{align}\label{eq: Gnot}
G(u,v)(x',x_3) =  \BBB {\rm sym}({G}_0) \EEE (x') + x_3{G}_1(x') =e(u) + \frac{1}{2} \nabla' v \otimes \nabla' v  -  x_3(\nabla')^2 v. 
\end{align}
Note that $G(u,v) \in L^2(\Omega;\R^{2 \times 2}_{\rm sym})$.

 \begin{rem}
 {\normalfont
Using \eqref{eq: Gnot} and the fact that $Q_W^2$ only depends on the symmetric part of matrices (see Lemma \ref{D-lin}(ii)), the energy  ${\phi}_0(u,v)$ defined in \eqref{eq: phi0}  (with $f \equiv 0$)  can be written as
\begin{align}\label{eq: G and phi}
{\phi}_0(u,v) & = \int_S \Big(\frac{1}{2}Q^2_W(G_0) + \frac{1}{24}Q_W^2(G_1)  \Big)\notag \\
& =  \int_S \int_{-1/2}^{1/2 }\frac{1}{2}Q^2_W(G_0(x') + x_3 G_1(x'))\, dx_3 \, dx'   = \int_\Omega \frac{1}{2}Q^2_W(G(u,v)).
\end{align} 
In view of the definition of ${\mathcal{D}}_0$ in \eqref{eq: D,D0-2},  a similar computation yields 
\begin{align}\label{eq: G and D}
\int_\Omega Q^2_D\big(G(u_1,v_1) - G(u_2,v_2)\big)  = {{\mathcal D}}_0\big( (u_1,v_1), (u_2,v_2)\big)^2. 
\end{align}
In the following we will mainly use the representation in terms of an integral over $\Omega$ as it is convenient for many proofs. Only at the very end we will pass to an integral over $S$ as indicated in \eqref{eq: G and phi}-\eqref{eq: G and D}.
}
\end{rem}

The following lemma about strong convergence of \BBB strain differences will be instrumental \EEE in the proof of the lower semicontinuity of the local slopes.  \BBB We introduce the notation $W^{2,p}_{0,\partial S}(\Omega;\R^3) = \lbrace y \in W^{2,p}(\Omega;\R^3): y(x',x_3)  = 0  \text{ for } x' \in \partial S, \ x_3 \in I \rbrace$, where $I=(-\frac{1}{2},\frac{1}{2})$.  We   also define ${\rm skew}(F) = \frac{1}{2}(F - F^\top)$ for $F \in \R^{3 \times 3}$.    \EEE 

%

\begin{lemma}[Strong convergence of \BBB strain differences\EEE]\label{lemma: strong convergence}
Let $M>0$. Let $(y^h)_h$ be a sequence with $y^h \in \mathscr{S}^M_h$ and let  $(z^h_s)_{s,h} \subset W^{2,p}_{0, \BBB \partial S}(\Omega;\R^3)$, $h>0$, $s \in (0,1)$,  be functions such that 
\begin{align}\label{eq: strong convergence assumptions}
(i) & \ \  \Vert \nabla_h z^h_s \Vert_{L^{\infty}(\Omega)} + \Vert \nabla^2_h z^h_s \Vert_{L^{\infty}(\Omega)} \le   M   s h. \notag \\
(ii) & \ \ \Vert {\rm sym}(\nabla_h z^h_s) \Vert_{L^{2}(\Omega)} \le   M   s h^2,\notag\\
(iii) & \ \ \BBB \big|  {\rm skew} (\nabla_h z^h_s)(x',x_3)  -  \int_{I} {\rm skew}(\nabla_h z^h_s)(x',x_3) \, dx_3 \big|   \le Msh^{5/2} \EEE \ \ \ \text{for a.e.\ $x \in \Omega$},\notag\\
(iv) & \ \ \PPP \text{there exist $E^s, F^s \in L^2(\Omega;\R^{3 \times 3}) $ for $s\in(0,1)$, and $\eta(h) \to 0$ as $h \to 0$  such that  } \notag \\
& \ \  \Vert h^{-2}{\rm sym}(\nabla_h z^h_s) - E^s \Vert_{L^2(\Omega)}  +  \Vert h^{-1}{\rm skew}(\nabla_h z^h_s) - F^s \Vert_{L^2(\Omega)}  \le s\eta(h). \EEE
\end{align}
Then the following holds for a subsequence of $(y^h)_h$ (not relabeled):

(a) For all $h$ sufficiently small,  $w^h_s := y^h + z^h_s$ lies in $\mathscr{S}^{M'}_h$ for some $M' = M'(M)>0$.  

(b) Let $(G^h(y^h))_h$, $(G^h(w^h_s))_h$ be the sequences in Lemma \ref{lemma: compacti2}   and let $G_y$, $G^s_w$ be their limits.   \BBB Then there exists $C=C(M)>0$ and $\rho(h)$ with $\rho(h)\to 0$ as $h \to 0$ such that   
\begin{align*}
(i) & \ \ \big\|  \big(G^h(y^h) - G^h(w^h_s)\big) -   \big(G_y - G^s_w  \big)\big\| _{L^2(\Omega)} \le s\rho(h),\notag \\
(ii) & \ \  \Vert G^h(y^h) - G^h(w^h_s)\Vert_{L^2(\Omega)} \le Cs.
\end{align*}

(c) \EEE Let $(u,v)$ and $(\bar{u}_s,\bar{v}_s)$ be limits corresponding to $y^h$ and $w^h_s$, respectively, as given in Lemma \ref{lemma: compacti}.  Then ${\rm sym}(G^s_w-G_y ) \, e_3 = E^s\, e_3 - \frac{1}{2}(|\nabla' {v}|^2 - |\nabla' \bar{v}_s|^2) e_3$ \BBB a.e.\ in $\Omega$. \EEE 
\end{lemma}

 \BBB
We note that weak convergence of the strain differences is already guaranteed by Lemma \ref{lemma: compacti2}. The important point here is that we actually obtain strong convergence \PPP with linear control in terms of $s$, \EEE see (b). Moreover, (c) provides a characterization of the out-of-plane components of the limiting strain \PPP difference.  Later in the proof of Theorem \ref{theorem: lsc-slope} we will use this lemma to construct competitor sequences for the local slope in the 3D setting. \EEE

\begin{proof}
 \BBB Let $R(y^h)$ be the $SO(3)$-valued mappings given by   Lemma \ref{lemma:rigidity}. For brevity, we introduce notations for the symmetric and skew-symmetric part of $\nabla_h z^h_s$ by
 $$E(z^h_s) = {\rm sym}(\nabla_h z^h_s), \ \ \ \ \  F(z^h_s) =  {\rm skew}(\nabla_h z^h_s), \ \ \ \ \  {\thickbar F}(z^h_s) = \int_I F(z^h_s) \, dx_3.$$
 The crucial point is to find a suitable $SO(3)$-valued mapping $R(w^h_s)$ associated to  $w^h_s = y^h + z^h_s$ satisfying the properties stated in Lemma \ref{lemma:rigidity} (Step 1). Once $R(w^h_s)$ has been defined, we can prove properties (a)-(c) (Step 2).

\emph{Step 1: Definition of $R(w^h_s)$.} We first define 
$$\tilde{R} =   R(y^h) \big(\Id + {\thickbar F}(z^h_s) - \tfrac{1}{2} {\thickbar F}(z^h_s)^\top {\thickbar F}(z^h_s)\big) $$ 
on $S$. By \eqref{eq:rigidity}(iii) and  \eqref{eq: strong convergence assumptions}(i) we can check that $\tilde{R}$ is in a small tubular neighborhood of $SO(3)$  and satisfies $\Vert \nabla' \tilde{R} \Vert^2_{L^2(S)} \le Ch^2$. We let $R(w^h_s) \in W^{1,2}(S;SO(3))$ be the map obtained from $\tilde{R}$ by nearest-point projection on $SO(3)$, see \cite[Remark 5]{hierarchy} for a similar argument. By \eqref{eq: strong convergence assumptions}(i) and  ${\thickbar F}(z^h_s)(x') \in \R^{3\times 3}_{\rm skew}$ for all $x' \in S$, it is elementary to check that  
$$
\big\| R(w^h_s) -  R(y^h) \big(\Id + {\thickbar F}(z^h_s) - \tfrac{1}{2} {\thickbar F}(z^h_s)^\top {\thickbar F}(z^h_s)\big)   \big\|_{L^\infty(S)}  \le C\Vert {\thickbar F}(z^h_s)  \Vert^3_{L^\infty(S)} \le  Csh^3. 
$$
Indeed, this follows from the fact that $|((\Id + A -\frac{1}{2}A^\top A)^\top (\Id + A -\frac{1}{2}A^\top A))^{\PPP 1/2 \EEE } - \Id| \le C|A|^3$ for all $A \in \R^{3 \times 3}_{\rm skew}$. Along with \eqref{eq: strong convergence assumptions}(iii) we thus get
\begin{align}\label{eq: w-rotat}
\big\| R(w^h_s) -   R(y^h) \big(\Id + F(z^h_s) - \tfrac{1}{2} F(z^h_s)^\top F(z^h_s)\big)   \big\|_{L^\infty(\Omega)}  \le Csh^{5/2}. 
\end{align}  

We now check that $R(w^h_s)$ satisfies the properties stated in Lemma \ref{lemma:rigidity}, see \eqref{eq:rigidity}. First, $\Vert \nabla' \tilde{R} \Vert^2_{L^2(S)} \le Ch^2$ implies $\Vert \nabla' R(w^h_s)  \Vert^2_{L^2(S)} \le Ch^2$. Moreover, \eqref{eq: strong convergence assumptions}(i), \eqref{eq: w-rotat},    and the fact that $R(y^h)$ satisfies \eqref{eq:rigidity}(iv) shows   $\Vert R(w^h_s) -\Id  \Vert_{L^q(S)}\le C_qh$ \PPP for $q \in [1,\infty)$. \BBB In a similar fashion, \eqref{eq:rigidity}(vi) \PPP and $\alpha <1$ yield \BBB  $\Vert R(w^h_s) -\Id  \Vert_{L^\infty(S)}\le Ch^\alpha$. It thus remains to check \PPP \eqref{eq:rigidity}(i), i.e., that \BBB  
\begin{align}\label{eq: last thing to check}
\Vert R(w^h_s)^\top \nabla_h w^h_s - \Id \Vert^2_{L^2(\Omega)}\le Ch^4
\end{align}
holds. For notational convenience, we denote by  $\omega^h_i \in L^2(\Omega;\R^{3 \times 3})$, $i \in \N$,  (generic) matrix valued functions whose $L^2$-norm is controlled in terms of a constant independent of $h$ and $s$.  By  \eqref{eq:rigidity}(v) \PPP (applied for $y^h$), \BBB \eqref{eq: strong convergence assumptions}(i), and \eqref{eq: w-rotat} we find 
\begin{align}\label{eq: new-strain}
R(w^h_s)^\top \nabla_h w^h_s  = \big(  \Id + F(z^h_s) - \tfrac{1}{2} F(z^h_s)^\top F(z^h_s) \big)^\top R(y^h)^\top (\nabla_h y^h + \nabla_h z^h_s) + sh^{5/2} \omega_1^h. 
\end{align}
We now consider the asymptotic expansion of $R(y^h)^\top (\nabla_h y^h + \nabla_h z^h_s)$  in terms of $h$: in view of \eqref{eq: strong convergence assumptions}(ii), $\nabla_h z^h_s  = E(z^h_s) +  F(z^h_s)$,    and $\Vert R(y^h) - \Id \Vert_{L^\infty(\Omega)} \le Ch^\alpha$ (see \eqref{eq:rigidity}(vi)) we find  
$$R(y^h)^\top (\nabla_h y^h + \nabla_h z^h_s) = R(y^h)^\top (\nabla_h y^h +  F(z^h_s)) + E(z^h_s) +  h^{2+\alpha}s \omega^h_2.$$
In a similar fashion, by  \PPP using \eqref{eq:rigidity}(i),(iv) (for $y^h$)  and \eqref{eq: strong convergence assumptions}(i),(ii), \BBB  we compute 
$$R(y^h)^\top (\nabla_h y^h + \nabla_h z^h_s) = \Id + R(y^h)^\top  F(z^h_s) + h^2 \omega^h_3 =  \Id +  F(z^h_s) + h^{2} \omega^h_4$$
as well as 
$$ R(y^h)^\top (\nabla_h y^h + \nabla_h z^h_s)   =  \Id  +   h \omega^h_5.$$ 
By inserting these three estimates in \eqref{eq: new-strain} and using \eqref{eq: strong convergence assumptions}(i) we find
\begin{align}\label{eq: new-strain2}
R(w^h_s)^\top \nabla_h w^h_s &  = R(y^h)^\top (\nabla_h y^h + F(z^h_s)) + E(z^h_s) +  F(z^h_s)^\top     + F(z^h_s)^\top F(z^h_s)\\
&   \ \ \ - \tfrac{1}{2} F(z^h_s)^\top F(z^h_s) + s(h^{5/2} + h^{2+\alpha}) \,\omega^h_6\notag \\
& = R(y^h)^\top \nabla_h y^h + (R(y^h) - \Id)^\top F(z^h_s)  + E(z^h_s) + \tfrac{1}{2} F(z^h_s)^\top F(z^h_s) + s\notag \omega^h_6 {\rm o}(h^{2}), 
\end{align}
where in the last step we used $ F(z^h_s)^\top +  F(z^h_s) = 0$. We now check that \eqref{eq: last thing to check} holds. Indeed, it suffices to use \eqref{eq: new-strain2}, \eqref{eq: strong convergence assumptions}(i),(ii),  $\omega_6^h \in L^2(\Omega;\R^{3\times 3})$,    and the fact that \eqref{eq:rigidity}(i),(iv) holds for $y^h$. In conclusion, this implies that the mapping $R(w^h_s)$ satisfies the properties stated in Lemma \ref{lemma:rigidity}.

\emph{Step 2: Proof of the statement.} We are now in a position to prove the statement. \EEE

(a) \PPP By   \eqref{eq:rigidity}(v)  and \eqref{eq: strong convergence assumptions}(i), $\nabla_h w^h_s$ is in a neighborhood of $\Id$. Thus, by \EEE \eqref{assumptions-W}(iii) there holds $\int_\Omega W(\nabla_h w^h_s) \le C\int_\Omega \dist^2(\nabla_h w^h_s,SO(3))$, and then by \eqref{eq: last thing to check} we get    $\int_\Omega W(\nabla_h w^h_s) \le  M' h^4$ for some $M'$ sufficiently large   (depending on $M$). Moreover, from \eqref{assumptions-P}(iii) and the triangle inequality we get $P(\nabla^2_h w^h_s) \le CP(\nabla^2_h y^h) + CP(\nabla^2_h z^h_s)$. Therefore, by possibly passing to a larger $M'$, we obtain  $h^{-\alpha p }\int_\Omega P(\nabla^2_h w^h_s) \le M'$ by \PPP \eqref{assumptions-P}(iii), \eqref{eq: strong convergence assumptions}(i), $\alpha <1$, \EEE and the fact that  $y^h \in \mathscr{S}_h^M$. Summarizing, since $z^h_s  \in W^{2,p}_{0,  \partial S}(\Omega;\R^3)$ and thus $w^h_s$ also satisfies the boundary conditions \eqref{eq: nonlinear boundary conditions}, we have shown that $w^h_s \in \mathscr{S}^{M'}_h$ for some $M'=M'(M)>0$    independent of $h$. In particular, this implies that the statement of Lemma \ref{lemma:rigidity} holds for $w^h_s$ with $R(w^h_s)$ as defined in Step 1.

(b) \BBB   By  $(u,v)$ we denote the limit corresponding to $y^h$ as given in Lemma \ref{lemma: compacti}.  Recalling the definition $G^h(y^h) = h^{-2}(R(y^h)^\top \nabla_h y^h -\Id)$ we find  by \eqref{eq: A convergence}, \PPP \eqref{eq: strong convergence assumptions}(iv), \EEE and \eqref{eq: new-strain2} that 
\begin{align}\label{eq: c2}
\big\|\big(G^h(w^h_s) - G^h(y^h)\big) - \big( A(v)^\top F^s + E^s +\tfrac{1}{2} (F^s)^\top F^s \big)\big\|_{L^2(\Omega)} \le   s{\rho}(h),
\end{align}
where ${\rho}(h) \to 0$ as $\rho \to 0$. By $G^s_w$ and $G_y$ we denote the weak $L^2$-limits of $G^h(w^h_s)$ and $G^h(y^h)$, respectively, which exist by  Lemma \ref{lemma: compacti2}. Then  \eqref{eq: c2}  implies 
\begin{align}\label{eq: c3}
G^s_w - G_y = A(v)^\top F^s + E^s +\tfrac{1}{2} (F^s)^\top F^s,
\end{align}
and the first part of (b) holds.  The second part of (b) is a consequence of \eqref{eq:rigidity}(iv) \PPP (for $y^h$), \EEE \eqref{eq: strong convergence assumptions}(i),(ii),  \eqref{eq: new-strain2}, and the fact that $\omega^h_6 \in L^2(\Omega;\R^{3 \times 3})$.

(c)  By  $(\bar{u}_s,\bar{v}_s)$ we denote the limit corresponding to $w^h_s$ as given in Lemma \ref{lemma: compacti}. By \eqref{eq: A convergence} (for $w^h_s$ and $y^h$, respectively), \eqref{eq: strong convergence assumptions}(i),(iv), and  \eqref{eq: w-rotat}  we observe that pointwise a.e.\ in $\Omega$ there holds
$$A(\bar{v}_s) = \lim_{h \to 0} \frac{1}{h}(R(w^h_s) - \Id) =  \lim_{h \to 0} \Big(  \frac{1}{h}(R(y^h) - \Id)(\Id + F(z^h_s)) + \frac{1}{h} F(z^h_s) \Big) = A(v) + F^s.   $$
  Then by  \eqref{eq: c3} and an expansion we get 
  \begin{align*}
{\rm sym}(G^s_w - G_y) & = E^s + {\rm sym} (A(v)^\top F^s)  +\tfrac{1}{2} (F^s)^\top F^s  \\
&= E^s + \tfrac{1}{2} (A(\bar{v}_s))^\top A(\bar{v}_s) -  \tfrac{1}{2} (A({v}))^\top A({v}) = E^s - \tfrac{1}{2} (A(\bar{v}_s))^2 +  \tfrac{1}{2} (A({v}))^2,
\end{align*}
where in the last step we used that $A(v) \in \R^{3 \times 3}_{\rm skew}$ pointwise a.e.\ in $\Omega$ and thus $A(v)^\top A(v) = -(A(v))^2$. \EEE
Therefore, recalling the definition of $A(v)$ in \eqref{eq: A convergence} we obtain 
$${\rm sym}( G^s_w - G_y) \, e_3 = E^s\, e_3  +  \tfrac{1}{2}(A(v))^2\, e_3 -  \tfrac{1}{2}(A(\bar{v}_s))^2 \, e_3  = E^s\, e_3 -  \tfrac{1}{2} (|\nabla' {v}|^2 - |\nabla' \bar{v}_s|^2) e_3.$$
This concludes the proof. 
\end{proof}

\subsection{$\Gamma$-convergence and lower semicontinuity of slopes}\label{sec: gamma}

In this section we establish a $\Gamma$-convergence result for the energies which \PPP is essentially proved in \EEE    \cite{hierarchy, lecumberry}. However, some adaptions are necessary due to the second order perturbation $P$ in the energy. For an exhaustive treatment of $\Gamma$-convergence we refer the reader to \cite{DalMaso:93}. Afterwards, we prove lower semicontinuity for the dissipation distances and the local slopes which is fundamental to use the theory in \cite{Ortner,S2}  (see also Section \ref{sec: auxi-proofs}, in particular \eqref{compatibility} and \eqref{eq: implication}).    

We first fix a topology for the convergence of the scaled in-plane and out-of-plane displacements induced by the compactness result in Section \ref{sec: compactness-hierarchy}, see \eqref{eq: convergence u,v}. We define mappings $\pi_h: \BBB \mathscr{S}_h \EEE \to {\mathscr{S}}_0$ by $\pi_h(y^h) = (u^h,v^h)$ for each $y^h \in   \mathscr{S}_h$, where $u^h$ and $v^h$ are the  scaled in-plane and out-of-plane displacements corresponding to $y^h$ (see \eqref{eq: out-plane-strain}).   We say 
\begin{align}\label{eq: convergence u,v-2}
\pi_h(y^h) = (u^h,v^h)  \stackrel{\sigma}{\to} (u,v) \ \ \ \text{if $u^h \rightharpoonup u$ in $W^{1,2}(S;\R^2)$  \ and \  $v^h \to v$   in $W^{1,2}(S)$.}
\end{align}
We also say $y^h \stackrel{\pi\sigma}{\to} (u,v)$ if $\pi_h(y^h)  \stackrel{\sigma}{\to} (u,v)$, cf.\ \eqref{eq: sigma'}. \BBB Recall the definitions \eqref{nonlinear energy-rescale} and \eqref{eq: phi0}. \EEE

\begin{theorem}[$\Gamma$-convergence]\label{th: Gamma}
Suppose that $W$ and $P$ satisfy the assumptions \eqref{assumptions-W} and \eqref{assumptions-P}. Then $\phi_h$ converges to ${\phi}_0$ in the sense of $\Gamma$-convergence. More precisely,\\

\noindent (i) (Lower bound) For all $(u,v) \in {\mathscr{S}}_0$ and all sequences $(y^h)_h$ such that  $y^h \stackrel{\pi\sigma}{\to} (u,v)$ we find
$$\liminf_{h \to 0}\phi_h(y^h) \ge {\phi}_0(u,v). $$

\noindent (ii) (Optimality of lower bound) For all $(u,v) \in {\mathscr{S}}_0$ there exists a sequence $(y^h)_h$, $y^h \in \mathscr{S}_h$ \BBB for all $h$, \EEE   such that  $y^h \stackrel{\pi\sigma}{\to} (u,v)$ and  
$$\lim_{h \to 0}\phi_h(y^h) = {\phi}_0(u,v).$$

\end{theorem}

\begin{proof}
(i) The result is essentially proved in \cite{lecumberry}. We give here the main steps for convenience of the reader.  By the   representation \eqref{eq: G and phi} and the fact that $P$ is nonnegative, for the lower bound it suffices to prove 
\begin{align}\label{eq: desired}
\liminf_{h \to 0} h^{-4}\int_\Omega W(\nabla_h y^h) \ge  \int_\Omega \frac{1}{2} Q^2_W(G(u,v)),
\end{align}
for all sequences $(y^h)_h$ with $y^h \stackrel{\pi\sigma}{\to} (u,v)$. We may suppose that \PPP $\liminf_{h \to 0}\phi_h(y^h)$  \EEE is finite as otherwise there is nothing to prove. Thus, we can assume that $y^h\in \mathscr{S}^M_h$ for $M>0$ large enough. By  Lemma \ref{lemma: metric space-properties}(iii) we find 
$$\liminf_{h \to 0} h^{-4}\int_\Omega W(\nabla_h y^h) = \liminf_{h \to 0} \int_\Omega \frac{1}{2} Q^3_W\big(G^h(y^h)\big).$$ 
Moreover, Lemma \ref{lemma: compacti} and Lemma \ref{lemma: compacti2} imply that, possibly passing to a  subsequence (not relabeled), $G^h(y^h) \rightharpoonup G$ in $L^2(\Omega;\R^{3 \times 3})$, where the $2 \times 2$ submatrix $G''$ satisfies $\BBB {\rm sym}(G'') \EEE =G(u,v)$, \BBB see \eqref{eq: Gnot}. \EEE This along with   the lower semicontinuity in $L^2$ (note that $Q_W^3$ is a positive semidefinite quadratic form) yields $\liminf_{h \to 0} h^{-4}\int_\Omega W(\nabla_h y^h) \ge \int_\Omega \frac{1}{2} Q^3_W(G)$. The fact that  $\BBB {\rm sym}(G'') \EEE = G(u,v)$ and \eqref{eq:Q2} give the desired lower bound \eqref{eq: desired}. 

(ii) By a general approximation  argument in the theory of $\Gamma$-convergence it suffices to establish the optimality of the lower bound only for sufficiently smooth mappings, precisely for $u \in W^{2,\infty}(S;\R^2)$ and $v \in W^{3,\infty}(S)$ with
\begin{align}\label{eq: BC}
u = \hat{u}, \ \  \ v = \hat{v}, \ \  \ \nabla' v  = \nabla' \hat{v} \ \ \  \text{ on }\partial S.
\end{align}
Indeed, a general \BBB $(u,v)$ can be approximated strongly in $W^{1,2}(S;\R^2) \times W^{2,2}(S)$ by such functions, \EEE see Lemma \ref{lemma: density}  below.  Note that the limiting energy ${\phi}_0$ is continuous with respect to this topology, see Lemma \ref{th: metric space-lin}(iii),(iv). 

Let us now construct recovery sequences for $u \in W^{2,\infty}(S;\R^2)$ and $v \in W^{3,\infty}(S)$ satisfying \eqref{eq: BC}. Define $d = -\frac{1}{2}|\nabla' v|^2 e_3 \in W^{2,\infty}(S;\R^3)$ and let $(d^h)_h \subset W^{2,\infty}_0(S;\R^3)$ with   $d^h \to d =  -\frac{1}{2}|\nabla' v|^2 e_3$ in $L^2(S;\R^3)$ and $\sup_h  \BBB \sqrt{h} \EEE \Vert d^h \Vert_{W^{2,\infty}(S)} < \infty$.  We take the ansatz for $y^h$ as given in Remark \ref{rem: compatible}. In Remark \ref{rem: compatible} we have already discussed that \BBB this ansatz is compatible with the convergence  \EEE $y^h \stackrel{\pi\sigma}{\to} (u,v)$. \EEE  By the representation of the scaled gradient in \eqref{eq: recovery sequence-derivative}, an elementary computation yields for the nonlinear strain 
\begin{align*}
(\nabla_h y^h)^\top \nabla_h y^h & = \Id + 2h^2 \big( e( u) - x_3 (\nabla')^2 v \big)  + h^2 \big(\nabla' v \otimes \nabla' v + |\nabla' v|^2 e_3 \otimes e_3 \big) \\ & \ \ \  + 2h^2{\rm sym}(d^h \otimes e_3)   + \BBB {\rm O}(h^{5/2}), \EEE
\end{align*}
\BBB  where we used $\sup_h  \sqrt{h}   \Vert d^h \Vert_{W^{2,\infty}(S)} < \infty$. \EEE Since there holds  $d^h \to d =  -\frac{1}{2}|\nabla' v|^2 e_3$ in $L^2(S;\R^3)$ and $\sup_h  \BBB \sqrt{h} \EEE \Vert d^h \Vert_{L^\infty(\Omega)} < \infty$,  we get
\begin{align*}
(\nabla_h y^h)^\top \nabla_h y^h & = \Id + 2h^2 \big( e( u) + \tfrac{1}{2}\nabla' v \otimes \nabla' v - x_3 (\nabla')^2 v \big)       + h^2 \omega^h
\end{align*}
for functions $\omega^h: \Omega \to \R^{3 \times 3}$ with $\Vert \omega ^h \Vert_{L^2(\Omega)} \to 0$ and $\sup_h \BBB \sqrt{h} \EEE \Vert \omega^h \Vert_{L^\infty(\Omega)} < \infty$. Therefore, \EEE
$$ (\nabla_h y^h)^\top \nabla_h y^h= \Id + 2h^2   G(u,v)^*   + h^2 \omega^h$$
where $G(u,v)^* \in L^2(\Omega; \R^{3 \times 3})$ denotes the mapping with $(G(u,v)^*)_{ij} = (G(u,v))_{ij}$ for $1 \le i,j \le 2$ and zero otherwise, see \eqref{eq: Gnot}.  Taking the square root, using the frame indifference of $W$, \BBB \eqref{assumptions-W}(i), \EEE and a Taylor expansion, we derive    (cf.\ also \cite[Proposition 19]{lecumberry})  
 \begin{align*}
 \frac{1}{h^4}\int_\Omega W(\nabla_h y^h) =  \frac{1}{h^4}\int_\Omega W\Big( \big((\nabla_h y^h)^\top \nabla_h y^h\big)^{1/2} \Big) \to \int_\Omega \frac{1}{2}Q_W^3(G(u,v)^*)
 \end{align*}
as $h \to 0$. Definition \eqref{eq:Q2} (and the assumption that the minimum is attained for $a=0$) \BBB yield \EEE $Q_W^3(G(u,v)^*) = Q_W^2(G(u,v))$. Then \eqref{eq: G and phi} implies $ \frac{1}{h^4}\int_\Omega W(\nabla_h y^h) \to {\phi}_0(u,v)$. The proof is now concluded by observing $\lim_{h \to 0} h^{-\alpha p} \int_\Omega P(\nabla_h^2 y^h) =  0$, see \eqref{eq: newP}.
\end{proof}

 \begin{rem}\label{rem: Poisson1}
{\normalfont

We remark that the assumption on $Q_W^2$ is actually not needed at the expense of a more involved recovery sequence, see \cite[equation (119)]{hierarchy}. However, \BBB the assumption \EEE will be instrumental for the lower semicontinuity of slopes, see Theorem \ref{theorem: lsc-slope} and Remark \ref{rem: Poisson2}.

}
 
 \end{rem}

 In the previous proof we have used the following density result. 
 
\begin{lemma}[Density of smooth functions with same boundary conditions]\label{lemma: density}
For each $(u,v) \in {\mathscr{S}}_0$ we find sequences $(u^h)_h \subset W^{2,\infty}(S;\R^2)$ and $(v^h)_h \subset W^{3,\infty}(S)$  such that 
\begin{align*}
(i)& \ \ u^h = \BBB \hat{u}, \EEE \ \  \ v^h = \hat{v}, \ \  \ \nabla' v^h  = \nabla' \hat{v} \ \ \ \text{ on }\partial S,\\
(ii) & \ \  {u}^h \to u \ \text{ in } \ W^{1,2}(S;\R^2), \ \ \  {v}^h \to v \ \text{ in } \ W^{2,2}(S).
\end{align*}

\end{lemma}

\begin{proof}
The proof is standard: we approximate $u-\hat{u}$ and $v-\hat{v}$ by smooth functions with compact support in $S$ and add $\hat{u} \in W^{2,\infty}(S;\R^2)$, $\hat{v} \in W^{3,\infty}(S)$, respectively. 
\end{proof}

We now proceed with the lower semicontinuity of the dissipation distances.   Recall the definitions in  \eqref{eq: D,D0-1} and \eqref{eq: D,D0-2}.

\begin{theorem}[Lower semicontinuity of dissipation distances]\label{th: lscD}
Suppose that $D$ satisfies the assumptions \eqref{eq: assumptions-D}. Let $M>0$. Then for sequences $(y_1^h)_h$ and  $(y^h_2)_h$, $y_1^h,y_2^h \in \mathscr{S}^M_h$, with  $y^h_1 \stackrel{\pi\sigma}{\to} (u_1,v_1)$ and $y^h_2 \stackrel{\pi\sigma}{\to} (u_2,v_2)$  we have
$$\liminf_{h \to 0} \mathcal{D}_h(y_1^h,y_2^h ) \ge {\mathcal{D}}_0 \big ((u_1,v_1),(u_2,v_2) \big).$$
\end{theorem}

\begin{proof}
The argument is \BBB similar to the one \EEE in \eqref{eq: desired}, see the proof of Theorem \ref{th: Gamma}(i), with the difference that we employ Lemma \ref{lemma: metric space-properties}(ii) in place of Lemma \ref{lemma: metric space-properties}(iii) and \eqref{eq: G and D} in place of \eqref{eq: G and phi}.
\end{proof}

 We close this section with the fundamental property that the local slopes are lower semicontinuous along the \PPP passage from the 3D to the 2D setting. \EEE As emphasized before, this is crucial for the application of the theory in \cite{Ortner, S2}, \BBB see \eqref{eq: implication}. \EEE   Recall the definition of $Q_W^2$, $Q_D^2$ in \eqref{eq:Q2}-\eqref{eq:Q22}. \BBB The fact that the minimum is attained for $a = 0$ implies
\begin{align}\label{eq: defA}
Q_W^3(F) =  Q_W^2(F'') + Q_W^3(F - F^*), \ \ \ \ \ Q_D^3(F) =  Q_D^2(F'') + Q_D^3(F-F^*)
\end{align}
for all $F \in \R^{3 \times 3}$, where $F''$ denotes the $2 \times 2$ matrix with entries $F''_{ij} = F_{ij}$ for $1 \le i,j \le 2$, and $F^*$ denotes the $3 \times 3$ matrix with entries $F^*_{ij} = F_{ij}$ for $1 \le i,j \le 2$, and zero otherwise.   \EEE

\begin{theorem}[Lower semicontinuity of slopes]\label{theorem: lsc-slope}
For each sequence $(y^h)_h$ with $y^h \in  {\mathscr{S}}_{h}^M$ such that $y^h \stackrel{\pi\sigma}{\to}  (u,v)$ we have 
$$\liminf_{n \to \infty}|\partial {\phi}_{h}|_{{\mathcal{D}}_{h}}(y^h) \ge |\partial  {\phi}_0|_{ {\mathcal{D}}_0}(u,v).$$ 
\end{theorem}

\begin{proof} We divide the proof into several steps. We first define approximations of $(u,v)$ which allow us to work with more regular functions (Step 1). We then construct  \emph{competitor sequences} $(w^h_s)_{h,s}$ \EEE for the local slope in the 3D setting  satisfying $w^h_s \to y^h$ as $s \to 0$    (Step 2). Afterwards, \PPP we identify the limiting strain of the sequences $(w^h_s)_h$ (Step 3), \EEE and  we \EEE prove the lower semicontinuity (Step 4). Some technical estimates are contained in Steps 5--7.

\emph{Step 1: Approximation.} By Lemma \ref{lemma: density}, for $\eps>0$ we fix $u_\eps \in  W^{2,\infty}(S;\R^2)$ and $v_\eps \in  W^{3,\infty}(S)$  with $(u_\eps,v_\eps) \in {\mathscr{S}}_0$ and
\begin{align}\label{eq: eps approx}
\Vert u_\eps - u \Vert_{W^{1,2}(S)} + \Vert v_\eps - v \Vert_{W^{2,2}(S)}  \le \eps .
\end{align}
This approximation will be necessary to construct sufficiently regular  competitor sequences \EEE for the local slope of the 3D setting.

We further fix $\tilde{u}\in W^{2,\infty}(S;\R^2)$ and $\tilde{v}\in W^{3,\infty}(S)$ with $(\tilde{u},\tilde{v}) \in {\mathscr{S}}_0$, and satisfying $\tilde{u} \neq u,u_\eps$, $\tilde{v} \neq v,v_\eps$.   The pair $(\tilde{u},\tilde{v})$ will represent the  competitor in the local slope of the 2D setting, see Lemma \ref{lemma: slopes}. Below in \eqref{eq: reg is enough}, we will see that by approximation  it is enough to work with functions of this regularity.   The convex combinations 
\begin{align}\label{eq: convexi}
(\tilde{u}_{s}, \tilde{v}_{s}) := (1-s)(u_\eps,v_\eps) + s(\tilde{u},\tilde{v}), \ \  s \in [0,1],
\end{align}
will be the starting point for the construction of competitor sequences \PPP $(w^h_s)_{h,s}$ \EEE for the 3D setting.  In the following, $\tilde{C}, C_\eps$ denote generic constant which may vary from line to line, where $\tilde{C}$ may depend on $\tilde{u},u,\tilde{v},v$, and $C_\eps$ additionally on $\eps$. \EEE

\emph{Step 2: Construction of competitor  sequences $(w^h_s)_{h,s}$.\EEE} We choose recovery sequences $y_\eps^h, \tilde{y}_s^h$ related to $(u_\eps,v_\eps)$ and $(\tilde{u}_{s},\tilde{v}_{s})$, exactly as in the proof of Theorem \ref{th: Gamma}(ii):  define $d_\eps = -\frac{1}{2}|\nabla' v_\eps|^2 e_3$, $\tilde{d}_s = -\frac{1}{2}|\nabla' \tilde{v}_{s}|^2 e_3$ and let $(d^h_\eps)_h, (\tilde{d}^h_s)_h \subset W^{2,\infty}_0(S;\R^3)$ be  sequences with $d^h_\eps \to d_\eps$ and $\tilde{d}^h_s\to \tilde{d}_s$  in  $L^2(S;\R^3)$. \PPP In view of \eqref{eq: convexi}, \EEE this can be done in such a way that there holds
\begin{align}\label{eq: uniform}
\Vert \tilde{d}^h_s -  {d}^h_\eps\Vert_{L^2(S)} \le s\rho(h), \ \ \ \ \ \ \ \ \  \sqrt{h} \Vert \tilde{d}^h_s -  {d}^h_\eps \Vert_{W^{2,\infty}(S)} \le C_\eps s,
\end{align}
where $\rho(h)$ depends on $v_\eps$, $\tilde{v}$, and satisfies $\rho(h) \to 0$ as $h \to 0$. 

We take the ansatz for $y^h_\eps, \tilde{y}_s^h$ as given in Remark \ref{rem: compatible}  and observe that   $y^h_\eps, \tilde{y}^h_s$ satisfy the boundary conditions, i.e., ${z}^h_s := \tilde{y}_s^h -y_\eps^h \in W^{2,p}_{0,  \partial S}(\Omega;\R^3)$. \PPP For $h>0$ small and $s \in [0,1]$, \EEE we define  
\begin{align}\label{eq: ulitmate w-def}
w^h_s := y^h + z^h_s= y^h - y_\eps^h + \tilde{y}_s^h.
\end{align}
By \eqref{eq: uniform}, \eqref{eq: recovery sequence-derivative}-\eqref{eq: newP} (with $\tilde{u}_s,u_\eps$ and  $\tilde{v}_s,v_\eps$ in place of $u$ and $v$, respectively) and the fact that $(\tilde{u}_s-u_\eps, \tilde{v}_s-v_\eps )=  s(\tilde{u}-u_\eps, \tilde{v}-v_\eps )$ we see
\begin{align}\label{eq: w1}
(i) & \ \  \Vert \nabla_h z^h_s\Vert_{L^{\infty}(\Omega)}  +  \Vert \nabla^2_h z^h_s\Vert_{L^{\infty}(\Omega)}  \le C_\eps sh,   \ \ \ \ \ \    \Vert {\rm sym}(\nabla_hz^h_s) \Vert_{L^2(\Omega)} \le C_\eps sh^2,\notag\\
(ii) & \ \ \big|  {\rm skew} (\nabla_h z^h_s)(x',x_3)  -  \int_{I} {\rm skew}(\nabla_h z^h_s)(x',x_3) \, dx_3 \big|   \le C_\eps sh^{5/2} \EEE \ \ \ \text{for a.e.\ $x \in \Omega$}.  
\end{align}
\PPP This shows that the assumptions   \eqref{eq: strong convergence assumptions}(i)-(iii) are satisfied for $(z^h_s)_{s,h}$ (for a constant $M=M(C_\eps)$). From \eqref{eq: recovery sequence-derivative} and \eqref{eq: uniform} we also get that \eqref{eq: strong convergence assumptions}(iv) holds for suitable $E^s$ and $F^s$. In particular, by the definition of $\tilde{d}^h_s$ and $d^h_\eps$ we observe that \EEE
\begin{align}\label{eq:Z3}
\PPP E^s \,  e_3 = \EEE \lim_{h\to 0} \frac{1}{h^2}{\rm sym}(\nabla_hz^h_s) \, e_3 = \lim_{h \to 0} {\rm sym}\big((\tilde{d}^h_s - d^h_\eps) \otimes e_3 \big) \,  e_3 =   \frac{1}{2}(|\nabla' v_\eps|^2 - |\nabla' \tilde{v}_{s}|^2) \, e_3.   
\end{align}
 Then  Lemma \ref{lemma: strong convergence}(a) implies $w^h_s \in \mathscr{S}^{M'}_h$ for a constant $M'>0$  sufficiently large depending on $\eps$,  but independent of $s,h$. Thus, by   \eqref{eq: ansatzi}, \eqref{eq: ulitmate w-def},  $\tilde{d}^h_s \to  {d}^h_\eps$ in $L^2(S;\R^3)$ as $s \to 0$ (see \eqref{eq: uniform}), and a compactness argument (see Lemma \ref{th: metric space}(ii)), one can check that 
\begin{align}\label{eq: sconv}
w_s^h \rightharpoonup y^h \ \  \ \text{in} \ \ \ W^{2,p}(\Omega;\R^3) \  \text{ as $s \to 0$}.
\end{align}

\PPP \emph{Step 3: Identification of limiting strains.} \EEE 
Since the ansatz for $(y_\eps^h)_h$ and $(\tilde{y}_s^h)_h$  is compatible with  the convergence results in Lemma \ref{lemma: compacti},   the convergence in \eqref{eq: convergence u,v-2} holds, i.e., the scaled  displacement fields corresponding to $(y_\eps^h)_h$ and $(\tilde{y}_s^h)_h$ converge to $(u_\eps,v_\eps)$ and $(\tilde{u}_{s}, \tilde{v}_{s})$, respectively. Thus, \PPP in view of \eqref{eq: ulitmate w-def}, \EEE the scaled  displacement fields corresponding to $w^h_s$ converge to $(u - u_\eps + \tilde{u}_{s}, v- v_\eps+ \tilde{v}_{s})$. In the following, it will be convenient to work with  the convex combination defined by 
\begin{align}\label{eq: convexi2}
(\hat{u}_{s}^\eps, \hat{v}_{s}^\eps) :=  (u,v) + s(\tilde{u}-u_\eps,\tilde{v}-v_\eps),  \ \ \ s \in [0,1].
\end{align} 
In fact, by \eqref{eq: convexi} we see that the scaled  displacement fields corresponding to $w^h_s$ converge to $(\hat{u}_{s}^\eps, \hat{v}_{s}^\eps)$.

The limits of the mappings $G^h(y^h)$ and $G^h(w^h_s)$ given by Lemma \ref{lemma: compacti2} are denoted by $G_y$ and  $G^s_w$,   i.e, we have (up to a subsequence)
\begin{align}\label{eq: weakcovi}
G^h(y^h)  \rightharpoonup G_y, \ \ \ \ \ \ \ G^h(w^h_s) \rightharpoonup G^s_w  \ \ \ \ \ \text{weakly in $L^2(\Omega;\R^{3 \times 3})$},
\end{align}
where the $2 \times 2$ submatrices $G_y''$ and  $(G^s_w)''$ satisfy  
\begin{align}\label{eq: two by two}
  {\rm sym}(G''_y) \EEE =  G(u,v) \ \ \ \ \  {\rm sym}((G^s_w)'') \EEE = G(\hat{u}_{s}^\eps, \hat{v}_{s}^\eps),
 \end{align}
respectively. (Recall notation \eqref{eq: Gnot}.) Above we have checked that the assumptions  \eqref{eq: strong convergence assumptions} hold for $(z^h_s)_{s,h}$. We can therefore apply Lemma \ref{lemma: strong convergence}(b) and obtain 
\begin{align}\label{eq: slope-lsc-1}
(i)& \ \ \Vert  (G^h(y^h) - G^h(w^h_s)) -   \big(G_y - G^s_w  \big)\Vert_{L^2(\Omega)} \le s\rho_\eps(h), \notag\\
(ii) & \  \  \Vert  G^h(y^h) - G^h(w^h_s) \Vert_{L^2(\Omega)} \le C_\eps s,
\end{align}
where $\rho_\eps(h)$ depends on $\eps$ and satisfies $\rho_\eps(h) \to 0$ as $h \to 0$. \BBB Moreover, in view of \eqref{eq: convexi} and \eqref{eq: convexi2}, an elementary but tedious computation leads to
$$|\nabla' v_\eps|^2 - |\nabla' \tilde{v}_{s}|^2 -|\nabla' v|^2 + |\nabla'\hat{v}_{s}^\eps|^2 = 2s\langle \nabla' v_\eps - \nabla' v, \nabla' v_\eps - \nabla' \tilde{v}  \rangle.$$
 \PPP Then \eqref{eq:Z3} and  Lemma \ref{lemma: strong convergence}(c) yield ${\rm sym}(G_y - G^s_w) \, e_3 = s\langle \nabla' v_\eps - \nabla' v, \nabla' v_\eps - \nabla' \tilde{v}  \rangle \, e_3$. \EEE  Thus, by \eqref{eq: eps approx}    and H\"older's inequality we get 
\begin{align}\label{eq:Z4}
\Vert{\rm sym}(G_y - G^s_w) \, e_3 \Vert_{L^2(\Omega)} \le    s\Vert \nabla' \tilde{v} - \nabla' v_\eps \Vert_{L^{\PPP 4\EEE}(\Omega)} \Vert \nabla' v_\eps- \nabla' v \Vert_{L^{\PPP 4\EEE}(\Omega)}  \le  \tilde{C} s \eps.
\end{align}
\PPP Likewise, \EEE by \eqref{eq: Gnot},   \eqref{eq: convexi2}, and  \eqref{eq: two by two} one can  check by an elementary expansion that
\begin{align}\label{eq: g1,g2}
{\rm sym}(G_y - G^s_w) = s g_1 + s^2 g_2 \ \ \ \text{for} \ g_1,g_2 \in L^2(S;\R^{3\times 3}_{\rm sym}).
\end{align}
where $g_1,g_2$ depend on  $u,\tilde{u}, u_\eps, v,\tilde{v},v_\eps$.
\EEE

\emph{Step 4: Lower semicontinuity of slopes.}  We will show that there exist \PPP a \EEE continuous function $\eta_\eps: [0,\infty) \to  [0,\infty)$ with $\eta_\eps(0) = 0$ and a constant $\tilde{C}$ depending on $u,v,\tilde{u},\tilde{v}$ such that for all $s \in [0,1]$ there holds 
\begin{align}\label{eq: three properties}
(i) & \ \  {\mathcal{D}}_{h}(y^h,w^h_s)  \le{\mathcal{D}}_0 \big((u,v),(\hat{u}_{s}^\eps,\hat{v}_{s}^\eps) \big) + s \eta_\eps(h) + \tilde{C}s\eps, \notag \\
(ii) & \ \   h^{-4}\int_\Omega  \big( W(\nabla_h y^h) - W(\nabla_h w^h_s) \big) \ge {\phi}_0(u, v) - {\phi}_0(\hat{u}_{s}^\eps, \hat{v}_{s}^\eps)   - s \eta_\eps(h) - \tilde{C}s\eps, \notag\\
(iii) & \ \   h^{-p\alpha} \int_\Omega \big( P(\nabla^2_h y^h) - P(\nabla^2_h w^h_s) \big) \ge - s \eta_\eps(h).  
\end{align}
We defer the proof of \eqref{eq: three properties} to Steps 5--7 below and now prove the lower semicontinuity. \BBB Recall the definition of $\phi_h$ in \eqref{nonlinear energy-rescale}. \EEE By combining the three estimates in \eqref{eq: three properties}  we obtain for all $s \in [0,1]$ 
\begin{align*}
\frac{({\phi}_{h}(y^h) - {\phi}_{h}(w^h_s))^+}{{\mathcal{D}}_{h}(y^h,w^h_s)}&  \ge  \frac{\big( {\phi}_0(u,v) - {\phi}_0(\hat{u}_{s}^\eps,\hat{v}_{s}^\eps) - 2s \eta_\eps(h) - s\tilde{C}\eps\big)^+ }{{\mathcal{D}}_{0}((u,v),(\hat{u}_{s}^\eps,\hat{v}_{s}^\eps)) + s\eta_\eps(h) + s\tilde{C}\eps}.
\end{align*}
\BBB Recall that $y^h \in \mathscr{S}^M_h$ and Theorem \ref{th: Gamma}(i) imply $\phi_0(u,v) \le M$. \EEE By applying  Lemma \ref{th: convexity2} with $(u_0,v_0) = (u,v)$ and $(u_1,v_1) = (\hat{u}_{1}^\eps,\hat{v}_{1}^\eps)$  we get
 \begin{align*}
\frac{({\phi}_{h}(y^h) - {\phi}_{h}(w^h_s))^+}{{\mathcal{D}}_{h}(y^h,w^h_s)}  &\ge \frac{s\big({\phi}_0(u,v) - {\phi}_0(\hat{u}_{1}^\eps,\hat{v}_{1}^\eps) - \Phi^2_M\big({\mathcal{D}}_0((u,v),(\hat{u}_{1}^\eps,\hat{v}_{1}^\eps))\big)    - 2 \eta_\eps(h) - \tilde{C}\eps \big)^+ }{s\Phi^1\big({\mathcal{D}}_0((u,v),(\hat{u}_{1}^\eps,\hat{v}_{1}^\eps))\big) +  \PPP s \EEE \eta_\eps(h) + \PPP s \EEE \tilde{C}\eps}, 
\end{align*}
where $\Phi^1$ and $\Phi^2_M$ are the functions introduced in Lemma \ref{th: convexity2}.    Thus, in view of \eqref{eq: sconv}, Lemma \ref{th: metric space}(iii), and  Definition \ref{main def2}, we find by letting $s \to 0$  
$$|\partial {\phi}_{h}|_{{\mathcal{D}}_{h}}(y^h)  \ge \frac{\big( {\phi}_0(u,v) - {\phi}_0(\hat{u}_{1}^\eps,\hat{v}_{1}^\eps) - \Phi^2_M\big({\mathcal{D}}_0((u,v),(\hat{u}_{1}^\eps,\hat{v}_{1}^\eps))\big)  - 2 \eta_\eps(h) - \tilde{C}\eps \big)^+ }{\Phi^1\big({\mathcal{D}}_0((u,v),(\hat{u}_{1}^\eps,\hat{v}_{1}^\eps))\big) +  \eta_\eps(h) + \tilde{C}\eps}.$$
\PPP Letting \EEE $h \to 0$ we then derive
$$ \liminf_{h \to 0} |\partial {\phi}_{h}|_{{\mathcal{D}}_{h}}(y^h) \ge \frac{\big( {\phi}_0(u,v) - {\phi}_0((\hat{u}_{1}^\eps, \hat{v}_{1}^\eps)) - \Phi^2_M\big({\mathcal{D}}_0((u,v),(\hat{u}_{1}^\eps, \hat{v}_{1}^\eps))\big) - \tilde{C}\eps \big)^+ }{\Phi^1\big({\mathcal{D}}_0((u,v),(\hat{u}_{1}^\eps, \hat{v}_{1}^\eps))\big) + \tilde{C}\eps }. 
$$
We observe that $\hat{u}_{1}^\eps \to \tilde{u}$ in $W^{1,2}(S;\R^2)$ and $\hat{v}_{1}^\eps \to \tilde{v}$ in $W^{2,2}(S;\R^2)$ as $\eps \to 0$, see \eqref{eq: eps approx} and \eqref{eq: convexi2}. Thus, letting $\eps \to 0$, using   Lemma \ref{th: metric space-lin}(iii),(iv), and then taking   the supremum with respect to $(\tilde{u},\tilde{v})$ we get
\begin{align}\label{eq: reg is enough}
\liminf_{h \to 0} |\partial {\phi}_{h}|_{{\mathcal{D}}_{h}}(y^h)   &\ge \sup \Big\{  \frac{\Big({\phi}_0(u,v) - {\phi}_0(\tilde{u},\tilde{v}) - \Phi^2_M\big({\mathcal{D}}_0((u,v),(\tilde{u},\tilde{v})) \big) \Big)^+}{\Phi^1\big({\mathcal{D}}_0((u,v),(\tilde{u},\tilde{v}))\big)}: \notag \\& \ \ \ \ \ \ \ \ \ \ \ \ \ \ \ \  (\tilde{u},\tilde{v}) \in \bar{\mathscr{S}}^{\rm reg}_0\setminus \lbrace (u,v) \rbrace \Big\},
\end{align}
where $ \bar{\mathscr{S}}^{\rm reg}_0 \subset  {\mathscr{S}}_0$ denotes the subset consisting of functions $u,v$ with regularity $W^{2,\infty}$ and $W^{3,\infty}$, respectively.   Since each  $(\tilde{u},\tilde{v}) \in {\mathscr{S}}_0$ can be approximated  \BBB in $W^{1,2}(S;\R^2) \times W^{2,2}(S)$ \EEE  by a sequence of functions in $\bar{\mathscr{S}}^{\rm reg}_0$ (see Lemma \ref{lemma: density}) and   the right hand side is continuous with respect to that convergence (see Lemma \ref{th: metric space-lin}), the previous inequality also holds for  ${\mathscr{S}}_0$ instead of $\bar{\mathscr{S}}^{\rm reg}_0$.   The representation given in  Lemma \ref{lemma: slopes} then implies
$$\liminf_{h \to 0} |\partial {\phi}_{h}|_{{\mathcal{D}}_{h}}(y^h)  \ge |\partial  {\phi}_0|_{ {\mathcal{D}}_0}(u,v).$$
To conclude the proof, it therefore remains to show  \eqref{eq: three properties}.

\emph{Step 5: Proof of \eqref{eq: three properties}(i).} By using Lemma \ref{lemma: metric space-properties}(ii)   and \eqref{eq: slope-lsc-1} we get 
\begin{align*} 
 {\mathcal{D}}_{h}(y^h,w^h_s)^2 &\le \int_\Omega  Q_D^3\big(G^h(y^h) - G^h(w^h_s)) + Ch^{\alpha}\Vert G^h(y^h) - G^h(w^h_s) \Vert^2_{L^2(\Omega)} \notag\\
 & \le \int_\Omega  Q_D^3(G_y - G^s_w) + s^2\big(C_\eps h^{\alpha} + C(\rho_\eps(h))^2\big)\notag \\
 & = \int_\Omega  Q_D^3({\rm sym}(G_y - G^s_w)) + s^2\big(C_\eps h^{\alpha} + C(\rho_\eps(h))^2\big).  
\end{align*}
Here, the last step follows from the fact that $Q_D^3(F) = Q_D^3({\rm sym}(F))$ for $F \in \R^{3 \times 3}$, see Lemma \ref{D-lin}(ii). Then, using  \eqref{eq: defA} and   \eqref{eq: two by two}  we find
\begin{align*}
 {\mathcal{D}}_{h}(y^h,w^h_s)^2  & \le \hspace{-0.1cm}  \int_\Omega  Q_D^2\big(G(u,v)  - G(\hat{u}^\eps_s, \hat{v}^\eps_s)\big) + C \hspace{-0.1cm}\int_\Omega |{\rm sym}(G_y-G^s_w)\, e_3|^2 +s^2\big(C_\eps h^{\alpha} + C(\rho_\eps(h))^2\big).
\end{align*}
By \eqref{eq: G and D} and \eqref{eq:Z4} we conclude
\begin{align*}
 {\mathcal{D}}_{h}(y^h,w^h_s)^2   \le  {\mathcal{D}}_0 \big((u,v), (\hat{u}^\eps_s, \hat{v}^\eps_s) \big)^2 + (\tilde{C}s \eps)^2 + s^2\big(C_\eps h^{\alpha} + C(\rho_\eps(h))^2\big).
\end{align*}
This yields \eqref{eq: three properties}(i).

\emph{Step 6: Proof of \eqref{eq: three properties}(ii).} First, by Lemma \ref{lemma: metric space-properties}(iv)   and \eqref{eq: slope-lsc-1}(ii) we get 
\begin{align}\label{eq: seco1}
\frac{2}{h^4}\int_\Omega  \Big( W(y^h) - W(w^h_s) \Big) & \ge  \int_\Omega  \Big(Q_W^3(G^h( y^h)) - Q_W^3(G^h(w^h_s))\Big) - Ch^{\alpha}\Vert G^h(y^h) - G^h(w^h_s) \Vert_{L^2(\Omega)} \notag \\
& \ge  \int_\Omega  \Big(Q_W^3(G^h( y^h)) - Q_W^3(G^h(w^h_s))\Big) - C_\eps h^{\alpha}s.
\end{align}
Recall the definition of $\C_W$ in \eqref{eq: order4}. An expansion and \eqref{eq: slope-lsc-1}(i) yield
\begin{align}\label{eq: seco2}
\int_\Omega & \Big(Q_W^3(G^h( y^h)) - Q_W^3(G^h (w^h_s))\Big) \notag \\
&=  - \int_\Omega  \Big(Q_W^3\big(G^h(w^h_s)  -G^h( y^h)\big) + 2\C_W^3[G^h(y^h), G^h(w^h_s) - G^h(y^h)]\Big) \notag   \\
&\ge - \int_\Omega  \Big(Q_W^3 (G_w^s  -G_y ) + 2\C_W^3[G^h(y^h),G_w^s  -G_y ]\Big)  - Cs\rho_\eps(h).
\end{align}
Inequalities  \eqref{eq: seco1}-\eqref{eq: seco2},  the weak convergence $G^h(y^h) \rightharpoonup G_y$ in $L^2(\Omega;\R^{3 \times 3})$ (see \eqref{eq: weakcovi}) and \eqref{eq: g1,g2} yield
\begin{align*}
\frac{2}{h^4}\int_\Omega  \Big( W(y^h) - W(w^h_s) \Big) & \ge   - \int_\Omega  \Big(Q_W^3 (G_w^s  -G_y ) + 2\C_W^3[G_y,G_w^s  -G_y ]\Big)  - s\tilde{\rho}_\eps(h) 
\end{align*}
for some $\tilde{\rho}_\eps(h)$, still satisfying $\tilde{\rho}_\eps(h) \to 0$ as $h \to 0$. 
Using  the fact that $Q_W^3(F) = Q_W^3({\rm sym}(F))$ (see Lemma \ref{D-lin}(ii)),  \eqref{eq: defA}, \eqref{eq: two by two}, and  \eqref{eq:Z4} we conclude
\begin{align*}
\frac{1}{h^4}\int_\Omega  \Big( W(y^h) - W(w^h_s) \Big) & \ge \int_\Omega  \frac{1}{2}\Big(   Q_W^3(G_y) - Q_W^3 (G_w^s))\Big)  - s\tilde{\rho}_\eps(h) \\ 
& \ge  \int_\Omega  \frac{1}{2}\Big(   Q_W^2(G(u,v)) - Q_W^2 (G(\hat{u}_{s}^\eps, \hat{v}_{s}^\eps)))\Big)  - s\tilde{\rho}_\eps(h) - \tilde{C}s\eps \\
& =  {\phi}_0(u, v) - {\phi}_0(\hat{u}_{s}^\eps, \hat{v}_{s}^\eps) - s\tilde{\rho}_\eps(h) - \tilde{C}s\eps, 
\end{align*}
 where the last step follows from \eqref{eq: G and phi}.

\emph{Step 7: Proof of \eqref{eq: three properties}(iii).} By convexity of $P$ and the definition  $w^h_s = y^h - y^h_\eps+ \tilde{y}^h_s$, see \eqref{eq: ulitmate w-def}, we find
\begin{align}\label{eq: Pconvi}
h^{-p\alpha} \int_\Omega \big( P(\nabla_h^2 y^h) - P(\nabla_h^2 w^h_s) \big) \ge h^{-p\alpha} \int_\Omega \partial_Z P(\nabla_h^2 w^h_s) : (\nabla_h^2 y^h_\eps- \nabla_h^2 \tilde{y}_s^h).
\end{align}
By  H\"older's inequality and  \eqref{assumptions-P}(iii) we get 
\begin{align*}
\int_\Omega &|\partial_{Z} P(\nabla_h^2 w^h_s) : (\nabla_h^2  y^h_\eps- \nabla_h^2 \tilde{y}_s^h)| \le \Vert  \partial_Z P(\nabla_h^2 w^h_s) \Vert_{L^{p/(p-1)}(\Omega)} \Vert  \nabla_h^2 \tilde{y}^h_s- \nabla_h^2 y_\eps^h \Vert_{L^{p}(\Omega)}\\
& \le C\Big( \int_\Omega P(\nabla_h^2 w^h_s) \Big)^{\frac{p-1}{p}}   \Vert  \nabla_h^2 \tilde{y}^h_s- \nabla_h^2 y_\eps^h \Vert_{L^{p}(\Omega)} .
\end{align*}
Using $\phi_h(w^h_s) \le    M'$ since $w^h_s \in \mathscr{S}^{M'}_h$   (see Lemma \ref{lemma: strong convergence}(a))   and   \eqref{eq: w1}(i) we then derive 
\begin{align*}
\int_\Omega &|\partial_{Z} P(\nabla_h^2 w^h_s) : (\nabla_h^2  y^h_\eps- \nabla_h^2 \tilde{y}_s^h)| \le Csh\Big( \int_\Omega P(\nabla_h^2 w^h) \Big)^{\frac{p-1}{p}} \hspace{-0.2cm} \le C_\eps sh  (h^{\alpha p})^{\frac{p-1}{p}}    \le  C_\eps sh  h^{\alpha (p-1)},
\end{align*}
where $C_\eps$ depends also on $M'$. By \eqref{eq: Pconvi} and \EEE $1+ \alpha(p-1) - \alpha p >0$, we finally get that \eqref{eq: three properties}(iii) holds. 
\end{proof}

\begin{rem}\label{rem: Poisson2}
{\normalfont
The previous proof is the only point where we need the assumption that the minimum in \eqref{eq:Q2}-\eqref{eq:Q22} is attained for $a=0$ which corresponds to a model    with  zero Poisson's ratio in  $e_3$ direction. \second Although this is a restrictive assumption, it is a good approximation for cellular materials such as cork.  We also note that similar assumptions already appeared in the literature, see \cite{BK}. \rm \EEE In fact, the sequence $(w^h_s)_h$ has to be constructed in \BBB such a \EEE way   that it is a recovery sequence (up to an $\eps$-error) for \emph{both} \eqref{eq: three properties}(i) and \eqref{eq: three properties}(ii). Without this assumption, a sound 2D model would necessarily have to depend on $Q_W^3$ and $Q_D^3$ (instead of $Q_W^2$ and $Q_D^2$) and extra variables in addition to $u$ and $v$ would be required to capture the extension or contraction of the vertical fibers along the evolution. \first  (Still, we are not sure whether our analysis can be adapted to this case or not.) \rm \EEE

}
\end{rem}

\section{Proof of the main results}\label{sec results}

In this section we give the \PPP proofs \EEE of Proposition \ref{maintheorem1}-Theorem \ref{maintheorem3}.

\subsection{Existence of time-discrete solutions in   3D  and passage from 3D to 2D}
 
 In this short subsection we prove Proposition  \ref{maintheorem1} and Theorem \ref{maintheorem3}.  
 
\begin{proof}[Proof of Proposition  \ref{maintheorem1}]
Let $y^h_0 \in \mathscr{S}_h^M = \lbrace y\in \mathscr{S}_h: \phi_h(y) \le M\rbrace$ for some $M>0$. We recall that for the choices $\beta_1 = 4-\alpha p$ and $\beta_2 = 3$ we have $I_h^{\beta_1,\beta_2} = h^4\phi_h$, see \eqref{nonlinear energy-rescale}. Moreover, there holds $ \mathcal{D}_h  = h^{-2}\mathcal{D}$ by \eqref{eq: D,D0-1}. 

It is clear that the minimization problem \eqref{incremental} on $\mathscr{S}_h$ can be restricted to the set $\mathscr{S}_h^M$. Then the existence of solutions to the incremental problem \eqref{incremental} follows from the direct method of the calculus of variations: Lemma \ref{th: metric space}(ii),(iii) yield compactness with respect to the topology induced by $\mathcal{D}_h$  and Lemma \ref{th: metric space}(iv) implies lower semicontinuity. 
\end{proof}

We now proceed with the proof of Theorem \ref{maintheorem3}. We formulate our problem in the setting of Section \ref{sec: auxi-proofs}. We consider the complete metric spaces $(\mathscr{S}_h^M,\mathcal{D}_h)$ and the limiting space $({\mathscr{S}}_0,{\mathcal{D}}_0)$ together with the functionals $\phi_h$ and ${\phi}_0$.  Let $\sigma$ be the topology on ${\mathscr{S}}_0$ introduced in \eqref{eq: convergence u,v-2}. Recall the definition of the mappings $\pi_h: \mathscr{S}^M_h \to {\mathscr{S}}_0$ defined by $\pi_h(y^h) = (u^h,v^h)$ for each $y^h \in \mathscr{S}_h$ and the convergence $y^h \stackrel{\pi\sigma}{\to} (u,v)$, see  below \eqref{eq: convergence u,v-2} and see also \eqref{eq: sigma'}.

\begin{proof}[Proof of Theorem  \ref{maintheorem3}]
\second We consider \EEE an initial datum $(u_0,v_0) \in {\mathscr{S}}_0$. We first see that the family of sequences of initial data $\mathcal{B}(u_0,v_0)$ defined in \eqref{eq:datasequence} is nonempty. This follows from Theorem \ref{th: Gamma}(ii). We check that all assumptions of  Theorem \ref{th:abstract convergence 2} are satisfied. First,  \eqref{compatibility} holds by Theorem \ref{th: lscD} and  \eqref{basic assumptions2} follows from Lemma \ref{lemma: compacti}.   Also \eqref{eq: implication} is satisfied by the $\Gamma$-liminf inequality (Theorem \ref{th: Gamma}(i)) and Theorem \ref{theorem: lsc-slope}. Finally, the local slope $|\partial {\phi}_{0}|_{{\mathcal{D}}_0}$ is a  strong upper gradient for $\phi_0$  by Lemma \ref{lemma: slopes}.

\second Now we consider \EEE a sequence $(y^h_0)_h \in \mathcal{B}(u_0,v_0)$ and a null sequence $(\tau_h)_h$. The definition of $\mathcal{B}(u_0,v_0)$ (see \eqref{eq:datasequence}) implies  \eqref{eq: abstract assumptions1}(ii) with $\bar{z}_0 = (u_0,v_0)$. In particular, as $\pi_h(y_0^h) \stackrel{\pi\sigma}{\to} (u_0,v_0)$, we get that the sequence $(\pi_h(y_0^h))_h$ is bounded in $W^{1,2}(S;\R^2) \times W^{2,2}(S)$. In view of Lemma \ref{th: metric space-lin}(iii), this yields \eqref{eq: abstract assumptions1}(i).

 Let  $\tilde{Y}_{h,\tau_h}$ be a sequence of time-discrete solutions as in \eqref{ds}   with  $\tilde{Y}_{h,\tau_h}(0)=y^h_0$.  Then the  scalings $I_h^{4-\alpha p,3} = h^4\phi_h$  and $ \mathcal{D}  = h^{2}\mathcal{D}_h$  (see \eqref{nonlinear energy-rescale} and \eqref{eq: D,D0-1}) imply that $\tilde{Y}_{h,\tau_h}$ is also a time-discrete solution as in  \eqref{eq: ds-new1}-\eqref{eq: ds-new2}.  The statement of Theorem  \ref{maintheorem3} now follows from the abstract convergence result formulated in  Theorem \ref{th:abstract convergence 2}.
 \end{proof}

\subsection{Fine representation of the slope and  solutions to the equations in 2D}

This subsection is devoted to the proof of Theorem \ref{maintheorem2}. We first note that Theorem \ref{maintheorem2}(i) follows directly from Theorem \ref{maintheorem3}. Therefore, we only  need to show Theorem \ref{maintheorem2}(ii). To this end, we derive a fine representation for the local slope in the 2D setting which will allow us to relate curves of maximal slope to solutions to the equations \eqref{eq: equation-simp}.

For the following proofs we introduce the abbreviation
\begin{align}\label{eq: H shorthand}
H(u,v|\tilde{v}) =   {\rm sym}(\nabla' u)  + {\rm sym}(\nabla' v \otimes \nabla' \tilde{v}) -    x_3 (\nabla')^2 v  \in L^2(\Omega; \R^{2\times 2}_{\rm sym})
\end{align} 
for $(u,v) \in {\mathscr{S}}_0$ and $\tilde{v} \in W^{2,2}(S)$. This definition captures the linear part of the difference of two strains. More precisely, with $G(u,v)$ and $G(\bar{u},\bar{v})$ as defined  in \eqref{eq: Gnot}, we have by an elementary computation
\begin{align}\label{eq: H shorthand2}
G(u,v) - G(\bar{u},\bar{v}) = H(u-\bar{u},v-\bar{v}|v) - \frac{1}{2} (\nabla' v - \nabla' \bar{v}) \otimes (\nabla' v - \nabla' \bar{v}).
\end{align} 
Recall that $\C^2_D$ defined in \eqref{eq: order4} is a fourth order symmetric tensor inducing   the quadratic form $G \mapsto Q_D^2(G)$ which is positive definite on $\R^{2 \times 2}_{\rm sym}$ (cf.\ Lemma \ref{D-lin}(ii)). Moreover, it maps $\R^{2 \times 2}$  to $\R^{2 \times 2}_{\rm sym}$, denoted by $G \mapsto \C_D^2 G$ in the following. More precisely, the mapping $G \mapsto \C_D^2 G$  from $\R^{2 \times 2}_{\rm sym}$ to $\R^{2 \times 2}_{\rm sym}$ is bijective. By $\sqrt{\C^2_D}$ we denote its (unique) root and by $\sqrt{\C^2_D}^{-1}$ the inverse of $\sqrt{\C^2_D}$, both mappings defined on $\R^{2 \times 2}_{\rm sym}$. The same properties also hold with $\C_W^2$ in place of $\C_D^2$.

 We \BBB now prove \EEE the following fine representation for the local slope.

\begin{lemma}[Slope in the 2D setting]\label{lemma: lin-slope}
There exists a differential operator  $\mathcal{L}: {\mathscr{S}}_0 \to L^2(\Omega;\R^{2 \times 2}_{\rm sym})$ satisfying 
\begin{align}\label{lin-slope2-diffi}
\int_\Omega \mathcal{L}(u,v): H(\varphi_u,\varphi_v|v) = 0  \ \ \ \text{for all} \  (u,v) \in {\mathscr{S}}_0\ \  \text{and} \ \  (\varphi_u,\varphi_v) \in W^{1,2}_0(S;\R^2) \times W^{2,2}_0(S) 
\end{align}
such  that the local slope at $(u,v) \in {\mathscr{S}}_0$ can be represented by  
$$|\partial {\phi}_0|_{{\mathcal{D}}_0}(u,v)  =  \Big\| \sqrt{\C^2_D}^{-1}\big(\C^2_W G(u,v) + \mathcal{L}(u,v) \big) \Big\|_{L^2(\Omega)}.$$
\end{lemma}

\Proof  To simplify the notation, we will write   $(\bar{u},\bar{v}) \to (u,v)$ instead of \BBB $\mathcal{D}_0((\bar{u},\bar{v}),(u,v)) \to 0$. \EEE     Recall the definition of the energy ${\phi}_0$ and the dissipation ${\mathcal{D}}_0$ in \eqref{eq: phi0} and \eqref{eq: D,D0-2},   as well as their representations in \eqref{eq: G and phi}-\eqref{eq: G and D}.    By Definition \ref{main def2}(ii) we have
\begin{align*}
|\partial {\phi}_0|_{{\mathcal{D}}_0}(u,v) &= \limsup_{(\bar{u},\bar{v}) \to (u,v)} \frac{({\phi}_0(u,v) - {\phi}_0(\bar{u},\bar{v}))^+}{{\mathcal{D}}_0\big((u,v),(\bar{u},\bar{v})\big)}\notag\\
& = \limsup_{(\bar{u},\bar{v}) \to (u,v)} \frac{\big(\int_\Omega \frac{1}{2}\big(Q^2_W\big(G(u,v) \big) - Q^2_W\big(G(\bar{u},\bar{v}) \big)\big) \big)^+}{\big(\int_\Omega Q^2_D \big( G(u,v)  - G(\bar{u},\bar{v}) \big)\big)^{1/2}}\notag\\
&= \limsup_{(\bar{u},\bar{v}) \to (u,v)} \frac{ \big(\int_\Omega  \C^2_W[G(u,v), G(u,v)  - G(\bar{u},\bar{v})] -  \frac{1}{2}Q^2_W\big(G(u,v)  - G(\bar{u},\bar{v})\big)  \big)^+} {\big(\int_\Omega  Q^2_D \big(G(u,v)  - G(\bar{u},\bar{v}) \big)\big)^{1/2}}.
\end{align*}
 This leads to
$$ |\partial {\phi}_0|_{{\mathcal{D}}_0}(u,v) = \limsup_{(\bar{u},\bar{v}) \to (u,v)} \frac{ \big(  \int_\Omega  \C^2_W[G(u,v), G(u,v)  - G(\bar{u},\bar{v})] \big)^+ } {\big(\int_\Omega Q^2_D \big(G(u,v)  - G(\bar{u},\bar{v}) \big)\big)^{1/2}}. $$
Indeed, to see this, we use that $(\bar{u},\bar{v}) \to (u,v)$  implies  $G(\bar{u},\bar{v}) \to G(u,v)$  strongly in $L^2(S;\R^{2\times 2})$ by Lemma \ref{th: metric space-lin}(iii) \BBB and \eqref{eq: Gnot}. Thus, \EEE we get
$$\int_\Omega Q^2_W \big(G(u,v)  - G(\bar{u},\bar{v}) \big) \Big(\int_\Omega   Q^2_D \big(G(u,v)  - G(\bar{u},\bar{v})\big)\Big)^{-1/2} \to 0.$$
Lemma \ref{th: metric space-lin}(iii) and  \BBB a \EEE Sobolev embedding also give
$$\Vert \nabla' (v_0-v_1) \otimes \nabla' (v_0-v_1)\Vert_{L^2(S)}\le C \Vert v_0 -v_1 \Vert^2_{W^{1,4}(S)} \le C \Vert v_0 -v_1 \Vert^2_{W^{2,2}(S)} \le C \mathcal{D}_0\big((u,v),(\bar{u},\bar{v}) \big)^2.  $$ 
This together with \eqref{eq: G and D}, \eqref{eq: H shorthand}-\eqref{eq: H shorthand2},  and the Cauchy-Schwartz inequality shows
\begin{align*}
|\partial {\phi}_0|_{{\mathcal{D}}_0}(u,v) = \limsup_{(\bar{u},\bar{v}) \to (u,v)} \frac{ \big( \int_\Omega  \C^2_W[G(u,v), H(u-\bar{u},v-\bar{v}|v)  ] \big)^+  } {\big(\int_\Omega Q^2_D(H(u-\bar{u},v-\bar{v}|v))\big)^{1/2}}.
\end{align*}
We introduce the space of test functions $\mathscr{T} = W^{1,2}_0(S;\R^2) \times W^{2,2}_0(S)$. Due to the linearity of   $H(\cdot, \cdot \, | v)$   we find
 \begin{align}\label{lin-slope1}
|\partial {\phi}_0|_{{\mathcal{D}}_0}(u,v) & = \sup_{({u}',{v}') \in \mathscr{T}} \frac{  \int_\Omega  \C^2_W[G(u,v), H(u',v'|v)  ]  } {\big(\int_\Omega  Q^2_D(H(u',v'|v))\big)^{1/2}}  =   \sup_{({u}',{v}') \in \mathscr{T}} \frac{  \int_\Omega  \C^2_W[G(u,v), H(u',v'|v)  ]  } {\Vert \sqrt{\C^2_D}H(u',v'|v) \Vert_{L^2(\BBB \Omega \EEE )}},
\end{align}
where in the second step we used the properties of $\C^2_D$.  We now consider the minimization problem 
$$ \min_{ (u',v') \in \mathscr{T}}  \BBB \mathcal{F} \EEE (u',v'),    $$
where 
 \begin{align*}
\mathcal{F}(u',v') :=    \frac{1}{2}   \int_\Omega \Big|\sqrt{\C^2_D} H(u',v'|v) \Big|^2  - \int_\Omega  \C_W[G(u,v), H(u',v'|v)]. 
 \end{align*}
 We note that the existence of a solution can be guaranteed by the direct method of the calculus of variations: to  show coercivity,   suppose $\mathcal{F}(u',v') \le C$. We  note that $\Vert H(u',v'|v) \Vert^2_{L^2(\Omega)} \le C$ by  Lemma \ref{D-lin}(ii)   with $C$ depending on $G(u,v)$.   A standard argument involving Poincar\'e's inequality and the boundary values yields $\Vert u' \Vert_{W^{1,2}(S)}  + \Vert v' \Vert_{W^{2,2}(S)}\le C $, where $C$  additionally depends on $\hat{u}$, $\hat{v}$, \BBB  and $\Vert v \Vert_{W^{2,2}(S)}$. \EEE   We refer to Remark \ref{rem: small computation}   for details. Moreover, the functional is lower semicontinuous as it is   convex in $H(u',v'|v)$ and $H(u',v'|v)$ is linear in $(u',v')$. 
 
 We denote a solution by $(u_*,v_*) \in \mathscr{T}$ and  we observe that $(u_*,v_*)$   satisfies 
\begin{align*}
&\int_\Omega  \sqrt{\C^2_D} H(u_*,v_*|v)   : \sqrt{\C^2_D}  H(\varphi_u,\varphi_v|v)   -  \int_\Omega  \C^2_W[G(u,v), H(\varphi_u,\varphi_v|v)  ] = 0 
\end{align*}
 for all $(\varphi_u,\varphi_v) \in \mathscr{T}$. This equation can also be formulated as
\begin{align}\label{lin-slope2}
 \int_\Omega \mathcal{L}(u,v): H(\varphi_u,\varphi_v|v) =  0
\end{align}
for all $(\varphi_u,\varphi_v) \in \mathscr{T}$, where we define the operator
\begin{align*}
\mathcal{L}(u,v): = \C^2_D H(u_*,v_*|v)   - \C^2_W   G(u,v).  
\end{align*}
\BBB By \EEE   the definition \eqref{eq: H shorthand} and the regularity of the functions, we find $\mathcal{L}(u,v) \in L^2(\Omega;\R^{2 \times 2}_{\rm sym})$.  By \eqref{lin-slope1}, \eqref{lin-slope2}, and the  definition of $\mathcal{L}$  we then get
\begin{align*}
|\partial {\phi}_0|_{{\mathcal{D}}_0}(u,v)& \ge   \frac{  \int_\Omega \big( \C^2_W G(u,v) + \mathcal{L}(u,v)\big) :  H(u_*,v_*|v)   } {\Vert \sqrt{\C^2_D}H(u_*,v_*|v) \Vert_{L^2(\Omega)}},  \\
&=   \frac{  \int_\Omega \sqrt{\C^2_D}^{-1}\big( \C^2_W G(u,v) + \mathcal{L}(u,v)\big) :  \sqrt{\C^2_D}H(u_*,v_*|v)   } {\Vert \sqrt{\C^2_D}H(u_*,v_*|v) \Vert_{L^2(\Omega)}} \\ &= \Big\|  \sqrt{\C^2_D} H(u_*,v_*|v)  \Big\| _{L^2(\Omega)}  =  \Big\|  \sqrt{\C^2_D}^{-1}\big(\C^2_W G(u,v) + \mathcal{L}(u,v)  \big)\Big\| _{L^2(\Omega)}. 
\end{align*} 
On the other hand, by a similar argument, in view of   \eqref{lin-slope1} and \eqref{lin-slope2}, we  find
\begin{align*}
|\partial {\phi}_0|_{{\mathcal{D}}_0}(u,v) &= \sup_{({u}',{v}') \in \mathscr{T}} \frac{  \int_\Omega \big( \C^2_W G(u,v) + \mathcal{L}(u,v)\big) :  H(u',v'|v)   } {\Vert \sqrt{\C^2_D}H(u',v'|v) \Vert_{L^2(\Omega)}}  \\ 
& \le  \Big\| \sqrt{\C^2_D}^{-1}\big(\C^2_W G(u,v) + \mathcal{L}(u,v) \big) \Big\| _{L^2(\Omega)},
 \end{align*}
 where in the inequality we again distributed $\sqrt{\C^2_D}$ suitably to the two terms and used the Cauchy-Schwartz inequality. This concludes the proof. \eop

Following ideas in \cite[Section 1.4]{AGS}, we now finally relate curves of maximal slope to solutions to the equations \eqref{eq: equation-simp}.

\begin{proof}[Proof of Theorem  \ref{maintheorem2}(ii)]
Since $(u(t),v(t))$ is a curve of maximal slope, we get that  ${\phi}_0(u(t),v(t))$ is decreasing in time, see \eqref{maximalslope}.   This together with Lemma \ref{th: metric space-lin}(ii) gives 
$$(u,v) \in L^\infty([0,\infty ); W^{1,2}(S;\R^2) \times W^{2,2}(S) ).$$
 Moreover, since  $|({u},{v})'|_{{\mathcal{D}}_0} \in L^2([0,\infty ))$ by \eqref{maximalslope}   and ${\mathcal{D}}_0$ is  equivalent to the strong topology on  $W^{1,2}(S;\R^2) \times W^{2,2}(S)$   (see Lemma \ref{th: metric space-lin}(iii)),  we observe that $u$ and $v$ are absolutely continuous curves in the Hilbert spaces $W^{1,2}(S;\R^2)$ and  $W^{2,2}(S;\R)$, respectively.  \BBB By \EEE using \cite[Remark 1.1.3]{AGS} we observe that $u$ and $v$ are differentiable for a.e.\ $t$  with $\partial_t  u(t) \in W^{1,2}(S;\R^2)$ and $\partial_t   v(t) \in W^{2,2}(S)$  for a.e.\ $t$. More precisely,  we have
\begin{align}\label{eq: regularity}
(u,v) \in W^{1,2}([0,\infty);W^{1,2}(S;\R^2) \times W^{2,2}(S))
\end{align}
and for all $0 \le s < t$, and  a.e.\ in $S$ there holds
\begin{align}\label{eq: tderiv}
&\nabla' u(t) - \nabla' u(s) =  \int_s^t \partial_t \nabla' u(r) \, dr, \notag\\  &\nabla' v(t) - \nabla' v(s) =  \int_s^t \partial_t \nabla' v(r)\, dr, \ \ \ \   (\nabla')^2 v(t) - (\nabla')^2 v(s) =  \int_s^t \partial_t (\nabla')^2 v(r)\, dr.
 \end{align}
As a preparation for the representation of the metric derivative, we now consider the difference $G(u(s),v(s))  - G(u(t),v(t)$. For a.e.\ $t$ and a.e.\ $x \in \Omega$ we obtain by \eqref{eq: H shorthand2} and the linearity of  $H(\cdot,\cdot \, | \, v(t))$  
  \begin{align}\label{eq: derivative in Hilbert1}
 \lim_{s \to t} &\frac{G(u(t),v(t))  - G(u(s),v(s))}{t-s} \notag \\&   = \lim_{s \to t} H\Big(\frac{u(t)-u(s)}{t-s},\frac{v(t)-v(s)}{t-s}  \, \Big|  \, v(t) \Big)\notag - \lim_{s \to t} \, (\nabla' v(t) - \nabla' v(s)) \otimes \frac{\nabla' v(t) - \nabla' v(s)}{2(t-s)} \notag\\
 &  =  H\big(\partial_t u(t), \partial_t v(t) \,  | \,  v(t)\big). 
 \end{align}
Similarly, by \eqref{eq: H shorthand2}, taking the integral over $\Omega$, using \eqref{eq: tderiv},   Poincar\'e's inequality,   and  H\"older's inequality  we get  for all $0 \le s < t$
   \begin{align}\label{eq: derivative in Hilbert2}
  \Big\|\big(G(u(t),& v(t))  -  G(u(s),v(s))\big) -  \int_s^t H(\partial_t u(r), \partial_t v(r) \,  | \, v(t)) \, dr \Big\|^2_{L^2(\Omega)}\notag \\ &\    \le   \int_\Omega  |\nabla' v(t) - \nabla' v(s)|^4   \le C\Big(\int_\Omega  |(\nabla')^2 (v(t) - v(s))|^2\Big)^2 \notag \\ & =  C\Big( \int_\Omega  \Big|\int_s^t (\nabla')^2 \partial_t v(r, \BBB x \EEE ) \, dr\Big|^2 \,dx\Big)^2 \notag \le  C\Big( |t-s|  \int_\Omega \int_s^t |(\nabla')^2 \partial_t v(r, \BBB x \EEE )|^2 \, dr \,dx \Big)^2  \notag
\\ & = C|t-s|^2  \Big(  \int_s^t  \Vert (\nabla')^2 \partial_t v(r)\Vert^2_{L^2(\Omega)}\, dr \Big)^2.   
 \end{align}
We now estimate the metric derivative $|(u,v)'|_{{\mathcal{D}}_0}$. By  \eqref{eq: G and D},   \eqref{eq: derivative in Hilbert1}, and Fatou's lemma  we get for a.e.\ $t$
\begin{align}\label{PDE1}
|(u,v)'|_{{\mathcal{D}}_0}(t) & = \lim_{s \to t} \Big(\frac{{\mathcal{D}}_0\big((u(t),v(t)),(u(s),v(s))\big)^2}{|t-s|^2}\Big)^{1/2} 
\notag \\
 &\ge   \Big(  \int_\Omega \liminf_{s\to t}  Q^2_D \Big( \frac{ G(u(t),v(t)) - G(u(s),v(s)) }  {|t-s|} \Big)\Big)^{1/2}  \notag \\
& = \Big\| \sqrt{\C_D^2} H\big(\partial_t u(t), \partial_t v(t) \, | \,  v(t)\big) \Big\|_{L^2(\Omega)}. 
\end{align}
We now analyze the derivative  $\frac{{\rm d}}{{\rm d}t} {\phi}_0(u(t),v(t))$ of the absolutely continuous curve ${\phi}_0 \circ (u,v)$. \BBB Note that for a.e.\ $t$  we have $\lim_{s \to t}\int_s^t \Vert (\nabla')^2 \partial_t v(r)\Vert^2_{L^2(\Omega)}\, dr = 0$   by \eqref{eq: regularity} and, in a similar fashion,  $\lim_{s \to t}(s-t)^{-1}\Vert \int_s^t  H(\partial_t u(r), \partial_t v(r) \,  | \, v(t))\Vert^2_{L^2(\Omega)}\, dr = 0$  by \eqref{eq: regularity}  and H\"older's inequality. \EEE Thus, using   \eqref{eq: G and phi} and   \eqref{eq: derivative in Hilbert2} we get for a.e.\ $t$ that   
\begin{align*}
\frac{{\rm d}}{{\rm d}t} {\phi}_0(u(t),v(t)) & = \lim_{s \to t} \frac{{\phi}_0(u(t),v(t)) - {\phi}_0(u(s),v(s))}{t-s} \\& \BBB \ge \EEE \liminf_{s \to t}   \frac{1}{(t-s)}\int_\Omega   \C^2_W[G(u(t),v(t)),G(u(t),v(t))-G(u(s),v(s))] \\
& \ \ \ - \BBB \limsup_{s \to t} \EEE   \frac{1}{2(t-s)} \int_\Omega Q^2_W\big(G(u(t),v(t)) - G(u(s),v(s)) \big)  \\ 
&\ge \liminf_{s \to t}   \frac{1}{(t-s)}\int_\Omega   \C^2_W[G(u(t),v(t)),G(u(t),v(t))-G(u(s),v(s))] \\
& = \int_\Omega   \C^2_W[G(u(t),v(t)),H(\partial_t u(t), \partial_t v(t) \, | \, v(t))].
\end{align*}
By the property of $\mathcal{L}$ stated in \eqref{lin-slope2-diffi} we get 
\begin{align*}
\frac{{\rm d}}{{\rm d}t} {\phi}_0(u(t),v(t)) & \BBB \ge \int_\Omega  \big( \C^2_W  G(u(t),v(t)) + \mathcal{L}(u(t),v(t)) \big) :   H\big(\partial_t u(t), \partial_t v(t) \, | \, v(t)\big)  \EEE \\
&= \int_\Omega  \Big( \sqrt{\C^2_D}^{-1} \big(\C^2_W G(u(t),v(t)) + \mathcal{L}(u(t),v(t)) \big)\\ & \ \ \ \ \ \ \ \ \ \ \ \ \ \ \ \ \ \ \ \ \ \ \ \   : \sqrt{\C^2_D} H\big(\partial_t u(t), \partial_t v(t) \, | \, v(t)\big)\Big).
\end{align*}
We find by Lemma \ref{lemma: lin-slope}, \eqref{PDE1}, and Young's inequality 
$$  \frac{{\rm d}}{{\rm d}t} {\phi}_0(u(t),v(t)) \ge - \frac{1}{2}  \Big(|\partial {\phi}_0|^2_{{\mathcal{D}}_0}(u(t),v(t)) +  |(u,v)'|^2_{{\mathcal{D}}_0}(t)\Big) \ge \frac{{\rm d}}{{\rm d}t} {\phi}_0(u(t),v(t)),$$
where  the last step is a consequence of the fact that $(u(t),v(t))$ is a curve of maximal slope with respect to ${\phi}_0$. Consequently, all inequalities employed in the proof are in fact equalities and we get
$$  \sqrt{\C^2_D}^{-1} \big(\C^2_W G(u(t),v(t)) + \mathcal{L}(u(t),v(t))\big) =  -\sqrt{\C^2_D} H(\partial_t u(t), \partial_t v(t)\, | \, v(t))
 $$
 pointwise a.e.\ in $\Omega$  for a.e.\ $t$. Multiplying the equation with  $\sqrt{\C^2_D}$ from the left and testing with $H(\varphi_u,\varphi_v \, |\, v(t))$ from the right  for $(\varphi_u, \varphi_v) \in  W^{1,2}_0(S;\R^2) \times W^{2,2}_0(S)$, \NNN we obtain 
\begin{align}\label{eq: given equation}
  \int_\Omega  \big(\C^2_W  G(u(t),v(t))    + {\C^2_D} H\big(\partial_t u(t), \partial_t v(t) \,  | \,  v(t)\big) \big) : H(\varphi_u,\varphi_v\, | \, v(t)) = 0,
\end{align}
 \EEE  where we again used   property \eqref{lin-slope2-diffi}. In the following we will again use the  representation $G(u(t),v(t)) = \BBB {\rm sym}(G_0)(t) \EEE + x_3 G_1(t)$ with the abbreviations 
$${\rm sym}(G_0)(t) = e( u(t)) + \frac{1}{2} \nabla' v(t) \otimes \nabla' v(t), \ \ \  \ \ \ \ G_1(t) = - (\nabla')^2 v(t)$$ 
 introduced in  \eqref{eq: Gnot}. Consider also the corresponding time derivatives
 $$\partial_t G_0(t) = e(\partial_t u  (t)  ) +    \nabla' \partial_t v(t)   \odot   \nabla' v(t),\ \ \  \ \ \ \ \partial_t G_1(t) = - (\nabla')^2 \partial_t v(t), $$
where  for convenience we use the notation $\odot$ for the symmetrized vector product. \EEE Recall also \eqref{eq: H shorthand}  \BBB  and observe that $H(\partial_t u(t), \partial_t v(t) \,  | \,  v(t)) = \partial_t G_0 \PPP (t) \EEE + x_3 \partial_t G_1 \PPP (t) \EEE$. \NNN Evaluating \eqref{eq: given equation} for  $\varphi_v = 0$  \EEE leads to the equations
 \begin{align}\label{eq: weak1}
  \int_S  \big(\C^2_W  G_0(t)  +  \C^2_D  \partial_t G_0(t)  \big) : \nabla' \varphi_u    = 0.
 \end{align}
We observe that \eqref{eq: weak1}  gives \eqref{eq: weak equation1}.  \NNN Evaluating \eqref{eq: given equation} for $\varphi_u  = 0$ yields  
\begin{align*}
\int_S \big(\C^2_W  G_0(t)  +  \C^2_D  \partial_t G_0(t)  \big)   : \big(\nabla' v(t)  \odot  \nabla' \varphi_v  \big) - \frac{1}{12} \Big(    \C^2_W G_1(t) + \C^2_D\partial_t G_1(t) \Big) : (\nabla')^2 \varphi_v  = 0.
 \end{align*}
\EEE This gives the second equation \eqref{eq: weak equation2}  and confirms that $(u,v)$ is a weak solution to \eqref{eq: equation-simp}.
 \end{proof}

\noindent \textbf{Acknowledgements} This work was funded by  the DFG project FR 4083/1-1 and  by the Deutsche Forschungsgemeinschaft (DFG, German Research Foundation) under Germany's Excellence Strategy EXC 2044 -390685587, Mathematics M\"unster: Dynamics--Geometry--Structure. M.F.~acknowledges support by the Alexander von Humboldt Stiftung and thanks for the warm hospitality at \'{U}TIA AV\v{C}R, where this project has been initiated. M.K.~acknowledges support by the GA\v{C}R project 17-04301S and GA\v{C}R-FWF project 19-29646L.


 \typeout{References}

\end{document}